\documentclass[12pt]{amsart}
\pdfoutput=1
\usepackage[margin=1in]{geometry} 
\usepackage{amssymb, float, subcaption, tikz, pgfplots, graphicx, multicol, longtable}

\usepackage[mathscr]{euscript}
\pgfplotsset{compat=1.14}
\usetikzlibrary{patterns, positioning, decorations.pathreplacing, arrows}

\usepackage[]{ytableau}
\ytableausetup{centertableaux}
\ytableausetup{smalltableaux}

\usepackage[bookmarks]{hyperref}

\makeatletter
\def\@secnumfont{\bfseries} 

\def\@seccntformat#1{%
  \protect\textup{\protect\@secnumfont
    \ifnum\pdfstrcmp{section}{#1}=0 \normalfont \fi 
    \csname the#1\endcsname
    \protect\@secnumpunct
  }%
}  
\makeatother

\numberwithin{equation}{subsection}
\numberwithin{figure}{subsection}

\theoremstyle{plain}

\newtheorem*{theorem*}{Theorem}
\newtheorem*{corollary*}{Corollary}
\newtheorem*{lemma*}{Lemma}
\newtheorem*{proposition*}{Proposition}
\newtheorem*{summary*}{Summary}

\theoremstyle{remark}
\newtheorem*{remark*}{Remark}

\theoremstyle{definition}

\newtheorem{definition}{Definition}[subsection]

\newtheorem*{case*}{Case}
\newtheorem*{subcase*}{Subcase}
\newtheorem*{definition*}{Definition}
\newtheorem*{example*}{Example}
\newtheorem*{nonexample*}{Non-Example}

\newcommand{\ds}{\displaystyle}

\newcommand{\mR}{\mathcal{R}}
\newcommand{\mS}{\mathcal{S}}
\newcommand{\mT}{\mathcal{T}}
\newcommand{\mF}{\mathcal{F}}

\newcommand{\sT}{\mathscr{T}}

\newcommand{\bN}{\mathbb{N}}
\newcommand{\bC}{\mathbb{C}}
\newcommand{\bS}{\mathbb{S}}

\newcommand{\fS}{\mathfrak{S}}

\newcommand{\gl}{\mathfrak{gl}}
\newcommand{\GL}{\mathop{GL}}
\newcommand{\Sym}{\operatorname{Sym}}

\title{Explicit Pieri Inclusions}
\author{Markus Hunziker}
\address{Department of Mathematics, Baylor University, Waco, TX 76798}
\email{markus\_hunziker@baylor.edu}
\author{John A. Miller}
\address{Department of Mathematics, Baylor University, Waco, TX 76798}
\email{john\_miller5@baylor.edu}
\author{Mark Sepanski}
\address{Department of Mathematics, Baylor University, Waco, TX 76798}
\email{mark\_sepanski@baylor.edu}
\date{\today}

\begin{document}

\maketitle

\tableofcontents

\section{Introduction} \label{section: Introduction}

\subsection{} \label{Introduction subsection: Pieri rule}
Let \(\lambda = (\lambda_1, \lambda_2, \ldots)\) be a partition, i.e., a sequence of non-negative integers in weakly decreasing order. 
The number of nonzero parts \(\lambda_i\) is called the length of \(\lambda\), denoted by \(l(\lambda)\).
If \(V\) is a complex vector space of dimension \(n\leq l(\lambda)\), we can apply the Schur--Weyl functor \(\bS_{\lambda}\) (as in \cite{fulton2013representation}) to \(V\) to obtain an irreducible representation \(\bS_{\lambda}(V)\) of the general linear group \(\GL(V)\). 
It follows from Pieri's formula for the product of an elementary symmetric polynomial and a Schur polynomial that the tensor product representation \(\wedge^m(V) \otimes \bS_{\lambda}(V)\) decomposes multiplicity-free into a direct sum of irreducible representations
\[
    \wedge^m(V) \otimes \bS_{\lambda}(V) \cong \bigoplus_{\mu} \bS_{\mu}(V),
\]
where the sum is over all partitions \(\mu\) with \(l(\mu)\leq n\) whose Young diagram is obtained from the Young diagram of \(\lambda\) by adding exactly \(m\) boxes, at most one to each row.
Since the decomposition is multiplicity-free, it is natural to ask for explicit descriptions of the embeddings 
\(\Phi_m:\bS_{\mu}(V) \hookrightarrow \wedge^m(V) \otimes \bS_{\lambda}(V)\).
Following \cite{sam2011pieri}, we call these embeddings (skew) Pieri inclusions.

\subsection{} \label{Introduction subsection: Pieri inclusions}

Given a basis \(\{e_1, e_2, \ldots, e_n\}\) of \(V\), the representation \(\bS_{\lambda}(V)\) is equipped with a canonical basis indexed by the set of semistandard tableaux of shape \(\lambda\) with fillings from the set \(\{1, \ldots, n\}\). 
In \cite{olver1982differential}, Olver gave an explicit description of the Pieri inclusions with respect to these canonical bases in the special case when \(m = 1\). 
When \(m > 1\), the Pieri inclusion \(\Phi_m\) can be obtained by iteration of the special case \cite[Corollary 1.8]{sam2011pieri}.
The main purpose of this paper is to give a new combinatorial description of \(\Phi_m\) that (a) leads to a more efficient algorithm and (b) can be given in a general closed form (avoiding iteration) for \(m \ge 1\). 
In regard to (a), we will show that our algorithm achieves an exponential speed-up over Olver's algorithm when it is restricted to partitions with a bounded number of distinct parts. 
More precisely, if we fix a positive integer \(N\) and consider partitions \(\lambda\) that can be written in exponential notation as \(\lambda=(1^{h_1},2^{h_2},3^{h_3},\dots )\) with at most \(N\) nonzero exponents \(h_i\), then our algorithm to compute the image of a highest weight vector under a Pieri inclusion \(\Phi_1: \bS_{\mu}(V) \hookrightarrow V \otimes \bS_{\lambda}(V)\) has a run-time complexity of \(O(l(\lambda)^N)\), whereas Olver's algorithm has a run-time complexity of \(\Omega(2^{\,l(\lambda)})\). 

\subsection{} \label{Introduction subsection: Our formula}

Our general formula for a Pieri inclusion \(\Phi_m : \bS_{\mu}(V) \hookrightarrow \wedge^m(V) \otimes \bS_{\lambda}(V)\), where \(\lambda \setminus \mu\) is a skew diagram with no two boxes in the same row, is as follows.
If \(T\) is a semistandard tableau of shape \(\mu\) with filling in \(\{1, \ldots, n\}\) and \(e_T \in \bS_{\mu}(V)\) is the corresponding basis element, then
\[
    \Phi_m(e_T) = \sum_{P} \frac{(-1)^P}{H(P)} P(T)
\]
where the sum is over a certain set of ``\(m\)-paths'' \(P\) which remove \(m\) boxes from the shape \(\lambda\), \((-1)^P\) is a sign, \(H(P)\) is a positive integer that is a product of certain ``hook lengths.''
We will write the path \(P\) acting on \(T\) as 
\[
    P(T) = e_{Y_P} \otimes e_{T_P}
\]
where \(e_{Y_P} = e_{i_1} \wedge \cdots \wedge e_{i_m} \in \wedge^m(V)\) is given by the entries of the boxes removed by \(P\), and \(T_P\) is a (not necessarily semistandard) tableau of shape \(\lambda\) with filings in \(\{1, \ldots, n\}\) such that
\[
    \{\text{numbers appearing in } T\} = \{\text{numbers appearing in } Y_P\} \cup \{\text{numbers appearing in } T_P\}
\]
as a multi-set.
All of this will be defined rigorously in Sections \ref{section: Constructing the Pieri Inclusion for Removing One Box} and \ref{section: Constructing the Pieri Inclusion for Removing Many Boxes}.

\subsection{} \label{Introduction subsection: Example HW}

To illustrate how our formula works, we look at an example in the case when \(n = 4\) and \(m = 1\). Let \(\lambda = (2,1,1,1)\), and \(\mu = (2,1,1)\). Then the Schur--Weyl module \(\bS_{\lambda}(V)\) appears as a summand in the decomposition of \(\bS_{(1)}(V) \otimes \bS_{(2,1,1)}(V) = V \otimes \bS_{(2,1,1)}\),
    \[
        \begin{ytableau}
            *(black!25)
        \end{ytableau}
        \otimes
        \begin{ytableau}
            ~ & ~\\
            ~\\
            ~
        \end{ytableau}
        =
        \begin{ytableau}
            ~ & ~ & *(black!25)\\
            ~\\
            ~
        \end{ytableau}
        \oplus
        \begin{ytableau}
            ~ & ~\\
            ~ & *(black!25)\\
            ~
        \end{ytableau}
        \oplus
        \begin{ytableau}
            ~ & ~\\
            ~ \\
            ~ \\
            *(black!25)
        \end{ytableau}.
    \]
Consider the Pieri inclusion
\[
    \Phi_1: \bS_{(2,1,1,1)}(V) \hookrightarrow V \otimes \bS_{(2,1,1)}(V).
\]
By abuse of notation, we will identify semistandard tableaux and their corresponding basis vectors.
We will compute
    \[
        \Phi_1\left( \begin{ytableau}
            1 & 1\\
            2 \\
            3 \\
            4
        \end{ytableau} \right)
        =
        \sum_{P} \frac{(-1)^P}{H(P)} P\left( \begin{ytableau}
            1 & 1\\
            2 \\
            3 \\
            4
        \end{ytableau} \right).
    \]
    
The sum is over all ``\(1\)-paths'' \(P\) on \(\lambda\), that is, certain maps on the boxes in \(\lambda\) that removing a single box. 
Below we illustrate all such \(1\)-paths with arrows, shading the boxes on which the path acts.
We will view the box removed by a \(1\)-path as being moved to the top of the diagram in a ``zeroth row'' of the diagram. We give the image up to a row permutation, so that it is semi-standard.

\begingroup
\setlength{\tabcolsep}{18pt}
\renewcommand{\arraystretch}{4}
\begin{center}
    \begin{tabular}{c c c}
        \begin{tikzpicture}[baseline=(O.base)]
            \node (O) at (0,.6) {~};
            \filldraw[fill=blue!25] (0,0) rectangle (.3,.3);
            \node at (.15,.15) {\scriptsize \(4\)};
            \draw (0,.3) rectangle (.3,.6);
            \node at (.15,.45) {\scriptsize \(3\)};
            \draw (0,.6) rectangle (.3,.9);
            \node at (.15,.75) {\scriptsize \(2\)};
            \draw (0,.9) rectangle (.3,1.2);
            \node at (.15,1.05) {\scriptsize \(1\)};
            \draw (.3,.9) rectangle (.6,1.2);
            \node at (.45,1.05) {\scriptsize \(1\)};
            \draw (0,1.2) rectangle (.3,1.5);
            \draw[blue, thick, ->] (.1,.15) to [out = 180, in = 180] (.15,1.35);
        \end{tikzpicture}
        \(\leadsto\)
        \(-\)
        \begin{ytableau}
            4
        \end{ytableau}
        \(\otimes\)
        \begin{ytableau}
            1 & 1\\
            2 \\
            3 \\
        \end{ytableau},
        &
        \begin{tikzpicture}[baseline=(O.base)]
            \node (O) at (0,.6) {~};
            \filldraw[fill=blue!25] (0,0) rectangle (.3,.3);
            \node at (.15,.15) {\scriptsize \(4\)};
            \draw[fill=blue!25] (0,.3) rectangle (.3,.6);
            \node at (.15,.45) {\scriptsize \(3\)};
            \draw (0,.6) rectangle (.3,.9);
            \node at (.15,.75) {\scriptsize \(2\)};
            \draw (0,.9) rectangle (.3,1.2);
            \node at (.15,1.05) {\scriptsize \(1\)};
            \draw (.3,.9) rectangle (.6,1.2);
            \node at (.45,1.05) {\scriptsize \(1\)};
            \draw (0,1.2) rectangle (.3,1.5);
            \draw[blue, thick] (.1,.15) to [out = 180, in = 180, looseness=2] (.1,.45);
            \draw[blue, thick, ->] (.1,.45) to [out = 180, in = 180] (.15,1.35);
        \end{tikzpicture}
        \(\leadsto\)
        \begin{ytableau}
            3
        \end{ytableau}
        \(\otimes\)
        \begin{ytableau}
            1 & 1\\
            2 \\
            4 \\
        \end{ytableau},
        &
        \begin{tikzpicture}[baseline=(O.base)]
            \node (O) at (0,.6) {~};
            \filldraw[fill=blue!25] (0,0) rectangle (.3,.3);
            \node at (.15,.15) {\scriptsize \(4\)};
            \draw[fill=blue!25] (0,.3) rectangle (.3,.6);
            \node at (.15,.45) {\scriptsize \(3\)};
            \draw[fill=blue!25] (0,.6) rectangle (.3,.9);
            \node at (.15,.75) {\scriptsize \(2\)};
            \draw (0,.9) rectangle (.3,1.2);
            \node at (.15,1.05) {\scriptsize \(1\)};
            \draw (.3,.9) rectangle (.6,1.2);
            \node at (.45,1.05) {\scriptsize \(1\)};
            \draw (0,1.2) rectangle (.3,1.5);
            \draw[blue, thick] (.1,.15) to [out = 180, in = 180, looseness=2] (.1,.45);
            \draw[blue, thick] (.1,.45) to [out = 180, in = 180, looseness=2] (.1,.75);
            \draw[blue, thick, ->] (.1,.75) to [out = 180, in = 180, looseness=1.5] (.15,1.35);
        \end{tikzpicture}
        \(\leadsto\)
        \(-\)
        \begin{ytableau}
            2
        \end{ytableau}
        \(\otimes\)
        \begin{ytableau}
            1 & 1\\
            3 \\
            4 \\
        \end{ytableau},
        \\
        \begin{tikzpicture}[baseline=(O.base)]
            \node (O) at (0,.6) {~};
            \filldraw[fill=blue!25] (0,0) rectangle (.3,.3);
            \node at (.15,.15) {\scriptsize \(4\)};
            \draw (0,.3) rectangle (.3,.6);
            \node at (.15,.45) {\scriptsize \(3\)};
            \draw (0,.6) rectangle (.3,.9);
            \node at (.15,.75) {\scriptsize \(2\)};
            \draw[fill=blue!25] (0,.9) rectangle (.3,1.2);
            \node at (.15,1.05) {\scriptsize \(1\)};
            \draw (.3,.9) rectangle (.6,1.2);
            \node at (.45,1.05) {\scriptsize \(1\)};
            \draw (0,1.2) rectangle (.3,1.5);
            \draw[blue, thick] (.1,.15) to [out = 180, in = 180] (.1,1.05);
            \draw[blue, thick, ->] (.1,1.05) to [out = 180, in = 180, looseness=2] (.15,1.35);
        \end{tikzpicture}
        \(\leadsto\)
        \(\frac{1}{4}\)
        \begin{ytableau}
            1
        \end{ytableau}
        \(\otimes\)
        \begin{ytableau}
            1 & 4\\
            2 \\
            3 \\
        \end{ytableau},
        &
        \begin{tikzpicture}[baseline=(O.base)]
            \node (O) at (0,.6) {~};
            \filldraw[fill=blue!25] (0,0) rectangle (.3,.3);
            \node at (.15,.15) {\scriptsize \(4\)};
            \draw[fill=blue!25] (0,.3) rectangle (.3,.6);
            \node at (.15,.45) {\scriptsize \(3\)};
            \draw (0,.6) rectangle (.3,.9);
            \node at (.15,.75) {\scriptsize \(2\)};
            \draw[fill=blue!25] (0,.9) rectangle (.3,1.2);
            \node at (.15,1.05) {\scriptsize \(1\)};
            \draw (.3,.9) rectangle (.6,1.2);
            \node at (.45,1.05) {\scriptsize \(1\)};
            \draw (0,1.2) rectangle (.3,1.5);
            \draw[blue, thick] (.1,.15) to [out = 180, in = 180, looseness=2] (.1,.45);
            \draw[blue, thick] (.1,.45) to [out = 180, in = 180, looseness=1.5] (.1,1.05);
            \draw[blue, thick, ->] (.1,1.05) to [out = 180, in = 180, looseness=2] (.15,1.35);
        \end{tikzpicture}
        \(\leadsto\)
        \(- \frac{1}{4}\)
        \begin{ytableau}
            1
        \end{ytableau}
        \(\otimes\)
        \begin{ytableau}
            1 & 3\\
            2 \\
            4 \\
        \end{ytableau},
        &
        \begin{tikzpicture}[baseline=(O.base)]
            \node (O) at (0,.6) {~};
            \filldraw[fill=blue!25] (0,0) rectangle (.3,.3);
            \node at (.15,.15) {\scriptsize \(4\)};
            \draw[fill=blue!25] (0,.3) rectangle (.3,.6);
            \node at (.15,.45) {\scriptsize \(3\)};
            \draw[fill=blue!25] (0,.6) rectangle (.3,.9);
            \node at (.15,.75) {\scriptsize \(2\)};
            \draw[fill=blue!25] (0,.9) rectangle (.3,1.2);
            \node at (.15,1.05) {\scriptsize \(1\)};
            \draw (.3,.9) rectangle (.6,1.2);
            \node at (.45,1.05) {\scriptsize \(1\)};
            \draw (0,1.2) rectangle (.3,1.5);
            \draw[blue, thick] (.1,.15) to [out = 180, in = 180, looseness=2] (.1,.45);
            \draw[blue, thick] (.1,.45) to [out = 180, in = 180, looseness=2] (.1,.75);
            \draw[blue, thick] (.1,.75) to [out = 180, in = 180, looseness=2] (.1,1.05);
            \draw[blue, thick, ->] (.1,1.05) to [out = 180, in = 180, looseness=2] (.15,1.35);
        \end{tikzpicture}
        \(\leadsto\)
        \(\frac{1}{4}\)
        \begin{ytableau}
            1
        \end{ytableau}
        \(\otimes\)
        \begin{ytableau}
            1 & 2\\
            3 \\
            4 \\
        \end{ytableau},
        \\
        \begin{tikzpicture}[baseline=(O.base)]
            \node (O) at (0,.6) {~};
            \filldraw[fill=blue!25] (0,0) rectangle (.3,.3);
            \node at (.15,.15) {\scriptsize \(4\)};
            \draw (0,.3) rectangle (.3,.6);
            \node at (.15,.45) {\scriptsize \(3\)};
            \draw (0,.6) rectangle (.3,.9);
            \node at (.15,.75) {\scriptsize \(2\)};
            \draw (0,.9) rectangle (.3,1.2);
            \node at (.15,1.05) {\scriptsize \(1\)};
            \draw[fill=blue!25] (.3,.9) rectangle (.6,1.2);
            \node at (.45,1.05) {\scriptsize \(1\)};
            \draw (0,1.2) rectangle (.3,1.5);
            \draw[blue, thick] (.1,.15) to [out = 180, in = 180, looseness=1.5] (.4,1.05);
            \draw[blue, thick, ->] (.4,1.05) to [out = 90, in = 0] (.15,1.35);
        \end{tikzpicture}
        \(\leadsto\)
        \(\frac{1}{4}\)
        \begin{ytableau}
            1
        \end{ytableau}
        \(\otimes\)
        \begin{ytableau}
            1 & 4\\
            2 \\
            3 \\
        \end{ytableau},
        &
        \begin{tikzpicture}[baseline=(O.base)]
            \node (O) at (0,.6) {~};
            \filldraw[fill=blue!25] (0,0) rectangle (.3,.3);
            \node at (.15,.15) {\scriptsize \(4\)};
            \draw[fill=blue!25] (0,.3) rectangle (.3,.6);
            \node at (.15,.45) {\scriptsize \(3\)};
            \draw (0,.6) rectangle (.3,.9);
            \node at (.15,.75) {\scriptsize \(2\)};
            \draw (0,.9) rectangle (.3,1.2);
            \node at (.15,1.05) {\scriptsize \(1\)};
            \draw[fill=blue!25] (.3,.9) rectangle (.6,1.2);
            \node at (.45,1.05) {\scriptsize \(1\)};
            \draw (0,1.2) rectangle (.3,1.5);
            \draw[blue, thick] (.1,.15) to [out = 180, in = 180, looseness=2] (.1,.45);
            \draw[blue, thick] (.1,.45) to [out = 180, in = 180, looseness=1.5] (.4,1.05);
            \draw[blue, thick, ->] (.4,1.05) to [out = 90, in = 0] (.15,1.35);
        \end{tikzpicture}
        \(\leadsto\)
        \(- \frac{1}{4}\)
        \begin{ytableau}
            1
        \end{ytableau}
        \(\otimes\)
        \begin{ytableau}
            1 & 3\\
            2 \\
            4 \\
        \end{ytableau},
        &
        \begin{tikzpicture}[baseline=(O.base)]
            \node (O) at (0,.6) {~};
            \filldraw[fill=blue!25] (0,0) rectangle (.3,.3);
            \node at (.15,.15) {\scriptsize \(4\)};
            \draw[fill=blue!25] (0,.3) rectangle (.3,.6);
            \node at (.15,.45) {\scriptsize \(3\)};
            \draw[fill=blue!25] (0,.6) rectangle (.3,.9);
            \node at (.15,.75) {\scriptsize \(2\)};
            \draw (0,.9) rectangle (.3,1.2);
            \node at (.15,1.05) {\scriptsize \(1\)};
            \draw[fill=blue!25] (.3,.9) rectangle (.6,1.2);
            \node at (.45,1.05) {\scriptsize \(1\)};
            \draw (0,1.2) rectangle (.3,1.5);
            \draw[blue, thick] (.1,.15) to [out = 180, in = 180, looseness=2] (.1,.45);
            \draw[blue, thick] (.1,.45) to [out = 180, in = 180, looseness=2] (.1,.75);
            \draw[blue, thick] (.1,.75) to [out = 180, in = 180, looseness=2] (.4,1.05);
            \draw[blue, thick, ->] (.4,1.05) to [out = 90, in = 0] (.15,1.35);
        \end{tikzpicture}
        \(\leadsto\)
        \(\frac{1}{4}\)
        \begin{ytableau}
            1
        \end{ytableau}
        \(\otimes\)
        \begin{ytableau}
            1 & 2\\
            3 \\
            4 \\
        \end{ytableau}
    \end{tabular}
\end{center}
\endgroup
    
Thus, up to row permutations we have
\begin{align*}
    \Phi_1\left( \begin{ytableau}
        1 & 1\\
        2 \\
        3 \\
        4
    \end{ytableau} \right)
    =
    &- 
    \begin{ytableau}
        4
    \end{ytableau}
    \otimes
    \begin{ytableau}
        1 & 1\\
        2 \\
        3 \\
    \end{ytableau}
    +  
    \begin{ytableau}
        3
    \end{ytableau}
    \otimes
    \begin{ytableau}
        1 & 1\\
        2 \\
        4 \\
    \end{ytableau}
    -  
    \begin{ytableau}
        2
    \end{ytableau}
    \otimes
    \begin{ytableau}
        1 & 1\\
        3 \\
        4 \\
    \end{ytableau}\\
    &+ \frac{1}{2} \  
    \begin{ytableau}
        1
    \end{ytableau}
    \otimes
    \begin{ytableau}
        1 & 4\\
        2 \\
        3 \\
    \end{ytableau}
    - \frac{1}{2} \ 
    \begin{ytableau}
        1
    \end{ytableau}
        \otimes
    \begin{ytableau}
        1 & 3\\
        2 \\
        4 \\
    \end{ytableau}
    + \frac{1}{2} \ 
    \begin{ytableau}
        1
    \end{ytableau}
    \otimes
    \begin{ytableau}
        1 & 2\\
        3 \\
        4 \\
    \end{ytableau}.
\end{align*}

\subsection{} \label{Introduction subsection: Example straightening}

In \ref{Introduction subsection: Example HW}, all terms \(T_P\) that appeared were (after row permutations) semi-standard. 
We now compute an example where some of the terms that appear in the image are not semi-standard, and so must be straightened. 
Let \(\Phi_1\) be as in \ref{Introduction subsection: Example HW}, we will compute
\[
    \Phi_1\left( \begin{ytableau}
        1 & 2\\
        2 \\
        3 \\
        4
    \end{ytableau} \right)
    =
    \sum_{P} \frac{(-1)^P}{H(P)} P\left( \begin{ytableau}
        1 & 2\\
        2 \\
        3 \\
        4
    \end{ytableau} \right),
\]
where the terms of the sum in the image are again indexed by the \(1\)-paths on \(\lambda\) removing a single box, which we illustrate below.
As before, we give the image up to row permutations and we now star the terms that need to be starightened.
    
\begingroup
\setlength{\tabcolsep}{18pt}
\renewcommand{\arraystretch}{4}
\begin{center}
    \begin{tabular}{c c c}
        \begin{tikzpicture}[baseline=(O.base)]
            \node (O) at (0,.6) {~};
            \filldraw[fill=blue!25] (0,0) rectangle (.3,.3);
            \node at (.15,.15) {\scriptsize \(4\)};
            \draw (0,.3) rectangle (.3,.6);
            \node at (.15,.45) {\scriptsize \(3\)};
            \draw (0,.6) rectangle (.3,.9);
            \node at (.15,.75) {\scriptsize \(2\)};
            \draw (0,.9) rectangle (.3,1.2);
            \node at (.15,1.05) {\scriptsize \(1\)};
            \draw (.3,.9) rectangle (.6,1.2);
            \node at (.45,1.05) {\scriptsize \(2\)};
            \draw (0,1.2) rectangle (.3,1.5);
            \draw[blue, thick, ->] (.1,.15) to [out = 180, in = 180] (.15,1.35);
        \end{tikzpicture}
        \(\leadsto\)
        \(-\)
        \begin{ytableau}
            4
        \end{ytableau}
        \(\otimes\)
        \begin{ytableau}
            1 & 2\\
            2 \\
            3 \\
        \end{ytableau},
        &
        \begin{tikzpicture}[baseline=(O.base)]
            \node (O) at (0,.6) {~};
            \filldraw[fill=blue!25] (0,0) rectangle (.3,.3);
            \node at (.15,.15) {\scriptsize \(4\)};
            \draw[fill=blue!25] (0,.3) rectangle (.3,.6);
            \node at (.15,.45) {\scriptsize \(3\)};
            \draw (0,.6) rectangle (.3,.9);
            \node at (.15,.75) {\scriptsize \(2\)};
            \draw (0,.9) rectangle (.3,1.2);
            \node at (.15,1.05) {\scriptsize \(1\)};
            \draw (.3,.9) rectangle (.6,1.2);
            \node at (.45,1.05) {\scriptsize \(2\)};
            \draw (0,1.2) rectangle (.3,1.5);
            \draw[blue, thick] (.1,.15) to [out = 180, in = 180, looseness=2] (.1,.45);
            \draw[blue, thick, ->] (.1,.45) to [out = 180, in = 180] (.15,1.35);
        \end{tikzpicture}
        \(\leadsto\)
        \begin{ytableau}
            3
        \end{ytableau}
        \(\otimes\)
        \begin{ytableau}
            1 & 2\\
            2 \\
            4 \\
        \end{ytableau},
        &
        \begin{tikzpicture}[baseline=(O.base)]
            \node (O) at (0,.6) {~};
            \filldraw[fill=blue!25] (0,0) rectangle (.3,.3);
            \node at (.15,.15) {\scriptsize \(4\)};
            \draw[fill=blue!25] (0,.3) rectangle (.3,.6);
            \node at (.15,.45) {\scriptsize \(3\)};
            \draw[fill=blue!25] (0,.6) rectangle (.3,.9);
            \node at (.15,.75) {\scriptsize \(2\)};
            \draw (0,.9) rectangle (.3,1.2);
            \node at (.15,1.05) {\scriptsize \(1\)};
            \draw (.3,.9) rectangle (.6,1.2);
            \node at (.45,1.05) {\scriptsize \(2\)};
            \draw (0,1.2) rectangle (.3,1.5);
            \draw[blue, thick] (.1,.15) to [out = 180, in = 180, looseness=2] (.1,.45);
            \draw[blue, thick] (.1,.45) to [out = 180, in = 180, looseness=2] (.1,.75);
            \draw[blue, thick, ->] (.1,.75) to [out = 180, in = 180, looseness=1.5] (.15,1.35);
        \end{tikzpicture}
        \(\leadsto\)
        \(-\)
        \begin{ytableau}
            2
        \end{ytableau}
        \(\otimes\)
        \begin{ytableau}
            1 & 2\\
            3 \\
            4 \\
        \end{ytableau},
        \\
        \begin{tikzpicture}[baseline=(O.base)]
            \node (O) at (0,.6) {~};
            \filldraw[fill=blue!25] (0,0) rectangle (.3,.3);
            \node at (.15,.15) {\scriptsize \(4\)};
            \draw (0,.3) rectangle (.3,.6);
            \node at (.15,.45) {\scriptsize \(3\)};
            \draw (0,.6) rectangle (.3,.9);
            \node at (.15,.75) {\scriptsize \(2\)};
            \draw[fill=blue!25] (0,.9) rectangle (.3,1.2);
            \node at (.15,1.05) {\scriptsize \(1\)};
            \draw (.3,.9) rectangle (.6,1.2);
            \node at (.45,1.05) {\scriptsize \(2\)};
            \draw (0,1.2) rectangle (.3,1.5);
            \draw[blue, thick] (.1,.15) to [out = 180, in = 180] (.1,1.05);
            \draw[blue, thick, ->] (.1,1.05) to [out = 180, in = 180, looseness=2] (.15,1.35);
        \end{tikzpicture}
        \(\overset{*}{\leadsto}\)
        \(\frac{1}{4}\)
        \begin{ytableau}
            1
        \end{ytableau}
        \(\otimes\)
        \begin{ytableau}
            2 & 4\\
            2 \\
            3 \\
        \end{ytableau},
        &
        \begin{tikzpicture}[baseline=(O.base)]
            \node (O) at (0,.6) {~};
            \filldraw[fill=blue!25] (0,0) rectangle (.3,.3);
            \node at (.15,.15) {\scriptsize \(4\)};
            \draw[fill=blue!25] (0,.3) rectangle (.3,.6);
            \node at (.15,.45) {\scriptsize \(3\)};
            \draw (0,.6) rectangle (.3,.9);
            \node at (.15,.75) {\scriptsize \(2\)};
            \draw[fill=blue!25] (0,.9) rectangle (.3,1.2);
            \node at (.15,1.05) {\scriptsize \(1\)};
            \draw (.3,.9) rectangle (.6,1.2);
            \node at (.45,1.05) {\scriptsize \(2\)};
            \draw (0,1.2) rectangle (.3,1.5);
            \draw[blue, thick] (.1,.15) to [out = 180, in = 180, looseness=2] (.1,.45);
            \draw[blue, thick] (.1,.45) to [out = 180, in = 180, looseness=1.5] (.1,1.05);
            \draw[blue, thick, ->] (.1,1.05) to [out = 180, in = 180, looseness=2] (.15,1.35);
        \end{tikzpicture}
        \(\overset{*}{\leadsto}\)
        \(- \frac{1}{4}\)
        \begin{ytableau}
            1
        \end{ytableau}
        \(\otimes\)
        \begin{ytableau}
            2 & 3\\
            2 \\
            4 \\
        \end{ytableau},
        &
        \begin{tikzpicture}[baseline=(O.base)]
            \node (O) at (0,.6) {~};
            \filldraw[fill=blue!25] (0,0) rectangle (.3,.3);
            \node at (.15,.15) {\scriptsize \(4\)};
            \draw[fill=blue!25] (0,.3) rectangle (.3,.6);
            \node at (.15,.45) {\scriptsize \(3\)};
            \draw[fill=blue!25] (0,.6) rectangle (.3,.9);
            \node at (.15,.75) {\scriptsize \(2\)};
            \draw[fill=blue!25] (0,.9) rectangle (.3,1.2);
            \node at (.15,1.05) {\scriptsize \(1\)};
            \draw (.3,.9) rectangle (.6,1.2);
            \node at (.45,1.05) {\scriptsize \(2\)};
            \draw (0,1.2) rectangle (.3,1.5);
            \draw[blue, thick] (.1,.15) to [out = 180, in = 180, looseness=2] (.1,.45);
            \draw[blue, thick] (.1,.45) to [out = 180, in = 180, looseness=2] (.1,.75);
            \draw[blue, thick] (.1,.75) to [out = 180, in = 180, looseness=2] (.1,1.05);
            \draw[blue, thick, ->] (.1,1.05) to [out = 180, in = 180, looseness=2] (.15,1.35);
        \end{tikzpicture}
        \(\leadsto\)
        \(\frac{1}{4}\)
        \begin{ytableau}
            1
        \end{ytableau}
        \(\otimes\)
        \begin{ytableau}
            2 & 2\\
            3 \\
            4 \\
        \end{ytableau},
        \\
        \begin{tikzpicture}[baseline=(O.base)]
            \node (O) at (0,.6) {~};
            \filldraw[fill=blue!25] (0,0) rectangle (.3,.3);
            \node at (.15,.15) {\scriptsize \(4\)};
            \draw (0,.3) rectangle (.3,.6);
            \node at (.15,.45) {\scriptsize \(3\)};
            \draw (0,.6) rectangle (.3,.9);
            \node at (.15,.75) {\scriptsize \(2\)};
            \draw (0,.9) rectangle (.3,1.2);
            \node at (.15,1.05) {\scriptsize \(1\)};
            \draw[fill=blue!25] (.3,.9) rectangle (.6,1.2);
            \node at (.45,1.05) {\scriptsize \(2\)};
            \draw (0,1.2) rectangle (.3,1.5);
            \draw[blue, thick] (.1,.15) to [out = 180, in = 180, looseness=1.5] (.4,1.05);
            \draw[blue, thick, ->] (.4,1.05) to [out = 90, in = 0] (.15,1.35);
        \end{tikzpicture}
        \(\leadsto\)
        \(\frac{1}{4}\)
        \begin{ytableau}
            2
        \end{ytableau}
        \(\otimes\)
        \begin{ytableau}
            1 & 4\\
            2 \\
            3 \\
        \end{ytableau},
        &
        \begin{tikzpicture}[baseline=(O.base)]
            \node (O) at (0,.6) {~};
            \filldraw[fill=blue!25] (0,0) rectangle (.3,.3);
            \node at (.15,.15) {\scriptsize \(4\)};
            \draw[fill=blue!25] (0,.3) rectangle (.3,.6);
            \node at (.15,.45) {\scriptsize \(3\)};
            \draw (0,.6) rectangle (.3,.9);
            \node at (.15,.75) {\scriptsize \(2\)};
            \draw (0,.9) rectangle (.3,1.2);
            \node at (.15,1.05) {\scriptsize \(1\)};
            \draw[fill=blue!25] (.3,.9) rectangle (.6,1.2);
            \node at (.45,1.05) {\scriptsize \(2\)};
            \draw (0,1.2) rectangle (.3,1.5);
            \draw[blue, thick] (.1,.15) to [out = 180, in = 180, looseness=2] (.1,.45);
            \draw[blue, thick] (.1,.45) to [out = 180, in = 180, looseness=1.5] (.4,1.05);
            \draw[blue, thick, ->] (.4,1.05) to [out = 90, in = 0] (.15,1.35);
        \end{tikzpicture}
        \(\leadsto\)
        \(- \frac{1}{4}\)
        \begin{ytableau}
            2
        \end{ytableau}
        \(\otimes\)
        \begin{ytableau}
            1 & 3\\
            2 \\
            4 \\
        \end{ytableau},
        &
        \begin{tikzpicture}[baseline=(O.base)]
            \node (O) at (0,.6) {~};
            \filldraw[fill=blue!25] (0,0) rectangle (.3,.3);
            \node at (.15,.15) {\scriptsize \(4\)};
            \draw[fill=blue!25] (0,.3) rectangle (.3,.6);
            \node at (.15,.45) {\scriptsize \(3\)};
            \draw[fill=blue!25] (0,.6) rectangle (.3,.9);
            \node at (.15,.75) {\scriptsize \(2\)};
            \draw (0,.9) rectangle (.3,1.2);
            \node at (.15,1.05) {\scriptsize \(1\)};
            \draw[fill=blue!25] (.3,.9) rectangle (.6,1.2);
            \node at (.45,1.05) {\scriptsize \(2\)};
            \draw (0,1.2) rectangle (.3,1.5);
            \draw[blue, thick] (.1,.15) to [out = 180, in = 180, looseness=2] (.1,.45);
            \draw[blue, thick] (.1,.45) to [out = 180, in = 180, looseness=2] (.1,.75);
            \draw[blue, thick] (.1,.75) to [out = 180, in = 180, looseness=2] (.4,1.05);
            \draw[blue, thick, ->] (.4,1.05) to [out = 90, in = 0] (.15,1.35);
        \end{tikzpicture}
        \(\leadsto\)
        \(\frac{1}{4}\)
        \begin{ytableau}
            2
        \end{ytableau}
        \(\otimes\)
        \begin{ytableau}
            1 & 2\\
            3 \\
            4 \\
        \end{ytableau}
    \end{tabular}
\end{center}
\endgroup
    
In this case, we must straighten the image of two of the \(1\)-paths (starred), which we show in Section \ref{Constructing the Pieri Inclusion for Removing One Box subsection: Straightening Example}.
After straightening we have, up to row permutations,
\begin{align*}
    \Phi_1\left( \begin{ytableau}
        1 & 2\\
        2 \\
        3 \\
        4
    \end{ytableau} \right)
    =
    &- 
    \begin{ytableau}
        4
    \end{ytableau}
    \otimes
    \begin{ytableau}
        1 & 2\\
        2 \\
        3 \\
    \end{ytableau}
    + 
    \begin{ytableau}
        3
    \end{ytableau}
    \otimes
    \begin{ytableau}
        1 & 2\\
        2 \\
        4 \\
    \end{ytableau}
    - \frac{3}{4} \ 
    \begin{ytableau}
        2
    \end{ytableau}
    \otimes
    \begin{ytableau}
        1 & 2\\
        3 \\
        4 \\
    \end{ytableau}\\
    &+ \frac{1}{4} \ 
    \begin{ytableau}
        1
    \end{ytableau}
    \otimes
    \begin{ytableau}
        2 & 2\\
        3 \\
        4 \\
    \end{ytableau}
    + \frac{1}{4} \ 
    \begin{ytableau}
        2
    \end{ytableau}
    \otimes
    \begin{ytableau}
        1 & 4\\
        2 \\
        3 \\
    \end{ytableau}
    - \frac{1}{4} \ 
    \begin{ytableau}
        2
    \end{ytableau}
    \otimes
    \begin{ytableau}
        1 & 3\\
        2 \\
        4 \\
    \end{ytableau}.
\end{align*}
    
For an example of two-box removal, see Section \ref{Constructing the Pieri Inclusion for Removing Many Boxes subsection: Two Path Example}.

\subsection{} \label{Introduction subsection: Related results}

In results of Eisenbud, Fl\o stad, and Weyman \cite{eisenbud2011existence} and of Sam and Weyman \cite{sam2011pieri}, Pieri inclusions are used to compute pure free resolutions for classical groups. Sam has also built a package for Macaulay2 (PieriMaps) \cite{sam2009computing} that computes Pieri inclusions explicitly using the algorithm from \cite{sam2011pieri}. 

In \cite{pragacz1985complexes}, Weyman and Pragacz use Pieri inclusion maps to describe Lascoux resolutions.
We use the explicit description of Pieri inclusions to give minimal free resolutions of modules of covariants (in the context of Weyl's fundamental theorems).


\subsection{} \label{Introduction subsection: Organization}

In Section \ref{section: Constructing Schur--Weyl Modules} we construct the Schur--Weyl modules \(\bS_{\lambda}\).
In Sections \ref{section: Constructing the Pieri Inclusion for Removing One Box} and \ref{section: Constructing the Pieri Inclusion for Removing Many Boxes} we describe the construction of the Pieri inclusion in the one-box removal case (\(V\)) and \(m\)-box removal case (\(\wedge^m(V)\)), respectively.
In Sections \ref{section: Showing the Pieri Inclusion Removing One Box is a GL(V)-map} and \ref{section: Showing the Pieri Inclusion Removing Many Boxes is a GL(V)-map} we show that the Pieri inclusions are \(\GL(V)\)-maps in the one-box removal and \(m\)-box removal cases, respectively, with the tools for these proofs given in Section \ref{section: Generating Garnir Relations and Tools for Collapsing Sums}.
In Section \ref{section: The Image of a Highest Weight Vector and Computational Complexity} we show the one-box removal map is the negative of Olver's description via the uniqueness of an equivariant map and, in a similar way, show that the \(m\)-box removal case is equal to iterating one-box removal. 
We also show in this section that the same description of Pieri inclusions gives a map in the case \(\bS_{\mu}(V) \hookrightarrow \Sym^m(V) \otimes \bS_{\lambda}(V)\) and we compare the computational complexity of the one-box removal description to that of Olver.\\

\section{Constructing Schur--Weyl Modules} \label{section: Constructing Schur--Weyl Modules}

\subsection{} \label{Constructing Schur--Weyl Modules subsection: Young symmetrizer}

From now on, let \(\lambda = (\lambda_1,\ldots,\lambda_r)\) be a fixed partition of \(d\).
Let \(\mT_{\lambda, n}\) be the set of all tableau \(T\) with shape \(\lambda\) and arbitrary filling from the alphabet \(\{1, \ldots, n\}\).
A tableau \(T \in \mT_{\lambda, n}\) is called semi-standard if the filling is non-decreasing across the rows and strictly increasing down the columns.
Fix the canonical tableau \(T_{0}\) of shape \(\lambda\) labeled with \(\left\{ 1, \ldots, d \right\}\), starting with the top left most box and filling across each row, so the first box of the first row is labeled \(1\), the first box of the second row is labeled \(\lambda_1 + 1\), etc.
Via this labeling, the symmetric group \(\fS_d\) acts on the set of tableau with shape \(\lambda\) with respect to any given alphabet. 

Let
\[
    P = P_{\lambda} = \left\{ \pi \in \fS_{d} : \pi \text{ preserves the rows of } T_{0}\right\}
\]
and
\[
    Q = Q_{\lambda} = \left\{ \sigma \in \fS_{d} : \sigma \text{ preserves the columns of } T_{0}\right\}.
\]
As elements of the group algebra of \(\fS_{d}\), \(\bC \fS_{d}\), define
\[
    A_{\lambda} = \sum_{\pi \in P} \pi
    \qquad \text{ and } \qquad 
    B_{\lambda} = \sum_{\sigma \in Q} \left( -1\right)^{\sigma} \sigma.
\]
The Young Symmetrizer is then defined as \( C_{\lambda} = A_{\lambda }B_{\lambda }\).
Note this convention symmetrizes along rows first and antisymmetrizes along columns second.

\subsection{} \label{Constructing Schur--Weyl Modules subsection: Schur--Weyl modules}

From now on, fix a complex vector space \(V\) of dimension \(n\). Let \(\fS_{d}\) also act on elements of \(V^{\otimes d}\) by permuting the coordinates. 
In particular, \(C_{\lambda}\) acts on \(V^{\otimes d}\). 
The corresponding Schur--Weyl module is \(\bS_{\lambda} = \bS_{\lambda}(V^{\otimes d}) = C_{\lambda } \cdot V^{\otimes d}\). 
Clearly \(\bS_{\lambda }\) is a \(\GL(V)\)-module. 
When the number of rows of \(\lambda\) is at most \(n\), it is known that \(\bS_{\lambda }\) is an irreducible representation of \(\GL(V)\) and that all (in-equivalent) polynomial irreducible representations are constructed this way.

Write \(\{e_i\}_{1 \le i \le n}\) for the standard basis of \(V\).
For \(T \in \mT_{\lambda ,n}\), define \(e_T \in \bS_{\lambda}\) by
\[
    e_T = C_{\lambda }\cdot ((e_{T_{11}}\otimes \cdots \otimes e_{T_{1\lambda_{1}}})\otimes
    \cdots \otimes (e_{T_{N1}}\otimes \cdots \otimes e_{T_{N \lambda_{N}}}))
\]
where \(T_{i,j}\) is the entry in the \(i\)th row and \(j\) column of \(T\) starting from the top left. 
Clearly \(\bS_{\lambda }\) is spanned by such elements, and it is known that a basis is given by the semistandard ones.

\subsection{} \label{Constructing Schur--Weyl Modules subsection: Highest Weights}

Let \(V_{\bullet}\) be the standard flag in \(V\),
\[
    V_{\bullet} : \quad V_i = \text{span}\{e_1, \ldots, e_i\} \quad (1 \le i \le n).
\]
Let \(B \subset \GL(V)\) be the Borel subgroup given by
\[
    B = \{g \in \GL(V) : g V_i \subset V_i \text{ for } 1 \le i \le n\}.
\]
Throughout this paper, all highest weights are with respect to \(B\).
The highest weight vector of \(\bS_{\lambda}(V)\) is
\[
    e_{T_\lambda} = C_{\lambda}\cdot (\underbrace{(e_{1}\otimes \cdots \otimes e_{1}}_{\lambda_1}) \otimes
    \cdots \otimes (\underbrace{(e_{r}\otimes \cdots \otimes e_{r}}_{\lambda_r})).
\]
That is, \(T_\lambda\) is the tableau of shape \(\lambda\) with all ones in the first row, all twos in the second row, etc.
For example, if \(\lambda = (5,3,3,1,1)\), 
\[
    T_{\lambda} = \ytableaushort{1 1 1 1 1, 2 2 2, 3 3 3, 4, 5}
\]
and
\[
    e_{T_\lambda} = C_{\lambda} \cdot ( e_1 \otimes e_1 \otimes e_1 \otimes e_1 \otimes e_1 \otimes e_2 \otimes e_2 \otimes e_2 \otimes e_3 \otimes e_3 \otimes e_3 \otimes e_4 \otimes e_5).
\]

\subsection{} \label{Constructing Schur--Weyl Modules subsection: Garnirs}

For any subset \(A\) of boxes of \(T_{0}\), let \(w_A\) be the maximum width of a row containing an element of \(A\). 
Let 
\[
    \fS_A = \left\{ \sigma \in \fS_{d}  :  \sigma \text{ preserves } A \text{ and fixes } T_0 \setminus A\right\}.
\]
When \(\left\vert A \right\vert > w_A\), define a Garnir operator as an element of \(\bC \fS_{d}\) by
\[
    G_A = \sum_{\sigma \in \fS_A} \sigma .
\]

\subsection{} \label{Constructing Schur--Weyl Modules subsection: The Construction}

Let \(\mF_{\lambda, n}\) be the formal \(\bC\)-span of symbols \(T \in \mT_{\lambda ,n}\) and let \(\mR_{\lambda, n}\) be the subspace of \(\mF_{\lambda, n}\) generated by all 
\begin{equation} \label{Constructing Schur--Weyl Modules subsection: The Construction equation: R1}
    T_{1} - T_{2}, \text{ where } T_{1} \text{ and } T_{2} \text{ agree up to a row permutation}, %
\end{equation}
and
\begin{equation} \label{Constructing Schur--Weyl Modules subsection: The Construction equation: R2}
    G_A (T), \text{ where } A \subset T_0 \text{ with } \left\vert A\right\vert > w_A.
\end{equation}

\begin{theorem*} 
    As \(\GL(V)\)-modules, we have
    \[
        \mF_{\lambda, n} / \mR_{\lambda, n} \cong \bS_{\lambda}.
    \]
\end{theorem*}

\begin{proof}
    The map is induced by \(T \mapsto e_T\). See, for example, \cite[\S 8]{fulton1997young}, where the convention is transpose to ours.
\end{proof}

\section{Constructing the Pieri Inclusion for Removing One Box} \label{section: Constructing the Pieri Inclusion for Removing One Box}

\subsection{} \label{Constructing the Pieri Inclusion for Removing One Box subsection: Block Form}

We will write \(\lambda = (\lambda_1, \ldots, \lambda_r)\) in block form as \(\lambda = (w_1^{h_1}, \ldots, w_N^{h_N})\), where \(w_i < w_{i + 1}\) and exactly \(h_i\) parts of \(\lambda\) are equal to \(w_i\). 
That is, \(N\) is the number of blocks in \(\lambda\), where block \(1\) is the lowest geometrically, \(w_b\) is the width of block \(b\), and \(h_b\) is the height of block \(b\).
See Figure \ref{Constructing the Pieri Inclusion for Removing One Box subsection: Block Form figure: lambda with N blocks}.
For example, we will write \((5,2,2,2,1,1)\) as \((1^2, 2^3, 5)\),
\[
    (5, 2, 2, 2, 1, 1) = (1^2, 2^3, 5) = %
    \begin{tikzpicture}[scale=.75, baseline={([yshift=-.5ex]current bounding box.center)}]
        \draw (0,0) rectangle (.5,1);
        \draw (0,.5) -- (.5,.5);
        \draw [decorate,decoration={brace,mirror}] (.6,.05) -- (.6,.95) node [midway, xshift=0.75cm] {\small{block \(1\)}};
        
        \draw (0,1) rectangle (1,2.5);
        \draw (.5,1) -- (.5,2.5);
        \draw (0,1.5) -- (1,1.5);
        \draw (0,2) -- (1,2);
        \draw [decorate,decoration={brace,mirror}] (1.1,1.05) -- (1.1,2.45) node [midway, xshift=0.75cm] {\small{block \(2\)}};
        
        \draw (0,2.5) rectangle (2.5,3);
        \draw (.5,2.5) -- (.5,3);
        \draw (1,2.5) -- (1,3);
        \draw (1.5,2.5) -- (1.5,3);
        \draw (2,2.5) -- (2,3);
        \draw [decorate,decoration={brace,mirror}] (2.6,2.55) -- (2.6,2.95) node [midway, xshift=0.75cm] {\small{block \(3\)}};
        
    \end{tikzpicture}.
\]

\begin{figure}[t]
\begin{tikzpicture}[scale=.75]
    \draw (0,0) rectangle (1.25,1);
    \draw[<->] (0,.25) -- (1.25,.25);
    \node at (.25,.4) {\tiny \(w_1\)};
    \draw[<->] (1,0) -- (1,1);
    \node at (.8,.7) {\tiny \(h_1\)};
    
    \draw (0,1) rectangle (2,2);
    \draw[<->] (0,1.25) -- (2,1.25);
    \node at (.5,1.4) {\tiny\(w_2\)};
    \draw[<->] (1.5,1) -- (1.5,2);
    \node at (1.25,1.75) {\tiny \(h_2\)};
    
    \draw[dashed] (0,2) rectangle (3,3);
    \node at (1.5,2.6) {\(\vdots\)};
    
    \draw (0,3) rectangle (4,6);
    \draw[<->] (0,4) -- (4,4);
    \node at (1,4.2) {\tiny \(w_{N-1}\)};
    \draw[<->] (3.5,3) -- (3.5,6);
    \node at (2.9,5) {\tiny \(h_{N-1}\)};
    
    \draw (0,6) rectangle (6,8);
    \draw[<->] (0,6.5) -- (6,6.5);
    \node at (2,6.7) {\tiny \(w_N\)};
    \draw[<->] (5.5,6) -- (5.5,8);
    \node at (5.15,7.5) {\tiny \(h_N\)};
    
    \node at (-.8,.5) {\tiny{block \(1\)}};
    \node at (-.8,1.5) {\tiny{block \(2\)}};
    \node at (-1.3,4.5) {\tiny{block \(N-1\)}};
    \node at (-.85,7) {\tiny{block \(N\)}};
    
\end{tikzpicture}
\caption{The shape \(\lambda\) with \(N\) blocks.}
\label{Constructing the Pieri Inclusion for Removing One Box subsection: Block Form figure: lambda with N blocks}
\end{figure}
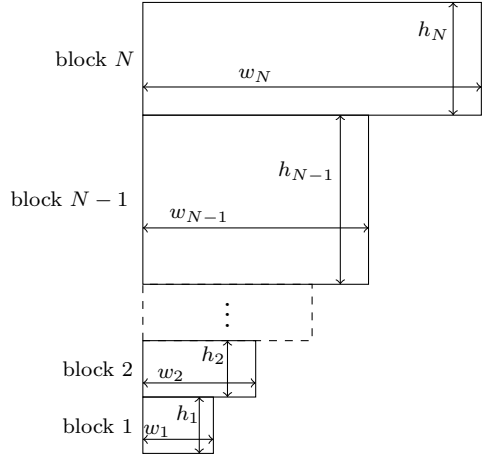

\subsection{} \label{Constructing the Pieri Inclusion for Removing One Box subsection: Phi Diagram}

For any box \(x \in T_0\) at the bottom right of some block \(k\), i.e. so that \(\lambda \setminus \{x\}\) is still a diagram, we will define the map
\[
    \Phi_1 : \mF_{\lambda,n} \to V \otimes \mF_{\lambda \setminus \{x\}, n} 
\]
on a basis and then show that \(\Phi_1\) descends to a \(\GL(V)\)-module map.

\begin{center}
    \begin{tikzpicture}
            \node at (0,2) (F) {\(\mF_{\lambda,n}\)};
                \node at (6,2) (Fminus) {\(V \otimes \mF_{\lambda \setminus \{x\}, n}\)};
                    \node (S) [below=of F]{\(\ds \mF_{\lambda ,n} \big/\mR_{\lambda ,n}\cong \bS_{\lambda}(V)\)};
                        \node (Sminus) [below=of Fminus]{\(V \otimes \mF_{\lambda \setminus \{x\}, n} \big/\mR_{\lambda, n} \cong V  \otimes \bS_{\lambda \setminus \{x\}}(V)\)};
    \node at (3,2.2) {\(\Phi_1\)};
    \node at (2.25,.6) {\(\Phi_1\)};
    
    \draw[->] (F.east) -- (Fminus.west);
    \draw[->] (F.south) -- (S.north);
    \draw[->] (Fminus.south) -- (Sminus.north);
    \draw[->, dashed] (S.east) -- (Sminus.west);
\end{tikzpicture}
\end{center}

\subsection{} \label{Constructing the Pieri Inclusion for Removing One Box subsection: Shape Notation}

We first introduce further notation. 
For a given shape \(\lambda\), let \([b]\) denote the \(b\)th block, \([b](i)\) denote the \(i\)th row of the \(b\)th block, and \([b](i,j)\) denote the box in block \(b\), row \(i\), and column \(j\), with block \(1\) and row \(1\) the lowest geometrically and column \(1\) the furthest left. 
We write
\[
    [b](i,j) \le [c](k,l)
\]
if \([b](i,j)\) is geometrically (weakly) lower than \([c](k,l)\), i.e. \(b < c\) or \(b = c\) and \(i \le k\). 
The strict inequality is defined in the natural way. 
We will extend this notation to compare boxes, rows, and blocks in the natural way. 
For a given \(T \in \mF_{\lambda, n}\), we denote the entry in box \([b](i,j)\) by \(T_{[b](i,j)}\). 
For example, if 
\[
    T = \begin{ytableau} 1 & 1 & 3 & 3 & 4 \\ 2 & 2 \\ 3 & 4 \\ 4 & 5\\ 6\\ 7\\ \end{ytableau}
\]
then \(T_{[1](2,1)} = 6\) and \(T_{[3](1,5)} = 4\).

\subsection{} \label{Constructing the Pieri Inclusion for Removing One Box subsection: Evacuation Route}

An evacuation route \(R\) is a selection of a string of boxes starting from the bottom of some block.
For example an evacuation route on \((1^2,3,5^3,7^2)\) is given by the shaded boxes in the diagram below.
\[
    \begin{ytableau}
        ~ & ~ & ~ & ~ & ~ & ~ & ~ \\
        ~ & ~ & *(blue!25) & ~ & ~ & ~ & ~ \\
        ~ & ~ & ~ & ~ & ~\\
        ~ & ~ & ~ & ~ & *(blue!25)\\
        ~ & *(blue!25) & ~ & ~ & ~\\
        ~ & ~ & ~\\
        *(blue!25)\\
        *(blue!25)
    \end{ytableau}
\]
This example shows that an evacuation route does not need to contain a box from every row, however, it cannot skip rows within a block. This is best illustrated with a non-example. 
The shaded selection of boxes below is not an evacuation route on \((2, 3^2, 5^4, 7^2)\) since a box in row \([3](3)\) (that is, the third row in the third block) is selected, while there is no box selected from row \([3](2)\).
\[
    \begin{ytableau}
        ~ & ~ & ~ & ~ & ~ & ~ & ~\\
        ~ & ~ & ~ & *(blue!25) & ~ & ~ & ~\\
        ~ & ~ & ~ & ~ & *(blue!25)\\
        ~ & ~ & *(blue!25) & ~ & ~\\
        ~ & ~ & ~ & ~ & ~\\
        *(blue!25) & ~ & ~ & ~ & ~\\
        ~ & ~ & *(blue!25)\\
        ~ & *(blue!25) & ~\\
        ~ & ~
    \end{ytableau}
\]

Formally, we have the following definition.

\begin{definition*}[Evacuation Route]
    An evacuation route \(R\) starting at \([b_0]\) is a subset of boxes in \(T_0\) such that \(R\) contains a box in row \([b_0](1)\), \(R\) contains at most one box per row, and if \([b](i,j) \in R\), then \([b](k,j_k) \in R\) for all \(1 \le k < i\) and some \(1 \le j_k \le w_b\).
\end{definition*}

\subsection{} \label{Constructing the Pieri Inclusion for Removing One Box subsection: 1-Path}

A \(1\)-path \(P\) on \(\lambda\) moves boxes up the diagram via some associated evacuation route \(R^P\). 
We will treat a \(1\)-path as acting on general shapes, where the highest box in \(R^P\) is ``removed'' by the \(1\)-path and viewed as being moved to the box \([N + 1](1,1)\) attached to the top of \(T_0\).
Below we illustrate a \(1\)-path moving boxes up via an evacuation route. 
We highlight only the boxes in the evacuation route.
\[
    \begin{tikzpicture}[scale = .75, baseline={([yshift=-.5ex]current bounding box.center)}]
        \draw (0,0) rectangle (.5,1);
        \draw (0,1) rectangle (1.5,3);
        \draw (0,3) rectangle (3,3.5);
        \draw (0,3.5) rectangle (3.5,5);
        \draw (0,5) rectangle (.5,5.5);
        
        \filldraw[fill=blue!20] (0,0) rectangle (.5,.5);
        \filldraw[fill=blue!20] (0,.5) rectangle (.5,1);
        \filldraw[fill=blue!20] (1,1) rectangle (1.5,1.5);
        \filldraw[fill=blue!20] (.5,1.5) rectangle (1,2);
        \filldraw[fill=blue!20] (2.5,3) rectangle (3,3.5);
        \filldraw[fill=blue!20] (2,3.5) rectangle (2.5,4);
        
        \draw[blue, thick] (.25,.25) to [out = 180, in = 180] (.25,.75);
        \draw[blue, thick] (.25,.75) to [out = 90, in = 180] (1.25,1.25);
        \draw[blue, thick] (1.25,1.25) to [out = 90, in = 0] (.75,1.75);
        \draw[blue, thick] (.75,1.75) to [out = 90, in = 180] (2.75,3.25);
        \draw[blue, thick] (2.75,3.25) to [out = 90, in = 0] (2.25,3.75);
        \draw[blue, thick, ->] (2.25,3.75) to [out = 90, in = 0] (.25,5.25);
        
    \end{tikzpicture}
\]

Formally, we have the following definition.

\begin{definition*}[\(1\)-path]
    Let
    \[
        X = \{x_1 := [b_1](1,w_{b_1})\},
    \]
    and
    \[
        Y = \{y_1 := [N+1](1,1)\}, 
    \]
    where \(Y\) is viewed as block \([N + 1]\) attached to the top of \(T_0\). A
    \[
        1\text{-path } P \text{ removing } X
    \]
    is a map of boxes
    \[
        P:\lambda \cup Y \to \lambda \cup Y
    \]
    along with an evacuation route \(R = R^P\) such that the following hold.
    \begin{itemize}
        \item \(R\) starts at \([b_1]\). Note that \(R\) can contain \(x_1\), though this is not a requirement. 
        \item \(P\) is geometrically increasing on rows, with \(P\) strictly increasing on \(R\). 
        That is, for all boxes \(x \in \lambda \cup Y\), \(x \le P(x)\).
        \item If \(R_1\) is the orbit of \(x_1\) under \(P^\bN\), then \(y_1 \in R_1\) and \(\ds R \cup X \cup Y = R_1\).
        \item \(P\) preserves row order in \(R\) within blocks. That is, if \([b](i,j), \ [b](k,l) \in R\) with \(i < k\) and \(P([b](i,j)), P([b](k,l)) \in [b]\), then \(P([b](i,j)) < P([b](k,l))\).
        \item \(P\) fixes those boxes not in \(R\) or \(X\), i.e. \(P = \text{id}_{\lambda \cup Y}\) except on \(R \cup X\), and \(P(R) = R \setminus X \cup Y\).
    \end{itemize}
\end{definition*}

\subsection{} \label{Constructing the Pieri Inclusion for Removing One Box subsection: Defining h(b) and H(B)}

We now define the components of the formulation of the Pieri inclusion \(\Phi_1\) removing one box.
For a \(1\)-path \(P\) removing \(X\) with evacuation route \(R^P\), let \(h^P\) be the number of rows in \(R^P\) and \((-1)^P := (-1)^{h^P}\). 
For \(b = b_1, \ldots, N\), let \(h_b^P\) to be the number of rows in \(R^P \cap [b]\). 
For \(b \ge b_1 + 1\) define \(h(b) = w_b - w_{b-1} + h_{b-1}\) to be the hook length of block \(b\), and for \(b = b_1 + 1, \ldots, N\) define
\[
    H(b) = \sum_{j=b_1 + 1}^{b} h(j).
\]

\ytableausetup{nosmalltableaux}
For \(b = b_1 + 1, \ldots, N\), let 
\[
    H_b(P) =    \begin{cases} 	
                    1 & \text{ if } R^P \cap [b] = \emptyset\\%
					H(b) & \text{ otherwise }
				\end{cases},
\]
and let \(H_{b_1}(P) = 1\). Then define
\[
    H(P) = \prod_{b=b_1}^n H_b(P).
\]

\begin{example*} 
For the partition \((1,3^2,4^3,6^2)\) and \(X = \{[1](1,1)\}\) (shaded),
\[
    \begin{tikzpicture}[scale=.8, baseline={([yshift=-.5ex]current bounding box.center)}]
        \filldraw[fill=gray!50] (0,0) rectangle (.5,.5);
        
        \draw (0,.5) rectangle (1.5,1.5);
        \draw (.5,.5) -- (.5,1.5);
        \draw (1,.5) -- (1,1.5);
        \draw (0,1) -- (1.5,1);
        
        \draw (0,1.5) rectangle (2,3);
        \draw (0,2) -- (2,2);
        \draw (0,2.5) -- (2,2.5);
        \draw (.5,1.5) -- (.5,3);
        \draw (1,1.5) -- (1,3);
        \draw (1.5,1.5) -- (1.5,3);
        
        \draw (0,3) rectangle (3,4);
        \draw (0,3.5) -- (3,3.5);
        \draw (.5,3) -- (.5,4);
        \draw (1,3) -- (1,4);
        \draw (1.5,3) -- (1.5,4);
        \draw (2,3) -- (2,4);
        \draw (2.5,3) -- (2.5,4);
        
        \draw[<->] (.6,0) -- (.6,.4) -- (1.5,.4);
        \node at (1.1,.1) {\tiny \(h(2)\)};
        
        \draw[<->] (1.6,.5) -- (1.6,1.4) -- (2,1.4);
        \node at (2,1) {\tiny \(h(3)\)};
    
        \draw[<->] (2.1,1.5) -- (2.1,2.9) -- (3,2.9);
        \node at (2.5,2.5) {\tiny \(h(4)\)};
    \end{tikzpicture},
\]
we have
\[
    h(2) = 3 - 1 + 1 = 3, \quad h(3) = 4 - 3 + 2 = 3, \quad h(4) = 6 - 4 + 3 = 5,
\]
\[
    H(1) = 1, \quad H(2) = 3, \quad H(3) = 6, \quad \text{ and } \quad H(4) = 11.
\]
\end{example*}

\subsection{} \label{Constructing the Pieri Inclusion for Removing One Box subsection: Defining Phi One Box Removal}
For \(T \in \mF_{\lambda,n}\), denote by \(\alpha_1^P\) the entry in the box \(P^{-1}(y_1) \in T\) and extend \(P\) to \(T\) by acting on the entries, with the image %
\[
    P(T) = Y_P \otimes T_P \in V \otimes \mF_{\lambda \setminus X,n},
\]
where
\[
    Y_P = E_X \ytableaushort{{\alpha_1^P}},
\]
which is standard form notation is \(e_{\alpha_1^P} \in V\), and \(T_P \in \mF_{\lambda \setminus X,n}\) is defined by \(\left( T_P \right)_{[b](i,j)} = T_{P^{-1}([b](i,j))}\). %
We omit \(E_X\) and just write 
\[
    \ytableaushort{{\alpha_1^P}} \quad \text{ in place of } \quad E_X \ytableaushort{{\alpha_1^P}}
\]
in the image of \(P(T)\).

\begin{definition*}
The map \(\ds \Phi_1 : \mF_{\lambda, n} \to V \otimes \mF_{\lambda \setminus X, n}\) is given by
\[
    \Phi_1(T) = \sum_{P} \frac{(-1)^P}{H(P)} P(T)
\]
where the sum is over all \(1\)-paths \(P\) removing \(X\).
\end{definition*}

\subsection{} \label{Constructing the Pieri Inclusion for Removing One Box subsection: Straightening Example}

\ytableausetup{smalltableaux}

We now compute the straightening from Section \ref{Introduction subsection: Example straightening}.
If 
    \[
        A_1 = \{[2](1,1), [2],(1,2), [1](2,1)\} = 
        \begin{ytableau}            
            *(black!25) & *(black!25)\\
            *(black!25) \\
            ~ \\
        \end{ytableau}
    \]
    then we have, mod \(\mR_{(2,1,1), 4}\),
    \[
        \frac{1}{2} G_{A_1}
        \left( \begin{ytableau}
            2 & 4\\
            2 \\
            3 \\
        \end{ytableau} \right)
        =
        2 ~
        \begin{ytableau}
            2 & 4\\
            2 \\
            3 \\
        \end{ytableau}
        +
        \begin{ytableau}
            2 & 2\\
            4 \\
            3 \\
        \end{ytableau}.
    \]
    Then if
    \[
        A_2 = \{[1](1,1), [1](2,1)\} = 
        \begin{ytableau}            
            ~ & ~\\
            *(black!25) \\
            *(black!25)
        \end{ytableau},
    \]
    we have, mod \(\mR_{(2,1,1), 4}\),
    \[
        G_{A_2}
        \left( \begin{ytableau}
            2 & 2\\
            4 \\
            3 \\
        \end{ytableau} \right)
        =
        \begin{ytableau}
            2 & 2\\
            4 \\
            3 \\
        \end{ytableau}
        +
        \begin{ytableau}
            2 & 2\\
            3 \\
            4 \\
        \end{ytableau}.
    \]
    Thus, as elements of \(\mF_{(2,1,1), 4} / \mR_{(2,1,1), 4}\) we have
    \[
        \frac{1}{4} \ 
        \begin{ytableau}
            1
        \end{ytableau}
        \otimes
        \begin{ytableau}
            2 & 4\\
            2 \\
            3 \\
        \end{ytableau}
        =
        \frac{1}{8} ~
        \begin{ytableau}
            1
        \end{ytableau}
        \otimes
        \begin{ytableau}
            2 & 2\\
            3 \\
            4 \\
        \end{ytableau}.
    \]
    Similarly, as elements of \(\mF_{(2,1,1), 4} / \mR_{(2,1,1), 4}\) we have
    \[
        \frac{1}{4} ~
        \begin{ytableau}
            1
        \end{ytableau}
        \otimes
        \begin{ytableau}
            2 & 3\\
            2 \\
            4 \\
        \end{ytableau}
        =
        - \frac{1}{8} \ 
        \begin{ytableau}
            1
        \end{ytableau}
        \otimes
        \begin{ytableau}
            2 & 2\\
            3 \\
            4 \\
        \end{ytableau}.
    \]
Note that in this example the terms that were straightened cancelled with each other and so did not appear in the image, this is not the case in general.

\section{Constructing the Pieri Inclusion for Removing Many Boxes} \label{section: Constructing the Pieri Inclusion for Removing Many Boxes}

\subsection{} \label{Constructing the Pieri Inclusion for Removing Many Boxes subsection: Phi Diagram}

Let \(X = \{x_1 = [b_1](1,w_{b_1}), \ldots, x_m = [b_m](i_m,w_{b_m})\}\) be a set of \(m\) boxes in \(\lambda\) with \(x_i < x_{i + 1}\) so that removing the boxes in \(X\) from \(T_0\) gives a Young diagram and let \(\lambda \setminus X\) be the associated partition.
We call such a set \(X\) a removal set for \(T_0\) (or for \(\lambda\)).
As before, we will define the map \(\Phi_m : \mF_{\lambda,n} \to F_m \otimes \mF_{\lambda \setminus X, n}\) on a basis, where \(F_m = \bigwedge^m V\), and then show that \(\Phi_m\) is a \(\GL(V)\)-map.

\begin{center}
\begin{tikzpicture}
    \node at (0,2) (F) {\(\mF_{\lambda,n}\)};
    \node at (6,2) (Fminus) {\(F_m \otimes \mF_{\lambda \setminus X,n}\)};
    \node (S) [below=of F]{\(\ds \mF_{\lambda ,n} \big/\mR_{\lambda ,n}\cong \bS_{\lambda}(V)\)};
    \node (Sminus) [below=of Fminus]{\(F_m \otimes \mF_{\lambda\setminus X ,n} \big/\mR_{\lambda ,n}\cong F_m 
        \otimes \bS_{\lambda\setminus X}(V)\)};
    \node at (3,2.2) {\(\Phi_m\)};
    \node at (2.25,.6) {\(\Phi_m\)};
    
    \draw[->] (F.east) -- (Fminus.west);
    \draw[->] (F.south) -- (S.north);
    \draw[->] (Fminus.south) -- (Sminus.north);
    \draw[->, dashed] (S.east) -- (Sminus.west);
\end{tikzpicture}
\end{center}

\subsection{} \label{Constructing the Pieri Inclusion for Removing Many Boxes subsection: m-path}

Extending the notion of a \(1\)-path, an \(m\)-path on \(\lambda\) is a map of boxes that moves boxes up the diagram via some associated evacuation route with \(m\) interlaced orbits. 
As with \(1\)-paths, we treat \(m\)-paths as acting on general shapes, where the highest \(m\) boxes in \(R^P\) are ``removed'' by the \(m\)-path and viewed as being moved to the boxes \([N + 1](1,1), \ldots, [N + 1](m, 1)\) attached to the top of \(T_0\).
An example of a \(2\)-path is pictured below. We highlight only the boxes in the evacuation route.
\[
    \begin{tikzpicture}[scale=.75, baseline={([yshift=-.5ex]current bounding box.center)}]
        \draw (0,0) rectangle (.5,1);
        \draw (0,1) rectangle (2.5,3);
        \draw (0,3) rectangle (3,4.5);
        \draw (0,4.5) rectangle (4,6);
        \draw (0,6) rectangle (6,6.5);
        \draw (0,6.5) rectangle (.5,7);
        \draw (0,7) rectangle (.5, 7.5);
        
        \filldraw[fill=blue!20] (2,1) rectangle (2.5,1.5);
        \filldraw[fill=blue!20] (1,2) rectangle (1.5,2.5);
        \filldraw[fill=blue!20] (0,3.5) rectangle (.5,4);
        \filldraw[fill=blue!20] (3.5,4.5) rectangle (4,5);
        \filldraw[fill=blue!20] (.5,6) rectangle (1,6.5);
        
        \draw[color=red,pattern=dots, pattern color=red] (2,1.5) rectangle (2.5,2);
        \draw[color=red,pattern=dots, pattern color=red] (.5,3) rectangle (1,3.5);
        \draw[color=red,pattern=dots, pattern color=red] (2,4) rectangle (2.5,4.5);
        \draw[color=red,pattern=dots, pattern color=red] (2,5) rectangle (2.5,5.5);

        \draw[blue, thick] (2.25,1.25) to [out = 90, in = 0] (1.25,2.25);
        \draw[blue, thick] (1.25,2.25) to [out = 90, in = 0] (.25,3.75);
        \draw[blue, thick] (.25,3.75) to [out = 90, in = 180] (3.75,4.75);
        \draw[blue, thick] (3.75,4.75) to [out = 90, in = 0] (.75,6.25);
        \draw[blue, thick, ->] (.75,6.25) to [out = 90, in = 0] (.25,6.75);
        
        \draw[red, very thick, dotted] (2.25,1.75) to [out = 90, in = 0] (.75,3.25);
        \draw[red, very thick, dotted] (.75,3.25) to [out = 90, in = 180] (2.25,4.25);
        \draw[red, very thick, dotted] (2.25,4.25) to [out = 200, in = 180] (2.25,5.25);
        \draw[red, ->, very thick, dotted] (2.25,5.25) to [out = 90, in = 0] (.25,7.25);
    \end{tikzpicture}
\]

The interlacing property for \(m\)-paths is not so strict as the above example suggests. We require that an \(m\)-path interlaces orbits only within blocks while multiple orbits are present. This is illustrated further in the example below.
\[
    \begin{tikzpicture}[scale=.75, baseline={([yshift=-.5ex]current bounding box.center)}]
        \draw (0,0) rectangle (.5,1.5);
        \draw (0,1.5) rectangle (2.5,3);
        \draw (0,3) rectangle (4,6);
        \draw (0,6) rectangle (6,7.5);
        \draw (0,7.5) rectangle (8,8.5);
        \draw (0,8.5) rectangle (.5, 9);
        \draw (0,9) rectangle (.5,9.5);
        
        \filldraw[fill=blue!20] (0,0) rectangle (.5,.5);
        \filldraw[fill=blue!20] (1.5,1.5) rectangle (2,2);
        \filldraw[fill=blue!20] (1,3.5) rectangle (1.5,4);
        \filldraw[fill=blue!20] (2,4.5) rectangle (2.5,5);
        \filldraw[fill=blue!20] (.5,5) rectangle (1,5.5);
        \filldraw[fill=blue!20] (.5,6) rectangle (1,6.5);
        \filldraw[fill=blue!20] (2,6.5) rectangle (2.5,7);
        \filldraw[fill=blue!20] (4.5,7) rectangle (5,7.5);
        
        \draw[color=red,pattern=dots, pattern color=red] (0,.5) rectangle (.5,1);
        \draw[color=red,pattern=dots, pattern color=red] (.5,2) rectangle (1,2.5);
        \draw[color=red,pattern=dots, pattern color=red] (0,2.5) rectangle (.5,3);
        \draw[color=red,pattern=dots, pattern color=red] (2.5,3) rectangle (3,3.5);
        \draw[color=red,pattern=dots, pattern color=red] (1.5,4) rectangle (2,4.5);
        \draw[color=red,pattern=dots, pattern color=red] (2.5,7.5) rectangle (3,8);
        
        \draw[blue, thick] (.25,.25) to [out = 90, in = 180] (1.75,1.75);
        \draw[blue, thick] (1.75,1.75) to [out = 90, in = 0] (1.25,3.75);
        \draw[blue, thick] (1.25,3.75) to [out = 90, in = 180] (2.25,4.75);
        \draw[blue, thick] (2.25,4.75) to [out = 90, in = 0] (.75,5.25);
        \draw[blue, thick] (.75,5.25) to [out = 200, in = 180] (.75,6.25);
        \draw[blue, thick] (.75,6.25) to [out = 90, in = 180] (2.25,6.75);
        \draw[blue, thick] (2.25,6.75) to [out = 90, in = 180] (4.75,7.25);
        \draw[blue, thick, ->] (4.75,7.25) to [out = 90, in = 0] (.25,8.75);
        
        \draw[red, very thick, dotted] (.25,.75) to [out = 90, in = 180] (.75,2.25);
        \draw[red, very thick, dotted] (.75,2.25) to [out = 90, in = 0] (.25,2.75);
        \draw[red, very thick, dotted] (.25,2.75) to [out = 90, in = 180] (2.75,3.25);
        \draw[red, very thick, dotted] (2.75,3.25) to [out = 90, in = 0] (1.75,4.25);
        \draw[red, very thick, dotted] (1.75,4.25) to [out = 90, in = 180] (2.75,7.75);
        \draw[red, very thick, dotted, ->] (2.75,7.75) to [out = 90, in = 0] (.25,9.25);
    \end{tikzpicture}
\]

Formally, we have the following definition.

\begin{definition*}[m-path]
Let
\[
    X = \{x_1 = [b_1](1,w_{b_1}), \ldots, x_m = [b_m](i_m,w_{b_m})\}
\]
be a removal set for \(T_0\) and
\[
    Y = \{ y_1:= [N+1](1,1), \ldots, y_m:= [N+1](m,1) \},
\]
where \(Y\) is viewed as block \(N + 1\) attached to the top of \(T_0\). An
\[
    m\text{-path } P \text{ removing } X
\]
is a map of boxes%
\[
    P:\lambda \cup Y \to \lambda \cup Y
\]
along with an evacuation route \(R = R^P\) such that the following hold.
	\begin{itemize}
    \item \(R\) starts at \([b_1]\). Note that \(R\) can intersect \(X\), though this is not a requirement.
    \item \(P\) is geometrically increasing on rows, with \(P\) strictly increasing on \(R\). %
    That is, for all boxes \(x \in \lambda \cup Y\), \(x \le P(x)\).
    \item If \(R_i\) is the orbit of \(x_i\) under \(P^\bN\), then \(y_i \in R_i\) and \(R \cup X \cup Y = \bigsqcup_{i=1}^m R_i\).
    \item If there are \(k\) distinct orbits in a block, then the first \(k\) rows of that block must be in different orbits. %
    i.e., if \(R^P_{i_1}, \ldots, R^P_{i_k}\) intersect some block \([b]\), then for \(j = 1, \ldots k\), up to relabeling, \(R^P_{i_j} \cap [b](j) \ne \emptyset\).
    \item \(P\) preserves row order in \(R\) within blocks, and so interlaces orbits. %
    That is, if \([b](i,j), \ [b](k,l) \in R\) with \(i < k\) and \(P([b](i,j)), P([b](k,l)) \in [b]\), then \(P([b](i,j)) < P([b](k,l))\).
    \item \(P\) fixes those boxes not in \(R\) or \(X\), i.e. \(P = \text{id}_{\lambda \cup Y}\) except on \(R \cup X\), and \(P(R) = R \setminus X \cup Y\).
    \end{itemize}
\end{definition*}

\subsection{} \label{Constructing the Pieri Inclusion for Removing Many Boxes subsection: Defining H(P)}

For an \(m\)-path \(P\) with evacuation route \(R^P\), let \(h^P\), \((-1)^P\), \(h_b^P\), and \(H(b)\) be defined as in Section \ref{Constructing the Pieri Inclusion for Removing One Box subsection: Defining h(b) and H(B)}. 
For \(b = b_1 + 1, \ldots, N\), let \(H_b(P) = 1\) if \(R^P \cap [b] = \emptyset\) and let \(H_{b_1}(P) = 1\). If \(b \ge b_1 + 1\) and \(\left\vert R^P \cap [b] \right\vert = k_b \ne 0\), then let
\[
    H_b(P) = \prod_{i = 1}^{k_b} \left(H(b) - (m - i)\right).
\]
Define
\[
    H(P) = \prod_{b=b_1}^n H_b(P).
\]

\subsection{} \label{Constructing the Pieri Inclusion for Removing Many Boxes subsection: Defining Phi}

For \(T \in \mF_{\lambda,n}\), denote by \(\alpha_i^P\) the entry in the box \(P^{-1}(y_i) \in T\) and extend \(P\) to \(T\) by acting on the entries, with the image %
\[
    P(T) = Y_P \otimes T_P \in F_m \otimes \mF_{\lambda %
    \setminus X,n} = \bigwedge^m V \otimes \mF_{\lambda \setminus X,n},
\]
where
\ytableausetup{nosmalltableaux}
\[
    Y_P = E_X \ytableaushort{{\alpha_m^P}, \vdots, {\alpha_1^P}}
\]
which is standard form notation is \(e_{\alpha_1^P} \wedge \cdots \wedge e_{\alpha_m^P} \in \bigwedge^m V\), and \(T_P \in \mF_{\lambda \setminus X,n}\) is defined by \(\left( T_P \right)_{[b](i,j)} = T_{P^{-1}([b](i,j))}\). %
As before, we omit \(E_X\) and just write 
\[
    \ytableaushort{{\alpha_m^P}, \vdots, {\alpha_1^P}} \quad \text{ in place of } \quad  E_X \ytableaushort{{\alpha_m^P}, \vdots, {\alpha_1^P}}
\]
in the image of \(P(T)\).\\

\ytableausetup{smalltableaux}

\begin{definition*}
    The map \(\ds \Phi_m : \mF_{\lambda, n} \to F_m \otimes \mF_{\lambda \setminus X, n}\) is given by
    \[
        \Phi_m(T) = \sum_{P} \frac{(-1)^P}{H(P)} P(T),
    \]
    where the sum is over all \(m\)-paths \(P\) removing \(X\).
\end{definition*}

\subsection{} \label{Constructing the Pieri Inclusion for Removing Many Boxes subsection: Two Path Example}

We now compute an example of the Pieri inclusion when \(m = 2\).
Let \(n = 6\), \(\lambda = (2,2,1,1,1,1)\), and \(\mu = (2,2,1,1)\). 
Then the Schur--Weyl module \(\bS_{\lambda}\) appears as a summand in the decomposition of \(\bS_{(1,1)} \otimes \bS_{\mu}\),
\[
    \begin{ytableau}
        *(black!25)\\
        *(black!25)
    \end{ytableau}
    \otimes
    \begin{ytableau}
        ~ & ~\\
        ~ & ~\\
        ~ \\
        ~
    \end{ytableau}
    =
    \begin{ytableau}
        ~ & ~ & *(black!25)\\
        ~ & ~ & *(black!25)\\
        ~ \\
        ~
    \end{ytableau}
    \oplus
    \begin{ytableau}
        ~ & ~ & *(black!25)\\
        ~ & ~\\
        ~ & *(black!25)\\
        ~
    \end{ytableau}
    \oplus
    \begin{ytableau}
        ~ & ~ & *(black!25)\\
        ~ & ~\\
        ~ \\
        ~ \\
        *(black!25)
    \end{ytableau}
    \oplus
    \begin{ytableau}
        ~ & ~ \\
        ~ & ~\\
        ~ & *(black!25)\\
        ~ \\
        *(black!25)
    \end{ytableau}
    \oplus
    \begin{ytableau}
        ~ & ~ \\
        ~ & ~\\
        ~ \\
        ~ \\
        *(black!25)\\
        *(black!25)
    \end{ytableau}.
\]
    
Consider the Pieri inclusion
\[
    \bS_{(2,2,1,1,1,1)}  \overset{\Phi_2}{\longrightarrow} \bS_{(1,1)} \otimes \bS_{(2,2,1,1)}.
\]
We will show the image of the highest weight vector under this map,
\[
    \Phi_2\left( \begin{ytableau}
        1 & 1\\
        2 & 2\\
        3 \\
        4 \\
        5 \\
        6
    \end{ytableau} \right)
    =
    \sum_{P} \frac{(-1)^P}{H(P)} P\left( \begin{ytableau}
        1 & 1\\
        2 & 2\\
        3 \\
        4 \\
        5 \\
        6
    \end{ytableau} \right),
\]
where the sum is over all \(m\)-paths \(P\) on \(\lambda\) removing \(X = \{x_1 = [1](1,1), x_2 = [1](2,1)\).
Below we illustrate all such paths with arrows, where we shade the boxes in the evacuation route, distinguishing the orbits of \(x_1\) and \(x_2\). 
For paths hitting rows \([2](1)\) and \([2](2)\), we only show the path that hits the first column of both rows, as the paths that hit the second column in either row will give the same result. 
As in the \(1\)-box removal example, we give the images up to row permutations and we star the paths whose images require straightening.
    
\begingroup
\setlength{\tabcolsep}{12pt}
\renewcommand{\arraystretch}{5}
\begin{longtable}[c]{c c c}
        \begin{tikzpicture}[baseline=(O.base)]
            \node (O) at (0,1.05) {~};
            \filldraw[fill=blue!25] (0,0) rectangle (.3,.3);
            \node at (.15,.15) {\scriptsize \(6\)};
            \filldraw[fill=red!10] (0,.3) rectangle (.3,.6);
            \node at (.15,.45) {\scriptsize \(5\)};
            \draw (0,.6) rectangle (.3,.9);
            \node at (.15,.75) {\scriptsize \(4\)};
            \draw (0,.9) rectangle (.3,1.2);
            \node at (.15,1.05) {\scriptsize \(3\)};
            \draw (0,1.2) rectangle (.3,1.5);
            \node at (.15,1.35) {\scriptsize \(2\)};
            \draw (.3,1.2) rectangle (.6,1.5);
            \node at (.45,1.35) {\scriptsize \(2\)};
            \draw (0,1.5) rectangle (.3,1.8);
            \node at (.15,1.65) {\scriptsize \(1\)};
            \draw (.3,1.5) rectangle (.6,1.8);
            \node at (.45,1.65) {\scriptsize \(1\)};
            \draw (0,1.8) rectangle (.3,2.1);
            \draw (0,2.1) rectangle (.3,2.4);
            \draw[blue, thick, ->] (.1,.15) to [out = 180, in = 180] (.15,1.95);
            \draw[red, dashed, very thick, ->] (.1,.45) to [out = 180, in = 180] (.15,2.25);
        \end{tikzpicture}
        \(\leadsto\)
        \begin{ytableau}
            5\\
            6
        \end{ytableau}
        \(\otimes\)
        \begin{ytableau}
            1 & 1\\
            2 & 2\\
            3 \\
            4
        \end{ytableau},
        &
        \begin{tikzpicture}[baseline=(O.base)]
            \node (O) at (0,1.05) {~};
            \filldraw[fill=blue!25] (0,0) rectangle (.3,.3);
            \node at (.15,.15) {\scriptsize \(6\)};
            \filldraw[fill=red!10] (0,.3) rectangle (.3,.6);
            \node at (.15,.45) {\scriptsize \(5\)};
            \filldraw[fill=red!10] (0,.6) rectangle (.3,.9);
            \node at (.15,.75) {\scriptsize \(4\)};
            \draw (0,.9) rectangle (.3,1.2);
            \node at (.15,1.05) {\scriptsize \(3\)};
            \draw (0,1.2) rectangle (.3,1.5);
            \node at (.15,1.35) {\scriptsize \(2\)};
            \draw (.3,1.2) rectangle (.6,1.5);
            \node at (.45,1.35) {\scriptsize \(2\)};
            \draw (0,1.5) rectangle (.3,1.8);
            \node at (.15,1.65) {\scriptsize \(1\)};
            \draw (.3,1.5) rectangle (.6,1.8);
            \node at (.45,1.65) {\scriptsize \(1\)};
            \draw (0,1.8) rectangle (.3,2.1);
            \draw (0,2.1) rectangle (.3,2.4);
            \draw[blue, thick, ->] (.1,.15) to [out = 180, in = 180] (.15,1.95);
            \draw[red, dashed, very thick] (.1,.45) to [out = 180, in = 180, looseness = 2] (.1,.75);
            \draw[red, dashed, very thick, ->] (.1,.75) to [out = 180, in = 180] (.15,2.25);
        \end{tikzpicture}
        \(\leadsto\)
        \(-\)
        \begin{ytableau}
            4\\
            6
        \end{ytableau}
        \(\otimes\)
        \begin{ytableau}
            1 & 1\\
            2 & 2\\
            3 \\
            5
        \end{ytableau}, 
        & 
        \begin{tikzpicture}[baseline=(O.base)]
            \node (O) at (0,1.05) {~};
            \filldraw[fill=blue!25] (0,0) rectangle (.3,.3);
            \node at (.15,.15) {\scriptsize \(6\)};
            \filldraw[fill=red!10] (0,.3) rectangle (.3,.6);
            \node at (.15,.45) {\scriptsize \(5\)};
            \filldraw[fill=red!10] (0,.6) rectangle (.3,.9);
            \node at (.15,.75) {\scriptsize \(4\)};
            \filldraw[fill=red!10] (0,.9) rectangle (.3,1.2);
            \node at (.15,1.05) {\scriptsize \(3\)};
            \draw (0,1.2) rectangle (.3,1.5);
            \node at (.15,1.35) {\scriptsize \(2\)};
            \draw (.3,1.2) rectangle (.6,1.5);
            \node at (.45,1.35) {\scriptsize \(2\)};
            \draw (0,1.5) rectangle (.3,1.8);
            \node at (.15,1.65) {\scriptsize \(1\)};
            \draw (.3,1.5) rectangle (.6,1.8);
            \node at (.45,1.65) {\scriptsize \(1\)};
            \draw (0,1.8) rectangle (.3,2.1);
            \draw (0,2.1) rectangle (.3,2.4);
            \draw[blue, thick, ->] (.1,.15) to [out = 180, in = 180] (.15,1.95);
            \draw[red, dashed, very thick] (.1,.45) to [out = 180, in = 180, looseness = 2] (.1,.75);
            \draw[red, dashed, very thick] (.1,.75) to [out = 180, in = 180, looseness = 2] (.1,1.05);
            \draw[red, dashed, very thick, ->] (.1,1.05) to [out = 180, in = 180] (.15,2.25);
        \end{tikzpicture}
        \(\leadsto\)
        \begin{ytableau}
            3\\
            6
        \end{ytableau}
        \(\otimes\)
        \begin{ytableau}
            1 & 1\\
            2 & 2\\
            4 \\
            5
        \end{ytableau},
        \\
\end{longtable}
\endgroup
    
If
\[
    A_1 = \{[2](1,1), [2],(1,2), [2](2,2)\} = 
    \begin{ytableau}            
        ~ & *(black!25)\\
        *(black!25) & *(black!25)\\
        ~ \\
        ~
    \end{ytableau}
\]
then we have, mod \(\mR_{(2,2,1,1), 6}\),
\[
    \begin{ytableau}
        1 & 5\\
        2 & 4\\
        3 \\
        6
    \end{ytableau}
    =
    \frac{1}{2} G_{A_1}
    \left( \begin{ytableau}
        1 & 5\\
        2 & 4\\
        3 \\
        6
    \end{ytableau} \right)
    -
    \begin{ytableau}
        1 & 2\\
        4 & 5 \\
        3 \\
        6
    \end{ytableau}
    -
    \begin{ytableau}
        1 & 4\\
        2 & 5 \\
        3 \\
        6
    \end{ytableau}.
\]
Then if
\[
    A_2 = \{[2](1,1), [2],(1,2), [1](2,1)\} = 
    \begin{ytableau}            
        ~ & ~\\
        *(black!25) & *(black!25)\\
        *(black!25) \\
        ~
    \end{ytableau}
\]
we have, mod \(\mR_{(2,2,1,1), 6}\),
\[
    \begin{ytableau}
        1 & 2\\
        4 & 5\\
        3 \\
        6
    \end{ytableau}
    =
    \frac{1}{2} G_{A_2}
    \left( \begin{ytableau}
        1 & 2\\
        4 & 5\\
        3 \\
        6
    \end{ytableau} \right)
    -
    \begin{ytableau}
        1 & 2\\
        3 & 5 \\
        4 \\
        6
    \end{ytableau}
    -
    \begin{ytableau}
        1 & 2\\
        3 & 4 \\
        5 \\
        6
    \end{ytableau}.
\]
Thus, mod \(\mR_{(2,2,1,1), 6}\),
\[
    \begin{ytableau}
        1 & 5\\
        2 & 4\\
        3 \\
        6
    \end{ytableau}
    =
    \begin{ytableau}
        1 & 2\\
        3 & 5 \\
        4 \\
        6
    \end{ytableau}
    +
    \begin{ytableau}
        1 & 2\\
        3 & 4 \\
        5 \\
        6
    \end{ytableau}
    -
    \begin{ytableau}
        1 & 4\\
        2 & 5 \\
        3 \\
        6
    \end{ytableau}.
\]
    
Similarly, via straightening we have, mod \(\mR_{(2,2,1,1), 6}\),
\[
    \begin{ytableau}
        1 & 5\\
        2 & 3\\
        4 \\
        6
    \end{ytableau}
    =
    -
    \begin{ytableau}
        1 & 2\\
        3 & 5 \\
        4 \\
        6
    \end{ytableau}
    -
    \begin{ytableau}
        1 & 3\\
        2 & 5 \\
        4 \\
        6
    \end{ytableau},
\]
\[
    \begin{ytableau}
        1 & 6\\
        2 & 5\\
        3 \\
        4
    \end{ytableau}
    =
    -
    \begin{ytableau}
        1 & 2\\
        3 & 6 \\
        4 \\
        5
    \end{ytableau}
    -
    \begin{ytableau}
        1 & 2\\
        3 & 5 \\
        4 \\
        6
    \end{ytableau}
    -
    \begin{ytableau}
        1 & 5\\
        2 & 6 \\
        3 \\
        4
    \end{ytableau},
\]
\[
    \begin{ytableau}
        1 & 6\\
        2 & 4\\
        3 \\
        5
    \end{ytableau}
    =
    \begin{ytableau}
        1 & 2\\
        3 & 6 \\
        4 \\
        5
    \end{ytableau}
    -
    \begin{ytableau}
        1 & 2\\
        3 & 4\\
        5 \\
        6
    \end{ytableau}
    -
    \begin{ytableau}
        1 & 4\\
        2 & 6 \\
        3 \\
        5
    \end{ytableau},
\]
\[
    \begin{ytableau}
        1 & 6\\
    2 & 3\\
        4 \\
        5
    \end{ytableau}
    =
    -
    \begin{ytableau}
        1 & 2\\
        3 & 6\\
        4 \\
        6
    \end{ytableau}
    -
    \begin{ytableau}
        1 & 3\\
        2 & 6 \\
        4 \\
        5
    \end{ytableau},
\]
and
\[
    \begin{ytableau}
        1 & 4\\
        2 & 3\\
        5 \\
        6
    \end{ytableau}
    =
    -
    \begin{ytableau}
        1 & 2\\
        3 & 4\\
        5 \\
        6
    \end{ytableau}
    -
    \begin{ytableau}
        1 & 3\\
        2 & 4\\
        5 \\
        6
    \end{ytableau}.
\]
Recall that for all \(2\)-paths \(P\), \(Y_P \in \bigwedge^2 V\), and so
\[
    \begin{ytableau}
        \alpha\\
        \beta
    \end{ytableau}
    =
    -
    \begin{ytableau}
        \beta\\
        \alpha
    \end{ytableau}.
\]
Thus,
\begingroup
    \addtolength{\jot}{.5em}
    \begin{align*}
        \Phi_2\left( \begin{ytableau}
            1 & 1\\
            2 & 2\\
            3 \\
            4 \\
            5 \\
            6
        \end{ytableau} \right)
        = ~ &
        \begin{ytableau}
            5\\
            6
        \end{ytableau}
        \otimes
        \begin{ytableau}
            1 & 1\\
            2 & 2\\
            3 \\
            4
        \end{ytableau}
        -
        \begin{ytableau}
            4\\
            6
        \end{ytableau}
        \otimes
        \begin{ytableau}
            1 & 1\\
            2 & 2\\
            3 \\
            5
        \end{ytableau}
        +
        \begin{ytableau}
            3\\
            6
        \end{ytableau}
        \otimes
        \begin{ytableau}
            1 & 1\\
            2 & 2\\
            4 \\
            5
        \end{ytableau}
        - \frac{1}{2} ~
        \begin{ytableau}
            2\\
            6
        \end{ytableau}
        \otimes
        \begin{ytableau}
            1 & 2\\
            2 & 5\\
            3 \\
            4
        \end{ytableau}
        + \frac{1}{2} ~
        \begin{ytableau}
            2\\
            6
        \end{ytableau}
        \otimes
        \begin{ytableau}
            1 & 1\\
            2 & 4\\
            3 \\
            5
        \end{ytableau}
        - \frac{1}{2} ~
        \begin{ytableau}
            2\\
            6
        \end{ytableau}
        \otimes
        \begin{ytableau}
            1 & 1\\
            2 & 3\\
            4 \\
            5
        \end{ytableau}\\[-.5em]
        + &
        \begin{ytableau}
            1\\
            6
        \end{ytableau}
        \otimes
        \begin{ytableau}
            1 & 2\\
            2 & 5\\
            3 \\
            4
        \end{ytableau}
        -
        \begin{ytableau}
            1\\
            6
        \end{ytableau}
        \otimes
        \begin{ytableau}
            1 & 2\\
            2 & 4\\
            3 \\
            5
        \end{ytableau}
        +
        \begin{ytableau}
            1\\
            6
        \end{ytableau}
        \otimes
        \begin{ytableau}
            1 & 2\\
            2 & 3\\
            4 \\
            5
        \end{ytableau}
        +
        \begin{ytableau}
            4\\
            5
        \end{ytableau}
        \otimes
        \begin{ytableau}
            1 & 1\\
            2 & 2\\
            3 \\
            6
        \end{ytableau}
        -
        \begin{ytableau}
            3\\
            5
        \end{ytableau}
        \otimes
        \begin{ytableau}
            1 & 1\\
            2 & 2\\
            4 \\
            6
        \end{ytableau}
        + \frac{1}{2} ~
        \begin{ytableau}
            2\\
            5
        \end{ytableau}
        \otimes
        \begin{ytableau}
            1 & 1\\
            2 & 6\\
            3 \\
            4
        \end{ytableau}\\
        -& \frac{1}{2} ~
        \begin{ytableau}
            2\\
            5
        \end{ytableau}
        \otimes
        \begin{ytableau}
            1 & 1\\
            2 & 4\\
            3 \\
            6
        \end{ytableau}
        + \frac{1}{2} ~
        \begin{ytableau}
            2\\
            5
        \end{ytableau}
        \otimes
        \begin{ytableau}
            1 & 1\\
            2 & 3\\
            4 \\
            6
        \end{ytableau}
        -
        \begin{ytableau}
            1\\
            5
        \end{ytableau}
        \otimes
        \begin{ytableau}
            1 & 2\\
            2 & 6\\
            3 \\
            4
        \end{ytableau}
        +
        \begin{ytableau}
            1\\
            5
        \end{ytableau}
        \otimes
        \begin{ytableau}
            1 & 2\\
            2 & 4\\
            3 \\
            6
        \end{ytableau}
        -
        \begin{ytableau}
            1\\
            5
        \end{ytableau}
        \otimes
        \begin{ytableau}
            1 & 2\\
            2 & 3\\
            4 \\
            6
        \end{ytableau}
        +
        \begin{ytableau}
            3\\
            4
        \end{ytableau}
        \otimes
        \begin{ytableau}
            1 & 1\\
            2 & 2\\
            5 \\
            6
        \end{ytableau}\\
        +& \frac{1}{2} ~
        \begin{ytableau}
            2\\
            4
        \end{ytableau}
        \otimes
        \begin{ytableau}
            1 & 1\\
            2 & 5\\
            3 \\
            6
        \end{ytableau}
        - \frac{1}{2} ~
        \begin{ytableau}
            2\\
            4
        \end{ytableau}
        \otimes
        \begin{ytableau}
            1 & 1\\
            2 & 3\\
            5 \\
            6
        \end{ytableau}
        - \frac{1}{2} ~
        \begin{ytableau}
            2\\
            4
        \end{ytableau}
        \otimes
        \begin{ytableau}
            1 & 1\\
            2 & 6\\
            3 \\
            5
        \end{ytableau}
        -
        \begin{ytableau}
            1\\
            4
        \end{ytableau}
        \otimes
        \begin{ytableau}
            1 & 2\\
            2 & 5\\
            3 \\
            6
        \end{ytableau}
        +
        \begin{ytableau}
            1\\
            4
        \end{ytableau}
        \otimes
        \begin{ytableau}
            1 & 2\\
            2 & 3\\
            5 \\
            6
        \end{ytableau}
        -
        \begin{ytableau}
            1\\
            4
        \end{ytableau}
        \otimes
        \begin{ytableau}
            1 & 2\\
            2 & 6\\
            3 \\
            5
        \end{ytableau}\\
        +& \frac{1}{2} ~
        \begin{ytableau}
            2\\
            3
        \end{ytableau}
        \otimes
        \begin{ytableau}
            1 & 1\\
            2 & 4\\
            5 \\
            6
        \end{ytableau}
        + \frac{1}{2} ~
        \begin{ytableau}
            2\\
            3
        \end{ytableau}
        \otimes
        \begin{ytableau}
            1 & 1\\
            2 & 6\\
            4 \\
            5
        \end{ytableau}
        - \frac{1}{2} ~
        \begin{ytableau}
            2\\
            3
        \end{ytableau}
        \otimes
        \begin{ytableau}
            1 & 1\\
            2 & 5\\
            4 \\
            6
        \end{ytableau}
        -
        \begin{ytableau}
            1\\
            3
        \end{ytableau}
        \otimes
        \begin{ytableau}
            1 & 2\\
            2 & 4\\
            5 \\
            6
        \end{ytableau}
        -
        \begin{ytableau}
            1\\
            3
        \end{ytableau}
        \otimes
        \begin{ytableau}
            1 & 2\\
            2 & 6\\
            4 \\
            5
        \end{ytableau}
        + \frac{1}{2} ~
        \begin{ytableau}
            1\\
            3
        \end{ytableau}
        \otimes
        \begin{ytableau}
            1 & 2\\
            2 & 5\\
            4 \\
            6
        \end{ytableau}\\
        +& \frac{2}{5} ~
        \begin{ytableau}
            1\\
            2
        \end{ytableau}
        \otimes
        \begin{ytableau}
            1 & 5\\
            2 & 6\\
            3 \\
            4
        \end{ytableau}
        - \frac{2}{5} ~
        \begin{ytableau}
            1\\
            2
        \end{ytableau}
        \otimes
        \begin{ytableau}
            1 & 4\\
            2 & 6\\
            3 \\
            5
        \end{ytableau}
        + \frac{2}{5} ~
        \begin{ytableau}
            1\\
            2
        \end{ytableau}
        \otimes
        \begin{ytableau}
            1 & 3\\
            2 & 6\\
            4 \\
            5
        \end{ytableau}
        - \frac{1}{5} ~
        \begin{ytableau}
            1\\
            2
        \end{ytableau}
        \otimes
        \begin{ytableau}
            1 & 2\\
            3 & 5\\
            4 \\
            6
        \end{ytableau}
        - \frac{1}{5} ~
        \begin{ytableau}
            1\\
            2
        \end{ytableau}
        \otimes
        \begin{ytableau}
            1 & 2\\
            3 & 4\\
            5 \\
            6
        \end{ytableau}\\
        +& \frac{2}{5} ~
        \begin{ytableau}
            1\\
            2
        \end{ytableau}
        \otimes
        \begin{ytableau}
            1 & 4\\
            2 & 5\\
            3 \\
            6
        \end{ytableau}
        + \frac{2}{5} ~
        \begin{ytableau}
            1\\
            2
        \end{ytableau}
        \otimes
        \begin{ytableau}
            1 & 3\\
            2 & 4\\
            5 \\
            6
        \end{ytableau}
        - \frac{2}{5} ~
        \begin{ytableau}
            1\\
            2
        \end{ytableau}
        \otimes
        \begin{ytableau}
            1 & 3\\
            2 & 5\\
            4 \\
            6
        \end{ytableau}
        + \frac{2}{5} ~
        \begin{ytableau}
            1\\
            2
        \end{ytableau}
        \otimes
        \begin{ytableau}
            1 & 2\\
            3 & 6\\
            4 \\
            5
        \end{ytableau}.
    \end{align*}
\endgroup

\section{Generating Garnir Relations and Tools for Collapsing Sums} \label{section: Generating Garnir Relations and Tools for Collapsing Sums}

\subsection{} \label{Generating Garnir Relations and Tools for Collapsing Sums subsection: One Garnir Theorem}

\ytableausetup{nosmalltableaux}

We will show that all Garnirs are generated by Garnirs of minimal size, i.e. those \(G_{A}\) with \(\vert A \vert = w_A + 1\) (where \(w_A\) is as in Section \ref{Constructing Schur--Weyl Modules subsection: Garnirs}). 
We then show that all Garnirs of minimal size are themselves generate by Garnirs over hooks, which are those \(G_A\) where \(A\) is of minimal size and consists of exactly a complete row and one other box. 
We start by formalizing the idea of a hook.

\begin{definition*}[Hook]
We say that \(A \subset T_0\) is a hook if for some row \([b](r)\),
\[
    A = [b](r) \cup \{a_0\}
\]
where 
\[
    a_0 =   \begin{cases} [b](r-1,1) & \text{ if } r \ne 1\\
                        [b-1](h_{b-1},1) & \text{ if } r = 1
            \end{cases}.
\]
That is,
\[
    A = \overbrace{\begin{ytableau} ~ & \cdots & ~ \\ ~ \end{ytableau}}^{w_b}
\]
\end{definition*}

\begin{theorem*} 
Let \(T \in \mF_{\lambda,n}\) and \(A \subset T_0\) such that \(\vert A \vert > w_A\). Then 
\[
    G_A (T)  \in \langle G_{A^\prime} (T^\prime)  :  T^\prime \in \mF_{\lambda, n}, ~ A^\prime \text{ is a hook}\rangle.
\]
\end{theorem*}

\subsection{} \label{Generating Garnir Relations and Tools for Collapsing Sums subsection: Minimal Size Lemma}

To prove Theorem \ref{Generating Garnir Relations and Tools for Collapsing Sums subsection: One Garnir Theorem}, we first show that \(G_A (T)\) is generated by Garnirs of minimal size for any \(T \in \mF_{\lambda,n}\) and any \(A \subset T_0\) with \(\vert A \vert > w_A\). We will also show that if \(\vert A \vert > w_A + 1\), then \(G_A (T)\) is generated by Garnirs over \(A \setminus \{y\}\) for any \(y \in A\).

\begin{lemma*} 
Let \(T \in \mF_{\lambda, n}\). If \(A \subset T_0\) with \(\vert A \vert > w_A\), then for any \(x \in T_0\) such that \(\vert A \cup \{x\} \vert > w_{A \cup \{x_0\}}\),
\[
    G_{A\cup \{x\}} (T) \in \langle G_A (T^\prime)  :  T^\prime \in \mF_{\lambda,n} \rangle.
\]
\end{lemma*}

\begin{proof}
Let \(A \subset T_0\) with \(\vert A \vert > w_A\) and \(x \in T_0 \setminus A\). For all \(y \in A\cup \{x\}\), let \(\tau_{x,y}\) be the permutation that switches \(x\) and \(y\) and fixes the rest of \(A \cup \{x\}\). Then for any \(\sigma \in \fS_{A \cup \{x\}}\),
\[
    \sigma (y) = x %
	    \iff \left(\sigma \tau_{x,y}\right) (x) = x
        \iff \sigma \tau_{x,y} \in \fS_{A}.
\]
Then,
\begin{align*}
    G_{A \cup \{x\}} (T) %
	    & = \sum_{\sigma \in \fS_{A \cup \{x\}}} \sigma T\\
	    & = \sum_{y \in A\cup \{x\}} \sum_{\substack{\sigma \in \fS_{A\cup \{x\}}\\ \textnormal{s.t. } %
	    \sigma (y) = x}} \sigma T\\
	    & = \sum_{y \in A\cup\{x\}} \sum_{\substack{\sigma \in \fS_{A\cup \{x\}}\\ \textnormal{s.t. } %
	    \sigma (y) = x}} \left(\sigma \tau_{x,y}\right) \left(\tau_{x,y} T \right)\\
	    & = \sum_{y \in A\cup\{x\}} \sum_{\tilde{\sigma}\in \fS_A} \tilde{\sigma}\left(\tau_{x,y} T \right)\\
	    & = \sum_{y \in A\cup\{x\}} G_{A} (\tau_{x,y} T) \in \langle G_A (T^\prime) :  T^\prime %
	    \in \mF_{\lambda,n} \rangle
\end{align*}
\end{proof}

\subsection{} \label{Generating Garnir Relations and Tools for Collapsing Sums subsection: Hooks with a Gap Lemma}

We now show that all Garnirs of minimal size are generated by Garnirs over a set consisting of a full row and a box below that row.

\begin{lemma*} 
Let \(T \in \mF_{\lambda,n}\). If \(A \subset T_0\) of minimal size, then
\[
    G_{A} (T) \in \langle G_{A' \cup \{b_0\}} (T^\prime)  :  T^\prime \in \mF_{\lambda,n}, A^\prime = [b](r) \text{ for some } [b](r) \subset T_0, b_0 < [b](r) \rangle.
\]
\end{lemma*}

\begin{proof} Let \(T \in \mF_{\lambda,n}\) and \(A \subset T_0\) such that \(A \subset [b](r)\). Assume, without loss of generality, that \(r \ne 1\) (else, replace \([b](r-1)\) in the following argument with \([b-1](h_{b - 1})\)). Let \(B \subset T_0\) such that \( B < [b](r)\), \(B \not \subset [b](r-1)\), and \(\vert A \cup B \vert = w_b + 1\). Assume \(A \ne [b](r)\), i.e. \(\vert B \vert \ne 1\). We will show that
\[
    G_{A \cup B} (T) \in \langle G_{A^\prime \cup B^\prime} (T^\prime) : T^\prime \in \mF_{\lambda, n}, ~ A^\prime = [b](r), ~ \vert B^\prime \vert = 1, B < A\rangle.
\]
Pick \(x_0 \in [b](r) \setminus A\) and \(b_0 \in B\).

\begin{center}
\begin{tikzpicture}
    \draw[dashed] (0,0) -- (0,.8);
    \draw[dashed] (1,0) -- (1,.8);
    \draw (0,.8) rectangle (2,2);
    \draw (0,2) rectangle (4,4);
    \draw[dashed] (0,4) -- (0,5);
    \draw[dashed] (4,4) -- (4,5);
    
    \draw (0,3.2) -- (4,3.2);
    \draw (0,3.6) -- (4,3.6);
    \node at (5,3.4) {\tiny{\(\leftarrow\) row \([b](r)\)}};
    
    \draw (2.4,3.2) rectangle (2.8,3.6);
    \node at (2.6,3.4) {\tiny{\(x_0\)}};
    
    \draw[color=red,pattern=horizontal lines, pattern color=red] (0,1.6) rectangle (.4,2);
    \draw[color=red,pattern=horizontal lines, pattern color=red] (1.2,1.2) rectangle (1.6,1.6);
    \draw[color=red,pattern=horizontal lines, pattern color=red] (.8,2) rectangle (1.2,2.4);
    \draw[color=red,pattern=horizontal lines, pattern color=red] (2,2) rectangle (2.4,2.4);
    \draw[color=red,pattern=horizontal lines, pattern color=red] (3.2,2.4) rectangle (3.6,2.8);
    \draw[color=red,pattern=horizontal lines, pattern color=red] (0,2.8) rectangle (0.4,3.2);
    
    \filldraw[color=blue,fill=blue!30] (0,3.2) rectangle (.4,3.6);
    \filldraw[color=blue,fill=blue!30] (.4,3.2) rectangle (.8,3.6);
    \filldraw[color=blue,fill=blue!30] (1.2,3.2) rectangle (1.6,3.6);
    \filldraw[color=blue,fill=blue!30] (1.6,3.2) rectangle (2,3.6);
    \filldraw[color=blue,fill=blue!30] (3.2,3.2) rectangle (3.6,3.6);
    
    \filldraw[color=blue,fill=blue!30] (-3.4,4) rectangle (-3,4.4);
    \node at (-2.5,4.2) {\(\in A\)};
    \draw[color=red,pattern=horizontal lines, pattern color=red] (-3.4,3) rectangle (-3,3.4);
    \node at (-2.5,3.2) {\(\in B\)};
    \node at (-2.7,2) {\(\vert A \cup B \vert = w_b +1\)};
\end{tikzpicture}  
\end{center}

For all \(x \in A \cup \{x_0\} \cup B\), define \(\tau_{x_0,x}\) as before. 
Then for all \(\sigma \in \fS_{A \cup \{x_0\} \cup B}\),
\[\sigma (x_0) = x %
	    \iff \left(\tau_{x_0, x} \sigma\right) x_0 = x_0
        \iff \tau_{x_0,x} \sigma \in \fS_{A \cup B}
\]
and
\[
    \sigma (x_0) = x %
	    \iff \left(\sigma \tau_{x_0, x}\right) x = x
        \iff \sigma \tau_{x_0,x} \in \fS_{A \cup B \cup \{x_0\} \setminus \{x\}}.
\]
Then,
\begin{align*}
    G_{A \cup \{x_0\} \cup B} (T) %
	    & = \sum_{\sigma \in \fS_{A \cup \{x_0\} \cup B}} \sigma T \\
        & = \sum_{a \in A \cup \{x_0\}} \sum_{\substack{\sigma \in \fS_{A \cup \{x_0\} \cup B}, \\%
        \sigma \left(x_0\right) = a}} \sigma T 
        + \sum_{b \in B} \sum_{\substack{\sigma \in \fS_{A \cup \{x_0\} \cup B}, \\%
        \sigma \left(x_0\right) = b}} \sigma T \\
        & = \sum_{a \in A \cup \{x_0\}} \sum_{\substack{\sigma \in \fS_{A \cup \{x_0\} \cup B}, \\%
        \sigma \left(x_0\right) = a}} \tau_{x_o,a} \left( \tau_{x_0,a} \sigma \right) T 
        + \sum_{b \in B} \sum_{\substack{\sigma \in \fS_{A \cup \{x_0\} \cup B}, \\%
        \sigma \left(x_0\right) = b}} \left(\sigma \tau_{x_0, b}\right) \tau_{x_0,b} T \\
	    & = \sum_{a\in A \cup \{x_0\}} \sum_{\tilde{\sigma} \in \fS_{A \cup B}} %
	    \tau_{x_o,a} \tilde{\sigma} T
        + \sum_{b\in B} \sum_{\tilde{\sigma} \in \fS_{A \cup \{x_0\} \cup B \setminus \{b\}}} %
        \tilde{\sigma} \tau_{x_0,b} T \\
	    & =  \sum_{a \in A \cup \{x_0\}} \tau_{x_0,a} G_{A \cup B} (T) 
        + \sum_{b \in B} G_{A \cup \{x_0\} \cup B \setminus \{b\}} (\tau_{x_0,b} T)
\end{align*}
and as \(\tau_{x_0,a}\) is a row permutation for all \(a \in A \cup \{x_0\}\), up to row permutations we have
\begin{equation} \label{Generating Garnir Relations and Tools for Collapsing Sums subsection: Hooks with a Gap Lemma equation: Garnir}
    G_{A \cup \{x_0\} \cup B} (T) %
        = \vert A \cup \{x_0\} \vert G_{A \cup B} (T) %
         + \sum_{b \in B} G_{A \cup \{x_0\} \cup B \setminus \{b\}} (\tau_{x_0,b} T). 
\end{equation}

Solving for \(G_{A \cup B} (T)\) in equation \ref{Generating Garnir Relations and Tools for Collapsing Sums subsection: Hooks with a Gap Lemma equation: Garnir} we get
\[
    G_{A \cup B} (T) = %
        \frac{1}{\vert A \cup \{x_0\} \vert} \left( G_{A \cup \{x_0\} \cup B} (T) %
        - \sum_{b \in B} G_{A \cup \{x_0\} \cup B \setminus \{b\}} (\tau_{x_0,b} T) \right). 
\]
By Lemma \ref{Generating Garnir Relations and Tools for Collapsing Sums subsection: Minimal Size Lemma}, \(G_{A \cup B \cup \{x_0\}} \left( T \right)\) is generated by Garnirs over %
\(A \cup \{x_0\} \cup B \setminus \{b_0\}\). Thus \(G_{A \cup B} (T)\) is generated by Garnirs over \(A^\prime \cup B^\prime\), 
where \(A^\prime = A \cup \{x_0\}\), so that \(\vert A^\prime \cap [b](r) \vert = \vert A \cap [b](r) \vert + 1\), 
and \(B' = B \setminus \{b\}\) for some \(b \in B\), so that \(\vert B^\prime \vert = \vert B \vert - 1\). By induction, we get that 
\[
    G_{A \cup B} (T) \in \langle G_{A^\prime \cup B^\prime} (T^\prime) : T^\prime \in \mF_{\lambda, n}, ~ A^\prime = [b](r), ~ \vert B^\prime \vert = 1\rangle.
\]
\end{proof}

\subsection{} \label{Generating Garnir Relations and Tools for Collapsing Sums subsection: Simplifying a Garnir}

We now give a way to to write \(G_{A \cup B} (T)\) as above as a sum of \(2\)-cycles, which will make our calculations easier throughout.

\begin{lemma*} 
Let \(A = [b](r)\) and \(b_0 \in T_0 \setminus A\). Then for all \(T \in \mF_{\lambda,n}\),
\[
G_A (T) = w_b! \sum_{a \in A} \tau_{a,b_0} T \mod \mR_{\lambda,n}.
\]
\end{lemma*}

\begin{proof}
\begin{align*}
    G_A (T) %
	    & = \sum_{\sigma \in \fS_{A \cup \{b_0\}}} \sigma T\\
	    & = \sum_{\tilde{\sigma} \in \fS_A} \sum_{a \in A \cup \{b_0\}} \tilde{\sigma} \tau_{a,b_0} T\\%
        & = w_b! \sum_{a \in A \cup \{b_0\}} \tau_{a,b_0} T \mod \mR_{\lambda,n},
\end{align*}
as all \(\tilde{\sigma} \in \fS_{A}\) are row permutations and \(\vert A \vert = w_b\).
\end{proof}

\subsection{} \label{Generating Garnir Relations and Tools for Collapsing Sums subsection: Hooks Lemma}

To prove Theorem \ref{Generating Garnir Relations and Tools for Collapsing Sums subsection: One Garnir Theorem}, it remains to show that all Garnirs of the form \(G_{A \cup B} (T)\) where \(A = [b](r)\) and \(\vert B \vert = 1\), with \(B < A\), are generated by Garnirs over hooks. We show that for any such \(A\) and \(B\), \(G_{A \cup B} (T)\) is generated by Garnirs over \(A^\prime \cup B^\prime\) where \(A^\prime\) is a full row and \(\vert B \vert = 1\) with \(B^\prime < A^\prime\), and where the distance between \(A^\prime\) and \(B^\prime\) is less than the distance between \(A\) and \(B\). Theorem \ref{Generating Garnir Relations and Tools for Collapsing Sums subsection: One Garnir Theorem} is then proved by iterating this until we get that \(G_{A \cup B} (T)\) is generated (up to row permutation) by Garnirs over hooks.

\begin{lemma*}
Let \(T \in \mF_{\lambda,n}\), \(A = [b](r)\), \(B \subset T_0\) with \(\vert B \vert = 1\) and \(B < A\). Then
\[
    G_{A \cup B} (T) \in \langle G_{A'} (T^\prime)  :  T^\prime \in \mF_{\lambda,n}, A^\prime \text{ is a hook}\rangle.
\]
\end{lemma*}

\begin{proof}
Let \(A = [b](r)\) and \(B = \{b_0\}\) with \(b_0 \in [c](s)\) and \([c](s) < [b](r)\). Let \(j\) be the number of rows between \([b](r)\) and \([c](s)\). Without loss of generality we will assume that \(r > j + 1\) and \(b_0 = [b](r - j - 1, 1)\). Then
\begin{align*}
    G_{A \cup B} (T) %
        & = \sum_{\sigma \in A \cup B} \sigma T \\ %
        & = \sum_{\tilde{\sigma} \in S_A} \sum_{a \in A \cup B} \tilde{\sigma} \tau_{a,b_0} T \\%
        & = w_b! \sum_{a \in A \cup B} \tau_{a,b_0} T.
\end{align*}
We also have that for all \(a \in A \cup B\),
\[
    G_{[b](r-j) \cup B} ( \tau_{a,b_0} T ) = w_b! \left( \tau_{a,b_0} T + \sum_{x \in [b](r-j)} \tau_{x,b_0} \tau_{a,b_0} T\right)
\]
and hence
\[
    \tau_{a,b_0} T = \frac{1}{w_b!} G_{[b](r-j) \cup B} ( \tau_{a,b_0} T ) - \sum_{x \in [b](r-j)} \tau_{x,b_0} \tau_{a,b_0} T.
\]
Now observe that for all \(a \in A\cup B\) and all \(x \in [b](r-j)\), \(\tau_{x,b_0} \tau_{a,b_0} T = \tau_{a,x} \tau_{x, b_0} T\). 
Then,
\begin{align*}
    G_{A \cup B}(T) %
        & = w_b! \sum_{a \in A \cup B} \tau_{a,b_0} T\\
        & = w_b! \left( \sum_{a \in A \cup B} \frac{1}{w_b!} G_{[b](r-j) \cup B} ( \tau_{a,b_0} T ) -\sum_{x \in [b](r-j)} \tau_{x,b_0} \tau_{a,b_0} T\right)\\
        & = \sum_{a \in A \cup B} G_{[b](r-j) \cup B} ( \tau_{a,b_0} T ) - w_b! \sum_{a \in A \cup B} \sum_{x \in [b](r-j)} \tau_{x,b_0} \tau_{a,b_0} T\\
        & = \sum_{a \in A \cup B} G_{[b](r - j) \cup B} ( \tau_{a,b_0} T ) - w_b! \sum_{x\in[b](r - j)} \left( \tau_{x,b_0} T + \sum_{a \in A} \tau_{x,b_0} \tau_{a,b_0} T \right)\\
        & = \sum_{a \in A \cup B} G_{[b](r - j) \cup B} ( \tau_{a,b_0} T ) - w_b! \sum_{x\in[b](r - j)} \left( \tau_{x,b_0} T + \sum_{a \in A} \tau_{a,x} \tau_{x, b_0} T \right)\\
        & = \sum_{a \in A \cup B} G_{[b](r - j) \cup B} ( \tau_{a,b_0} T ) - w_b! \sum_{x\in[b](r - j)} G_{A \cup \{x\}} (\tau_{x,b_0} T).
\end{align*}
\end{proof}

So we have that for any \(T \in \mF_{\lambda,n}\) and any \(A \subset T_0\) with \(\vert A \vert > w_A\), \(G_A (T)\) is generated by Garnirs over hooks.

\subsection{} \label{Generating Garnir Relations and Tools for Collapsing Sums subsection: sigma A k}

The rest of this section is devoted to collapsing the sum in the image \(\Phi_m(T)\). 
We first consider the \(1\)-path case, where the idea is that the sum over all possible paths between two boxes can be collapsed to a single tableau, mod \(\mR_{\lambda \setminus X,n}\), with parity depending only on the number of rows between the two boxes. 
See Figure \ref{Generating Garnir Relations and Tools for Collapsing Sums subsection: sigma A k figure: Collapsing Sums}. 
We then generalize the result to \(2\)-paths, before considering the \(m\)-path case.

\begin{definition*}[\(\sigma^A_k\)]
Let \(A \subset T_0\) be a hook with \([b](r)\) the top row of \(A\). 
Label the boxes in \([b](r)\) as \(a_1, \ldots, a_{w_b}\) and let \(a_0\) be the box in \(A\) below \([b](r)\). 
For \(k = 0, \ldots, w_b\), define \(\sigma^A_k\) to be the permutation of \(A\) that switches \(a_0\) and \(a_k\) and is the identity otherwise. 
For \(T \in \mF_{\lambda, n}\) and \(0 \le k \le w_b\), let \(A_k = T_{a_k}\) and extend \(\sigma^A_k\) to act on the entries of \(T\), so that \(\sigma^A_k A_k = A_0\) and \(\sigma^A_k\) is the identity on \(T\) otherwise. 
Then, by Lemma \ref{Generating Garnir Relations and Tools for Collapsing Sums subsection: Simplifying a Garnir},
\[
    G_A (T) = w_b! \sum_{k = 0}^{w_b} \sigma^A_k T \mod \mR_{\lambda,n}.
\]
\end{definition*}

\begin{figure}[ht]
    \begin{tikzpicture}
        \draw (0,0) rectangle (4,6);
        
        \draw (.5,4) rectangle (1,4.5);
        \node at (.75,4.25) {\small \(z\)};
        \node at (2.5,4.25) {\small \(\leftarrow\) row \(r\)};
        
        \draw (2,-1) rectangle (2.5,-.5);
        \node at (2.25,-.75) {\small \(u\)};
        
        \node at (2,2) {\small{\(\sum\) over all}};
        \node at (2,1.5) {\small{possible paths}};
        
        \draw[blue] (2.25,-.55) to [out = 90, in = 0] (2,.25);
        \draw[blue] (2,.25) to [out = 200, in = 180] (2,.75);
        \draw[blue] (2,.75) to [out = 200, in = 180] (2,1.25);
        \draw[blue] (2,2.5) to [out = 200, in = 180] (2,3);
        \draw[blue] (2,3) to [out = 200, in = 180] (2,3.5);
        \draw[blue] (2,3.5) to [out = 90, in = 0] (.95,4.25);
        \draw[blue, ->] (.75,4.45) to [out = 90, in = 0] (.25,6.5);
        
        \draw (0,6.25) rectangle (.5,6.75);
    \end{tikzpicture}
        \hspace{.25in}
    \begin{tikzpicture}
        \draw (0,0) rectangle (4,6);
        
        \draw (.5,4) rectangle (1,4.5);
        \node at (.75,4.25) {\small \(u\)};
        
        \node at (2,2) {\small{(original tableau)}};
        
        \draw (0,6.25) rectangle (.5,6.75);
        \node at (.25,6.5) {\small \(z\)};
        
        \draw (2,-1) rectangle (2.5,-.5);
        \node at (3.25,-.75) {\small \(\leftarrow P^{-1}(u)\)};
        
        \node at (-2.25,3) {\Large \(\sim\)};
        
        \node at (-1,3) {\((-1)^{f(r)}\)};
    \end{tikzpicture}
\caption{Collapsing a sum of paths.}
\label{Generating Garnir Relations and Tools for Collapsing Sums subsection: sigma A k figure: Collapsing Sums}
\end{figure}

\subsection{} \label{Generating Garnir Relations and Tools for Collapsing Sums subsection: Path Extensions}

It will be useful to be able to identify those paths that are similar to a given \(m\)-path. Given an \(m\)-path \(P\) and two rows \([b](r)\) and \([c](s)\), a \(([b](r),[c](s))\)-path extension of \(P\) is an \(m\)-path \(Q\) that is identical to \(P\) except on the interval of rows \(([b](r),[c](s))\) and on any boxes whose image under \(P\) is in the interval of rows \(([b](r),[c](s))\). In the row interval \(([b](r),[c](s))\), \(Q\) can differ from \(P\), and in fact can even act on different boxes. 

\begin{example*}
Let the \(1\)-path \(P\) be given below. 
\[
    \begin{tikzpicture}[scale=.75, baseline={([yshift=-.5ex]current bounding box.center)}]
        \draw (0,0) rectangle (.5,1);
        \draw (0,1) rectangle (1.5,3);
        \draw (0,3) rectangle (3,3.5);
        \draw (0,3.5) rectangle (3.5,5);
        \draw (0,5) rectangle (.5,5.5);
        
        \filldraw[fill=blue!20] (0,0) rectangle (.5,.5);
        \filldraw[fill=blue!20] (0,.5) rectangle (.5,1);
        \filldraw[fill=blue!20] (1,1) rectangle (1.5,1.5);
        \filldraw[fill=blue!20] (.5,1.5) rectangle (1,2);
        \filldraw[fill=blue!20] (2.5,3) rectangle (3,3.5);
        \filldraw[fill=blue!20] (2,3.5) rectangle (2.5,4);
        
        \draw[blue] (.25,.25) to [out = 200, in = 180] (.25,.75);
        \draw[blue] (.25,.75) to [out = 90, in = 180] (1.25,1.25);
        \draw[blue] (1.25,1.25) to [out = 90, in = 0] (.75,1.75);
        \draw[blue] (.75,1.75) to [out = 90, in = 180] (2.75,3.25);
        \draw[blue] (2.75,3.25) to [out = 90, in = 0] (2.25,3.75);
        \draw[blue,->] (2.25,3.75) to [out = 90, in = 0] (.25,5.25);
        
    \end{tikzpicture}
\]
For any \(([2](2),[4](1))\)-path extension \(Q\) of \(P\), it must be that \(\{[1](1,1), \{[1](2,1), [2](1,3)\} \subset R^Q\) as these are the boxes in \(R^P\) outside of the interval of rows \(([2](2),[4](1))\). 
As \([2](1,3) \in P^{-1}(([2](2),[4](1)))\), \(Q\) must be identical to \(P\) on \(\{[1](1,1), \{[1](2,1)\}\), but it can be the case that \(Q([2](1,3)) \ne P([2](1,3))\). Two such examples of \(([2](2),[4](1))\)-path extensions of \(P\) are given below.
\[
    \begin{tikzpicture}[scale=.75, baseline={([yshift=-.5ex]current bounding box.center)}]
        \draw (0,0) rectangle (.5,1);
        \draw (0,1) rectangle (1.5,3);
        \draw (0,3) rectangle (3,3.5);
        \draw (0,3.5) rectangle (3.5,5);
        \draw (0,5) rectangle (.5,5.5);
        
        \filldraw[fill=blue!20] (0,0) rectangle (.5,.5);
        \filldraw[fill=blue!20] (0,.5) rectangle (.5,1);
        \filldraw[fill=blue!20] (1,1) rectangle (1.5,1.5);
        \filldraw[fill=blue!20] (.5,1.5) rectangle (1,2);
        \filldraw[fill=blue!20] (.5,2) rectangle (1,2.5);
        \filldraw[fill=blue!20] (1,2.5) rectangle (1.5,3);
        \filldraw[fill=blue!20] (1.5,3.5) rectangle (2,4);
        
        \draw[blue] (.25,.25) to [out = 200, in = 180] (.25,.75);
        \draw[blue] (.25,.75) to [out = 90, in = 180] (1.25,1.25);
        \draw[blue] (1.25,1.25) to [out = 90, in = 0] (.75,1.75);
        \draw[blue] (.75,1.75) to [out = 200, in = 180] (.75,2.25);
        \draw[blue] (.75,2.25) to [out = 90, in = 180] (1.25,2.75);
        \draw[blue] (1.25,2.75) to [out = 90, in = 180] (1.75,3.75);
        \draw[blue,->] (1.75,3.75) to [out = 90, in = 0] (.25,5.25);
        
    \end{tikzpicture}
    \hspace{.5in}
    \begin{tikzpicture}[scale=.75, baseline={([yshift=-.5ex]current bounding box.center)}]
        \draw (0,0) rectangle (.5,1);
        \draw (0,1) rectangle (1.5,3);
        \draw (0,3) rectangle (3,3.5);
        \draw (0,3.5) rectangle (3.5,5);
        \draw (0,5) rectangle (.5,5.5);
        
        \filldraw[fill=blue!20] (0,0) rectangle (.5,.5);
        \filldraw[fill=blue!20] (0,.5) rectangle (.5,1);
        \filldraw[fill=blue!20] (1,1) rectangle (1.5,1.5);
        \filldraw[fill=blue!20] (1,1.5) rectangle (1.5,2);
        \filldraw[fill=blue!20] (0,2) rectangle (.5,2.5);
        \filldraw[fill=blue!20] (1.5,3) rectangle (2,3.5);
        
        \draw[blue] (.25,.25) to [out = 200, in = 180] (.25,.75);
        \draw[blue] (.25,.75) to [out = 90, in = 180] (1.25,1.25);
        \draw[blue] (1.25,1.25) to [out = 200, in = 180] (1.25,1.75);
        \draw[blue] (1.25,1.75) to [out = 90, in = 0] (.25,2.25);
        \draw[blue] (.25,2.25) to [out = 90, in = 180] (1.75,3.25);
        \draw[blue,->] (1.75,3.25) to [out = 90, in = 0] (.25,5.25);
        
    \end{tikzpicture}
\]

Now let the \(2\)-path \(P\) be given below.
\[
    \begin{tikzpicture}[scale=.75, baseline={([yshift=-.5ex]current bounding box.center)}]
        \draw (0,0) rectangle (.5,1.5);
        \draw (0,1.5) rectangle (3,3);
        \draw (0,3) rectangle (4,6);
        \draw (0,6) rectangle (6,7.5);
        \draw (0,7.5) rectangle (8,8.5);
        \draw (0,8.5) rectangle (.5,9);
        \draw (0,9) rectangle (.5,9.5);
        
        \filldraw[fill=blue!20] (0,0) rectangle (.5,.5);
        \filldraw[fill=blue!20] (1.5,1.5) rectangle (2,2);
        \filldraw[fill=blue!20] (1,3.5) rectangle (1.5,4);
        \filldraw[fill=blue!20] (2,4.5) rectangle (2.5,5);
        \filldraw[fill=blue!20] (.5,5) rectangle (1,5.5);
        \filldraw[fill=blue!20] (.5,6) rectangle (1,6.5);
        \filldraw[fill=blue!20] (2,6.5) rectangle (2.5,7);
        \filldraw[fill=blue!20] (4.5,7) rectangle (5,7.5);
        
        \draw[color=red,pattern=dots, pattern color=red] (0,.5) rectangle (.5,1);
        \draw[color=red,pattern=dots, pattern color=red] (.5,2) rectangle (1,2.5);
        \draw[color=red,pattern=dots, pattern color=red] (0,2.5) rectangle (.5,3);
        \draw[color=red,pattern=dots, pattern color=red] (2.5,3) rectangle (3,3.5);
        \draw[color=red,pattern=dots, pattern color=red] (1.5,4) rectangle (2,4.5);
        \draw[color=red,pattern=dots, pattern color=red] (2.5,7.5) rectangle (3,8);
        
        \draw[blue, thick] (.25,.25) to [out = 90, in = 180] (1.75,1.75);
        \draw[blue, thick] (1.75,1.75) to [out = 90, in = 0] (1.25,3.75);
        \draw[blue, thick] (1.25,3.75) to [out = 90, in = 180] (2.25,4.75);
        \draw[blue, thick] (2.25,4.75) to [out = 90, in = 0] (.75,5.25);
        \draw[blue, thick] (.75,5.25) to [out = 200, in = 180] (.75,6.25);
        \draw[blue, thick] (.75,6.25) to [out = 90, in = 180] (2.25,6.75);
        \draw[blue, thick] (2.25,6.75) to [out = 90, in = 180] (4.75,7.25);
        \draw[blue, thick, ->] (4.75,7.25) to [out = 90, in = 0] (.25,8.75);
        
        \draw[red, very thick, dotted] (.25,.75) to [out = 90, in = 180] (.75,2.25);
        \draw[red, very thick, dotted] (.75,2.25) to [out = 90, in = 0] (.25,2.75);
        \draw[red, very thick, dotted] (.25,2.75) to [out = 90, in = 180] (2.75,3.25);
        \draw[red, very thick, dotted] (2.75,3.25) to [out = 90, in = 0] (1.75,4.25);
        \draw[red, very thick, dotted] (1.75,4.25) to [out = 90, in = 180] (2.75,7.75);
        \draw[red, very thick, dotted, ->] (2.75,7.75) to [out = 90, in = 0] (.25,9.25);
        
    \end{tikzpicture}
\]
For any \(([3](1), [4](1))\)-path extension \(Q\) of \(P\), it must be that \(R^Q \setminus ([3](1), [4](1)) = R^P \setminus ([3](1), [4](1))\). 
As \(\{[2](2,2),[2](3,1)\} \subset P^{-1}(([3](1), [4](1)))\), \(Q\) must be identical to \(P\) on \(R^Q \setminus \left(([3](1), [4](1)) \cup \{[2](2,2),[2](3,1)\} \right)\), but \(Q\) can differ from \(P\) otherwise. An example of a \(([3](1), [4](1))\)-path extension of \(P\) is given below.
\[
    \begin{tikzpicture}[scale=.75, baseline={([yshift=-.5ex]current bounding box.center)}]
        \draw (0,0) rectangle (.5,1.5);
        \draw (0,1.5) rectangle (3,3);
        \draw (0,3) rectangle (4,6);
        \draw (0,6) rectangle (6,7.5);
        \draw (0,7.5) rectangle (8,8.5);
        \draw (0,8.5) rectangle (.5,9);
        \draw (0,9) rectangle (.5,9.5);
        
        \filldraw[fill=blue!20] (0,0) rectangle (.5,.5);
        \filldraw[fill=blue!20] (1.5,1.5) rectangle (2,2);
        \filldraw[fill=blue!20] (.5,3) rectangle (1,3.5);
        \filldraw[fill=blue!20] (0,4) rectangle (.5,4.5);
        \filldraw[fill=blue!20] (1.5,6) rectangle (2,6.5);
        \filldraw[fill=blue!20] (2.5,7.5) rectangle (3,8);
        
        \draw[color=red,pattern=dots, pattern color=red] (0,.5) rectangle (.5,1);
        \draw[color=red,pattern=dots, pattern color=red] (.5,2) rectangle (1,2.5);
        \draw[color=red,pattern=dots, pattern color=red] (0,2.5) rectangle (.5,3);
        \draw[color=red,pattern=dots, pattern color=red] (2,3.5) rectangle (2.5,4);
        \draw[color=red,pattern=dots, pattern color=red] (1.5,4.5) rectangle (2,5);
        \draw[color=red,pattern=dots, pattern color=red] (3,5) rectangle (3.5,5.5);
        \draw[color=red,pattern=dots, pattern color=red] (2,6.5) rectangle (2.5,7);
        \draw[color=red,pattern=dots, pattern color=red] (4.5,7) rectangle (5,7.5);
        
        \draw[blue, thick] (.25,.25) to [out = 90, in = 180] (1.75,1.75);
        \draw[blue, thick] (1.75,1.75) to [out = 90, in = 0] (.75,3.25);
        \draw[blue, thick] (.75,3.25) to [out = 90, in = 0] (.25,4.25);
        \draw[blue, thick] (.25,4.25) to [out = 90, in = 180] (1.75,6.25);
        \draw[blue, thick] (1.75,6.25) to [out = 90, in = 180] (2.75,7.75);
        \draw[blue, thick, ->] (2.75,7.75) to [out = 90, in = 0] (.25,8.75);
        
        \draw[red, very thick, dotted] (.25,.75) to [out = 90, in = 180] (.75,2.25);
        \draw[red, very thick, dotted] (.75,2.25) to [out = 90, in = 0] (.25,2.75);
        \draw[red, very thick, dotted] (.25,2.75) to [out = 90, in = 180] (2.25,3.75);
        \draw[red, very thick, dotted] (2.25,3.75) to [out = 90, in = 0] (1.75,4.75);
        \draw[red, very thick, dotted] (1.75,4.75) to [out = 90, in = 180] (3.25,5.25);
        \draw[red, dotted, very thick] (3.25,5.25) to [out = 90, in = 0] (2.25,6.75);
        \draw[red, dotted, very thick] (2.25,6.75) to [out = 90, in = 180] (4.75,7.25);
        \draw[red, very thick, dotted, ->] (4.75,7.25) to [out = 90, in = 0] (.25,9.25);

    \end{tikzpicture}
\]
\end{example*}

Given an evacuation route \(R\) and a row \([b](r)\), define
\[
    R_{<[b](r)} := \{x \in R : x < [b](r)\} \qquad \text{ and } \qquad R_{>[b](r)} := \{ x \in R : [b](r) < x\}.
\]
We formalize the notion of path extensions with the following definitions.

\begin{definition}[Route Extension] \label{Generating Garnir Relations and Tools for Collapsing Sums subsection: Path Extensions definition: Route Extension}
Given an evacuation route \(R\) and two rows \([b](r)\) and \([c](s)\) with \([b](r) \le [c](s)\), an evacuation route \(B\) is a \(([b](r),[c](s))\)-route extension of \(R\) if \(R_{<[b](r)} = B_{<[b](r)}\) and \(R_{>[c](s)} = B_{>[c](s)}\).
\end{definition}

\begin{definition}[Path Extension] \label{Generating Garnir Relations and Tools for Collapsing Sums subsection: Path Extensions definition: Path Extension}
Given an \(m\)-path \(P\) and two rows \([b](r)\) and \([c](s)\) with \([b](r) \le [c](s)\), an \(m\)-path \(Q\) is a \(([b](r),[c](s))\)-path extension of \(P\) if:%
	\begin{itemize}
    \item \(R^Q\) is a \(([b](r),[c](s))\)-route extension of \(R^P\),
    \item \(\left. P \right\vert_{R^P_{>[c](s)}} = \left. Q \right\vert_{R^P_{>[c](s)}}\)
    \item \(\left. P \right\vert_{R^P_{<[b](r)} \setminus I} = \left. Q \right\vert_{R^P_{<[b](r)} \setminus I}\), where \(I = \{x \in R^P_{<[b](r)}  :  P(x) \in ([b](r),[c](s))\}\). \end{itemize}
\end{definition}

\subsection{} \label{Generating Garnir Relations and Tools for Collapsing Sums subsection: Calculation Lemma for 1-paths}

For any \(T \in \mF_{\lambda,n}\), let 
\[
    X = \{x_1 := [b_1](1,w_{b_1})\} \text{ and } Y = \{y_1 := [N+1](1,1)\}
\]
and let
\[
    z_1 := T_{[b_z](i_1,j_1)}
\]
for some \(1 \le b_1 \le b_z \le N\), \(1 \le i_1 \le h_{b_z}\), and \(1 \le j_1 \le w_{b_z}\). 
Let \(u := T_{[b_u](i_u,j_u)})\) for some \(b_1 \le b_u \le b_z\), \(1 \le i_u \le h_{b_u}\), and \(1 \le j_u \le w_{b_u}\), and let \(P\) be any \(1\)-path on \(\lambda\) removing \(X\) such that \(P([b_u](i_u,j_u)) \in [b_z](1)\) and \(P([b_z](i_1,j_1)) > [b_z](i_1)\), including the case \(\ds P([b_z](i_1,j_1)) = y_1\). 
Let
\begin{align*}
    \left[ P \right] %
        = & \{1 \text{-paths } Q \text{ on } \lambda  :  Q \text{ is a } ([b_z](1),[b_z](i_1)) \text{-path extension of } P\\
        & \text{ with } Q([b_z](i_1,j_1)) = P([b_z](i_1,j_1))\}
\end{align*}
and \(T^\prime \in \mF_{\lambda \setminus X, n}\) be the unique tableau such that \(T^\prime = T_P \text{ on } (\lambda \setminus X)\), except on the interval of rows \(([b_z](1), [b_z](i_1))\), where \(T^\prime = T\), except \(T^\prime_{[b_z](i_1,j_1)} = u\). 
We then have the following.

\begin{lemma*} 
\[
    \sum_{Q \in \left[ P \right]} Q(T) = (-1)^{i_1-1} ~ \ytableaushort{{\alpha^P_1}} \otimes T^\prime %
    \mod F_1 \otimes \mR_{\lambda \setminus X, n}.
\]
\end{lemma*}

\begin{proof}
Assume, without loss of generality, that \(j_1 = 1\). 
We will show the case \(b_u < b_z\), the case \(b_u = b_z\) is similar. 
If \(i_1 = 1\), then \(\left[ P \right] = \{ P \}\), and so 
\begin{align*}
    \sum_{Q \in \left[ P \right]} Q(T) %
        & = P(T) \mod F_1 \otimes \mR_{\lambda \setminus X, n} \\
        & = \ytableaushort{{\alpha^P_1}} \otimes T^\prime \mod F_1 \otimes \mR_{\lambda \setminus X, n},
\end{align*}
as desired. Let \(i_1 = 2\) and 
\[
    A = \{a_1 := [b_z](1,k), \ldots, a_{w_{b_z}} := [b_z](1,w_{b_z})\} \cup \{a_0 := [b_z](2,1)\}.
\]
Then by Lemma \ref{Generating Garnir Relations and Tools for Collapsing Sums subsection: Simplifying a Garnir} we have the following (see Figure \ref{Generating Garnir Relations and Tools for Collapsing Sums subsection: Calculation Lemma for 1-paths figure: Calculation for 1-paths}). 
\begin{align*}
    \sum_{Q \in \left[ P \right]} Q(T) %
        & = \sum_{k = 1}^{w_{b_z}} ~ \ytableaushort{{\alpha^P_1}} ~ \otimes \sigma^A_k T^\prime\\
        & = \ytableaushort{{\alpha^P_1}} ~ \otimes \left( \dfrac{1}{w_{b_z}!} G_A \left( T^\prime \right) %
        - T^\prime \right) \mod F_1 \otimes \mR_{\lambda \setminus X, n}\\
        & = - ~ \ytableaushort{{\alpha^P_1}} ~ \otimes T^\prime \mod F_1 \otimes \mR_{\lambda \setminus X, n}.
\end{align*}

\begin{figure}[ht]
\begin{tikzpicture}
    \draw (0,0) rectangle (2,2);
    \draw (0,-.25) -- (0,0);
    \draw (0,2) -- (0,2.25);
    \draw (1.5,-.25) -- (1.5,0);
    \draw (2,2) -- (2.5,2) -- (2.5,2.25);
    
    \draw (0,.5) rectangle (.5,1);
    \node at (.25,.75) {\small \(a_k\)};
    
    \draw (0,.5) -- (2,.5);
    \draw (.75,0) -- (.75,.5);
    \draw (1.25,0) -- (1.25,.5);
    \node at (1,.25) {\small \(u\)};
    \node at (1,.65) {\tiny \(k\)};
    
    \node at (.5,-.5) {\small \ };
    
    \node at (-1,1) {\(\ds \sum_{k=1}^{w_b}\)};
\end{tikzpicture}
\begin{tikzpicture}
    \node at (-2,1) {\large \(\sim\)};

    \node at (-1,1) {\large \(-\)};

    \draw (0,0) rectangle (2,2);
    \draw (0,-.25) -- (0,0);
    \draw (0,2) -- (0,2.25);
    \draw (1.5,-.25) -- (1.5,0);
    \draw (2,2) -- (2.5,2) -- (2.5,2.25);
    
    \draw (0,.5) rectangle (.5,1);
    \node at (.25,.75) {\small \(u\)};
    
    \draw (0,.5) -- (2,.5);
    
    \node at (.5,-.5) {\small \ };
\end{tikzpicture}
\caption{}
\label{Generating Garnir Relations and Tools for Collapsing Sums subsection: Calculation Lemma for 1-paths figure: Calculation for 1-paths}
\end{figure}

Now let \(i_1 > 2\) and 
\[
    B = \{b_1 := [b_z](i_1 - 1,k), \ldots, b_{w_{b_z}} := [b_z](i_1 - 1,w_{b_z})\} \cup \{b_0 := [b_z](i_1,1)\}.
\]
Then by Lemma \ref{Generating Garnir Relations and Tools for Collapsing Sums subsection: Simplifying a Garnir} and induction applied to each entry in \((i_1-1)^{b_z}\), we have %
\begin{align*}
    \sum_{Q \in \left[ P \right]} Q(T) %
        & = \sum_{k = 1}^{w_{b_z}} \left( -1 \right)^{i_1-2} ~  \ytableaushort{{\alpha^P_1}} ~ \otimes %
        \sigma^B_k T^\prime \mod F_1 \otimes \mR_{\lambda \setminus X, n}\\
        & = \left( -1 \right)^{i_1-2} ~ \ytableaushort{{\alpha^P_1}} ~ \otimes \left( \dfrac{1}{w_{b_z}!} G_B %
        \left( T^\prime \right) - T^\prime \right) \mod F_1 \otimes \mR_{\lambda \setminus X, n}\\
        & = \left( -1 \right)^{i_1-1} ~ \ytableaushort{{\alpha^P_1}} ~ \otimes T^\prime %
        \mod F_1 \otimes \mR_{\lambda \setminus X, n}.
\end{align*}
Thus the claim holds for \(1 \le i_1 \le h_{b_z}\). 
\end{proof}

\subsection{} \label{Generating Garnir Relations and Tools for Collapsing Sums subsection: Calculation Lemma for 2-paths}

Lemma \ref{Generating Garnir Relations and Tools for Collapsing Sums subsection: Calculation Lemma for 1-paths} also allows for calculations of sums of \(2\)-paths, by applying the technique of the proof twice and ``skipping'' certain rows each time. That is, for any \(T \in \mF_{\lambda,n}\), let 
\[
    X := \{x_1 := [b_1](1,w_{b_1}), x_2 := [b_2](i_2,w_{b_2})\},
\]
\[
    Y := \{y_1 := [N+1](1,1), y_2 := [N+1](2,1)\}
\]
and let
\[
    z_1 := T_{[b_z](i_1,j_1)}, ~ z_2 := T_{[b_z](i_2,j_2)}
\]
for some \(1 \le b_1 \le b_z \le N\), \(1 \le i_2 < i_1 \le h_{b_z}\), and \(1 \le j_1,j_2 \le w_{b_z}\). 
Let
\[
    u_1 := T_{[b_{u_1}](i_{u_1},j_{u_1})}, ~ u_2 := T_{[b_{u_2}](i_{u_2},j_{u_2})}
\]
for some \(b_1 \le b_{u_1},b_{u_2} \le b_z\), \(1 \le i_{u_1} \le h_{b_{u_1}}\), \(1 \le j_{u_1} \le w_{b_{u_1}}\), \(1 \le i_{u_2} \le h_{b_{u_2}}\), and \(1 \le j_{u_2} \le w_{b_{u_2}}\). 
If \(b_{u_1} = b_{u_2}\), then we also assume that \(i_{u_1} \ne i_{u_2}\). 
Let \(P\) be any \(2\)-path on \(\lambda\) such that
\[
    P( [b_{u_1}](i_{u_1},j_{u_1}) ) \in [b_z](1), P( [b_{u_2}](i_{u_2},j_{u_2}) ) \in [b_z](2),
\]
\[
     P( [b_z](i_1,j_1) ),  P( [b_z](i_2,j_2) ) > [b_z](i_1).  
\]
Assume, without loss of generality, that \(P([b_z](i_1,j_1)), P([b_z](i_2,j_2)) \not \in Y\). Let
\begin{align*}
    \left[ P \right] %
        = &\{2 \text{-paths } Q \text{ on } \lambda  :  Q \text{ is a } ([b_z](1),[b_z](i_1))%
        \text{-path extension of } P\\
        & \text{ with } Q([b_z](i_1,j_1)) = P([b_z](i_1,j_1)), Q([b_z](i_2,j_2)) = P([b_z](i_2,j_2))\}
\end{align*}
and \(T^\prime \in \mF_{\lambda \setminus X, n}\) be the unique tableau such that \(T^\prime = T_P\) on \((\lambda \setminus X)\), except on the interval of rows \(([b_z](1), [b_z](i_1))\), where \(T^\prime = T\), except \(T^\prime_{[b_z](i_1,j_1)} = u, ~ T^\prime_{[b_z](i_2,j_2)} = v\).
We then have the following.

\begin{corollary*}
\[
    \sum_{Q \in \left[ P \right]} Q(T) = (-1)^{i_1 - 2 + i_2 - 2} ~ %
    \ytableaushort{{\alpha^P_2}, {\alpha^P_1}} %
    \otimes T^\prime \mod F_2 \otimes \mR_{\lambda \setminus X, n}.
\]
\end{corollary*}

\begin{proof}
Apply the techniques from the proof of Lemma \ref{Generating Garnir Relations and Tools for Collapsing Sums subsection: Calculation Lemma for 1-paths} to get a sum of tableaux with \(u_1\) in the box \([b_z](i_1,j_1)\), skipping row \([b_z](i_2)\), then apply the techniques again to get \(u_2\) in the box \([b_z](i_2,j_2)\), skipping row \([b_z](i_2 - 1)\).
\end{proof}

\subsection{} \label{Generating Garnir Relations and Tools for Collapsing Sums subsection: Calculation Lemma for m-paths}

The same technique used above immediately generalizes to sums of \(m\)-path extensions.
Fix \(m > 2\). 
For any \(T \in \mF_{\lambda,n}\), let 
\[
    X = \{x_1 = [b_1](1,w_{b_1}), \ldots, x_m = [b_m](i_m,w_{b_m})\}
\]
be a removal set and
\[
    z_k = T_{[b_z](i_k,j_k)} \quad \text{ for } 1 \le k \le m
\]
for some \(1 \le b_1 \le b_z \le N\), \(1 \le i_m < \cdots < i_1 \le h_{b_z}\), and \(1 \le j_k \le w_{b_z}\) for \(1 \le k \le m\). 
Let
\[
    u_k = T_{[b_{u_k}](i_{u_k},j_{u_k})} \quad \text{ for } 1 \le k \le m
\]
for some \(b_1 \le b_{u_k} \le b_z\), \(1 \le i_{u_k} \le h_{b_{u_k}}\), \(1 \le j_{u_k} \le w_{b_{u_k}}\). 
If \(b_{u_k} = b_{u_l}\) for \(k \ne l\), then we also assume that \(i_{u_k} \ne i_{u_l}\). 
Let \(P\) be any \(m\)-path on \(\lambda\) such that, for \(1 \le k \le m\),
\[
    P([b_{u_k}](i_{u_k},j_{u_k})) \in [b_z](k)
\]
and
\[
    P([b_z](i_k,j_k)) > [b_z](i_1).  
\]
Assume, without loss of generality, that \(P([b_z](i_k,j_k)) \not \in Y\) for \(1 \le k \le m\). Let
\begin{align*}
    \left[ P \right] %
        = & \{m \text{-paths } Q \text{ on } \lambda  :  Q \text{ is a } ([b_z](1),[b_z](i_1)) \text{-path extension of } P \\ %
        & \text{ such that } Q([b_z](i_k,j_k)) = P([b_z](i_k,j_k)) \text{ for } 1 \le k \le m\}
\end{align*}
and \(T^\prime \in \mF_{\lambda \setminus X, n}\) be the unique tableau such that \(T^\prime = T_P \text{ on } (\lambda \setminus X)\) except on \(([b_z](1), [b_z](i_2))\), where \(T^\prime = T\) except \(T^\prime_{[b_z](i_k,j_k)} = u_k\) for \(1 \le k \le m\).
We then have the following.

\begin{corollary*}
\[
    \sum_{Q \in \left[ P \right]} Q(T) = (-1)^{i_1 - m + \cdots + i_m - m} ~ %
    \begin{ytableau} \alpha^P_m \\ \vdots \\ \alpha^P_1 \end{ytableau} %
    \otimes T^\prime \mod F_m \otimes \mR_{\lambda \setminus X, n}
\]
\end{corollary*}

\begin{proof}
Assume, without loss of generality, that 
\begin{align*}
    i_1 &> i_2 \\
    i_2 &> i_3 + 1 \\
        \vdots & \\
    i_{m-1} &> i_m + (m-1) - 1.\\   
\end{align*}
Otherwise, the following goes through by skipping the appropriate rows. 
Apply the techniques from the proof of Lemma \ref{Generating Garnir Relations and Tools for Collapsing Sums subsection: Calculation Lemma for 1-paths} to get a sum of tableaux with \(u_1\) in the box \([b_z](i_1,j_1)\), skipping rows \(i_m^z, \ldots, i_2^z\). 
Then iterate the techniques again to get \(u_k\) in the box \([b_z](i_k,j_k)\), skipping rows \([b_z](i_m), \ldots, [b_z](i_{k+1})\) and rows \([b_z](i_k - 1), \ldots [b_z](i_k - (k - 1))\).
\end{proof}

\section{Showing the Pieri Inclusion Removing One Box is a \(GL(V)\)-map} \label{section: Showing the Pieri Inclusion Removing One Box is a GL(V)-map}

\subsection{} \label{Showing the Pieri Inclusion Removing One Box is a GL(V)-map subsection: The Theorem}

For all of Section \ref{section: Showing the Pieri Inclusion Removing One Box is a GL(V)-map}, fix \(X = \{x_1 := [b_1](1,w_{b_1})\} \subset T_0\) to be removed. 
Let
\[
    \Phi_1 : \mF_{\lambda,n} \to V \otimes \mF_{\lambda \setminus X, n} 
\]
be as in Section \ref{Constructing the Pieri Inclusion for Removing One Box subsection: Defining Phi One Box Removal}.

\begin{theorem*}
    \(\Phi_1\) is a \(\GL(V)\)-map, i.e. \(\Phi_1\) descends to  
    \[
        \Phi_1 : \bS_{\lambda} (V) \to F_1 \otimes \bS_{\lambda \setminus X} (V)
    \]
    and \(\Phi_1\) is \(\GL(V)\)-equivaraint.
\end{theorem*}

For each simple root vector \(\alpha_i\) with respect the standard Cartan subalgebra, the action of \(e_{\alpha_i}\) on a tableau \(T\) generates a sum of tableau \(\widetilde T\) where each entry \(i\) in \(T\) is replaced by an \(i + 1\).
Similarly, for each \(e_{-\alpha_i}\), where each entry \(i\) in \(T\) is replaced by an \(i - 1\).
As \(\Phi_1\) is a sum over \(1\)-paths that move entries up the diagram, acting with \(e_{\alpha_i}\) and applying \(\Phi_1\) to the sum is the same as the opposite order.
As the simple root vectors generate \(\gl(V)\), \(\Phi_1\) is \(\gl(V)\)-equivariant.

To prove Theorem \ref{Showing the Pieri Inclusion Removing One Box is a GL(V)-map subsection: The Theorem}, it remains to show that
\[
    \Phi_1(\mR_{\lambda, n}) \subset F_1 \otimes \mR_{\lambda \setminus X, n}.
\]
It is clear that \(\Phi_1\) preserves Property \ref{Constructing Schur--Weyl Modules subsection: The Construction equation: R1} as it is a sum over all \(1\)-paths, and hence we must show that Property \ref{Constructing Schur--Weyl Modules subsection: The Construction equation: R2} holds, i.e. for all \(T \in \mF_{\lambda,n}\) and all \(A \subset T_0\) with \(\vert A \vert > w_A\),
\begin{equation} \label{Showing the Pieri Inclusion Removing One Box is a GL(V)-map subsection: The Theorem equation: Phi Preserves Garnirs One Box}
    \Phi_1(G_A (T)) \in F_1 \otimes \mR_{\lambda \setminus X, n}.
\end{equation}

By Theorem \ref{Generating Garnir Relations and Tools for Collapsing Sums subsection: One Garnir Theorem}, it is enough to show that Equation \ref{Showing the Pieri Inclusion Removing One Box is a GL(V)-map subsection: The Theorem equation: Phi Preserves Garnirs One Box} holds for all hooks \(A\). 
If \(A\) is a hook, either \(A\) is completely contained in a block \(b\), with \(1 \leq b \leq N\), or \(A\) is contained in two blocks, \(b\) and \(b + 1\), with \(1 \leq b \leq N-1\). 
We consider these two options separately. 

\begin{center}
    \begin{tikzpicture}
        \draw[thick] (0,.5) -- (0,3.5);
        \draw[thick] (4,.5) -- (4,3.5);
        \draw[thick, dashed] (0,0) -- (0,4);
        \draw[thick, dashed] (4,0) -- (4,4);
        \draw[] (0,3) -- (4,3);
        \draw[] (0,2) -- (4,2);
        \draw[] (0,1) -- (1,1);
        \draw[] (1,1) -- (1,3);
        \draw[] (3,2) -- (3,3);
        \node at (.5,1.5) {\small \(A_0\)};
        \node at (.5,2.5) {\small \(A_1\)};
        \node at (2,2.5) {\(\cdots\)};
        \node at (3.5,2.5) {\small \(A_{w_b}\)};
        \node at (2,-.5) {\(A \subset [b]\).};
    \end{tikzpicture}
    \qquad \qquad
    \begin{tikzpicture}
        \draw[thick] (0,.5) -- (0,3.5);
        \draw[thick] (2,.5) -- (2,2);
        \draw[thick] (4,2) -- (4,3.5);
        \draw[thick, dashed] (0,0) -- (0,4);
        \draw[thick, dashed] (2,0) -- (2,.5);
        \draw[thick,dashed] (4,3.5) -- (4,4);
        \draw[] (0,3) -- (4,3);
        \draw[] (0,2) -- (4,2);
        \draw[] (0,1) -- (1,1);
        \draw[] (1,1) -- (1,3);
        \draw[] (3,2) -- (3,3);
        \node at (.5,1.5) {\small \(A_0\)};
        \node at (.5,2.5) {\small \(A_1\)};
        \node at (2,2.5) {\(\cdots\)};
        \node at (3.5,2.5) {\small \(A_{w_{b + 1}}\)};
        \node at (2,-.5) {\(A \subset [b] \cup [b + 1]\).};
    \end{tikzpicture}
\end{center}

\subsection{One Box Removal Preserves Garnirs for Hooks Contained in a Single Block} \label{Showing the Pieri Inclusion Removing One Box is a GL(V)-map subsection: One Box Removal Preserves Garnirs for Hooks Contained in a Single Block}

We first show that Equation \ref{Showing the Pieri Inclusion Removing One Box is a GL(V)-map subsection: The Theorem equation: Phi Preserves Garnirs One Box} holds for all hooks \(A \subset [b]\), for some \(1 \le b \le N\).
For the rest of Section \ref{Showing the Pieri Inclusion Removing One Box is a GL(V)-map subsection: One Box Removal Preserves Garnirs for Hooks Contained in a Single Block}, fix \(T \in \mF_{\lambda, n}\) and
\[
    A = \{a_0 := [b](i_0,1), a_1 := [b](i_0 + 1,1), \ldots, a_{w_b} := [b](i_0 + 1,w_b)\} \subset T_0    
\]
with \(1 \le i_0 < h_b\), so that \(A \subset [b]\). 
Denote the entries of \(A\) in \(T\) by \(A_k = T_{a_k}\) for \(k = 0, 1, \ldots, w_b\). 
Then by Lemma \ref{Generating Garnir Relations and Tools for Collapsing Sums subsection: Simplifying a Garnir}, mod \(F_1 \otimes \mR_{\lambda \setminus X, n}\) we have 
\[
    \Phi_1\left( G_A (T) \right) 
        = \sum_{P} \frac{(-1)^{P}}{H(P)} P\left(\sum_{\sigma \in S_A} \sigma T \right)
        = C \sum_{P} \sum_{k = 0}^{w_b} \frac{(-1)^{P}}{H(P)} P\left(\sigma^A_k T \right),
\]
where the sum is over all \(1\)-paths \(P\) on \(\lambda\) removing \(X\).
The set of all \(P_k := P(\sigma^A_k T)\) appearing in the image \(\Phi_1\left(G_A (T) \right)\) above is the union of the following disjoint sets.\\

The \(P_k\)s that miss \(A\),
\begin{equation} \label{Showing the Pieri Inclusion Removing One Box is a GL(V)-map subsection: One Box Removal Preserves Garnirs for Hooks Contained in a Single Block equation: T1}
    \sT_1 = \{P_k  :  R^P \cap A = \emptyset\}.\\
\end{equation}

The \(P_k\)s that hit \(A\) and keep \(A\) in block \(b\),
\begin{equation} \label{Showing the Pieri Inclusion Removing One Box is a GL(V)-map subsection: One Box Removal Preserves Garnirs for Hooks Contained in a Single Block equation: T2}
    \sT_2 = \{P_k  :  R^P \cap A \ne \emptyset, ~ P(A) \le [b]\}.\\
\end{equation}

The \(P_k\)s that hit \(A\) and move the entry \(A_i\) above block \(b\), including \(P(\sigma_k^A A_i) \in Y\),
\begin{equation} \label{Showing the Pieri Inclusion Removing One Box is a GL(V)-map subsection: One Box Removal Preserves Garnirs for Hooks Contained in a Single Block equation: T3}
    \sT_3 = \bigsqcup_{i = 0}^{w_b} \sT_3^i, \quad \text{ where } \quad \sT_3^i = \{P_k :  R^P \cap A \ne \emptyset, ~ P(\sigma^A_k A_i) > [b]\}.\\
\end{equation}

We then have, mod \(F_1 \otimes \mR_{\lambda \setminus X, n}\),
\[
    \Phi_1(G_A(T)) 
	    = C \sum_{P} \sum_{k = 0}^{w_b} \frac{(-1)^{P}}{H(P)} P\left(\sigma^A_k T \right)
	    = C \sum_{j = 1}^{3} \sum_{P_k \in \sT_j} \frac{(-1)^{P}}{H(P)} P_k.
\]
We show that for each of the cases (\ref{Showing the Pieri Inclusion Removing One Box is a GL(V)-map subsection: One Box Removal Preserves Garnirs for Hooks Contained in a Single Block equation: T1}) - (\ref{Showing the Pieri Inclusion Removing One Box is a GL(V)-map subsection: One Box Removal Preserves Garnirs for Hooks Contained in a Single Block equation: T3}), 
\[
    \sum_{P_k \in T_j} \frac{(-1)^{P}}{H(P)} P_k \in F_1 \otimes \mR_{\lambda \setminus X,n},
\]
and hence Equation \ref{Showing the Pieri Inclusion Removing One Box is a GL(V)-map subsection: The Theorem equation: Phi Preserves Garnirs One Box} holds for all hooks \(A\) with \(A \subset [b]\).\\

\begin{case*}[\ref{Showing the Pieri Inclusion Removing One Box is a GL(V)-map subsection: One Box Removal Preserves Garnirs for Hooks Contained in a Single Block equation: T1}]
In this case we show that the sum over all paths that miss \(A\) is in \(F_1 \otimes \mR_{\lambda \setminus X,n}\), i.e.
\[
    \sum_{P_k \in \sT_1} \frac{(-1)^{P}}{H(P)} P_k \in F_1 \otimes \mR_{\lambda \setminus X,n}.
\]
\end{case*}

\begin{proof}
As \(P\) misses \(A\) for all \(P_k \in \sT_1\), \(P\vert_A = \text{id}_A\), and thus 
\[
    P\left( \sigma^A_k T \right) = Y_P \otimes \sigma^A_k T_P \text{ for all } 0 \le k \le w_b.    
\]
Then we have, mod \(F_1 \otimes \mR_{\lambda \setminus X, n}\), 
\begin{align*}
    \sum_{P_k \in \sT_1} \frac{(-1)^{P}}{H(P)} P_k %
    	& = \sum_{P_0 \in \sT_1} \sum_{k = 0}^{w_b} \frac{(-1)^{P}}{H(P)} P\left(\sigma^A_k T\right)\\
        & = \sum_{P_0 \in \sT_1} \sum_{k = 0}^{w_b} \frac{(-1)^{P}}{H(P)} \ytableaushort{{\alpha^P_1}}%
        \otimes \sigma^A_k T_P\\
        & = \sum_{P_0 \in \sT_1} \dfrac{1}{C_A} \frac{(-1)^{P}}{H(P)} \ytableaushort{{\alpha^P_1}} %
        \otimes G_A(T_P)\\
        & = 0.
\end{align*}
\end{proof}

\begin{case*}[\ref{Showing the Pieri Inclusion Removing One Box is a GL(V)-map subsection: One Box Removal Preserves Garnirs for Hooks Contained in a Single Block equation: T2}]
In this case we show that the sum over all paths that hit \(A\) and keep \(A\) in block \(b\) is in \(F_1 \otimes \mR_{\lambda \setminus X,n}\), i.e.
\[
    \sum_{P_k \in \sT_2} \frac{(-1)^{P}}{H(P)} P_k \in F_1 \otimes \mR_{\lambda \setminus X,n}.
\]
\end{case*}

\begin{proof}
For a 1-path \(P\), let \(P^{-1}\) be the unique map of boxes 
\[
    P^{-1}: \lambda \cup Y \to \lambda \cup Y
\]
such that for all \(x \in \lambda \cup Y\), \(P^{-1} (P(x)) = x\). 
For all \(k = 0, 1, \ldots, w_b\), let 
\[
    \tau^A_k := P \sigma^A_k P^{-1} \in S_{P(A)},
\]
so that \(\tau^A_k\) permutes \(P(a_0)\) and \(P(a_k)\) and is the identity otherwise. 
Extend \(\tau^A_k\) to act on the entries of \(T_P\). 
Then, mod \(F_1 \otimes \mR_{\lambda \setminus X, n}\), 
\begin{align*}
	\sum_{P_k \in \sT_2} \frac{(-1)^{P}}{H(P)} P_k %
    	& = \sum_{P_0 \in \sT_2} \sum_{k=0}^{w_b} \frac{(-1)^{P}}{H(P)} P\left( \sigma^A_k T \right)\\
    	& = \sum_{P_0 \in \sT_2} \sum_{k=0}^{w_b} \frac{(-1)^{P}}{H(P)} P\left( \sigma^A_k P^{-1} \left( P (T)\right) \right)\\
        & = \sum_{P_0 \in \sT_2} \sum_{k=0}^{w_b} \frac{(-1)^{P}}{H(P)} \ytableaushort{{\alpha^P_1}} \otimes \tau^A_k T_P.
\end{align*}
Then as \(P(A) \subset [b]\) and \(\vert P(A) \vert = w_b + 1\), by the proof of Lemma \ref{Generating Garnir Relations and Tools for Collapsing Sums subsection: Simplifying a Garnir} we have, mod \(F_1 \otimes \mR_{\lambda \setminus X, n}\), 
\[
    \sum_{P_0 \in \sT_2} \sum_{k=0}^{w_b} \frac{(-1)^{P}}{H(P)} \ytableaushort{{\alpha^P_1}} \otimes \tau^A_k T_P 
    = \sum_{P_0 \in \sT_2} \dfrac{1}{C_A} \frac{(-1)^{P}}{H(P)} \ytableaushort{{\alpha^P_1}} \otimes G_{P(A)} T_P
    = 0.
\]
\end{proof}

\begin{remark*}
    Notice that the proofs of Case (\ref{Showing the Pieri Inclusion Removing One Box is a GL(V)-map subsection: One Box Removal Preserves Garnirs for Hooks Contained in a Single Block equation: T1}) and Case (\ref{Showing the Pieri Inclusion Removing One Box is a GL(V)-map subsection: One Box Removal Preserves Garnirs for Hooks Contained in a Single Block equation: T2}) did not depend on removing a single box nor on \(A\) being contained in a single block, and so this will generalize to \(m \ge 1\) for both options of a hook \(A\).
\end{remark*}

\begin{case*}[\ref{Showing the Pieri Inclusion Removing One Box is a GL(V)-map subsection: One Box Removal Preserves Garnirs for Hooks Contained in a Single Block equation: T3}]
In this case we show that the sum over all paths that hit \(A\) and move the entry \(A_i\) above block \(b\) is in \(F_1 \otimes \mR_{\lambda \setminus X,n}\).
We will assume that \(b > b_1\), as the case \(b \le b_1\) can be treated similarly. 
It is enough to show that for each \(i = 0,\ldots,w_b\),
\[
    \sum_{P_k \in \sT_3^i} \frac{(-1)^{P}}{H(P)} P_k \in F_1 \otimes \mR_{\lambda \setminus X,n}.
\]
We will show the case \(i = 0\), with the cases \(i = 1, \ldots, w_b\) being similar.
\end{case*}

\begin{proof}
For the rest of Case (\ref{Showing the Pieri Inclusion Removing One Box is a GL(V)-map subsection: One Box Removal Preserves Garnirs for Hooks Contained in a Single Block equation: T3}) let \(\sT := \sT_3^0\) and, for any \(1\)-path \(P\), let \(\tilde{h^P} := h^P - h_b^P\). Define the relation \(\sim\) on \(\sT\) by 
\[
    P_k \sim Q_j \iff Q \text{ is a } ([b](1),[b](i_0 + 1)) \text{-path extension of } P.
\]
It is clear that this defines an equivalence relation on \(\sT\), so that %
\[
    \sum_{P_k \in \sT} \frac{(-1)^{P}}{H(P)} P_k = \sum_{\left[P_k\right] \in \sT / \sim} ~ \sum_{Q_k \in %
    \left[P_k\right]} \frac{(-1)^{Q}}{H(Q)} Q_k.
\]
Pick \(P_0 \in \sT\) with \([b](i,1) \in R^P\) for all \(i=1, \ldots, i_0\), and let \([b_u](i_u,j_u) = P^{-1}\left([b](1,1)\right)\), with \(u := T_{[b_u](i_u,j_u)}\).

\begin{center}
\begin{tikzpicture}
    \node at (-1,1.5) {\(P_0 = \)};
    
    \draw (0,0) rectangle (2,3);
    \draw (0,-.25) -- (0,0);
    \draw (0,3) -- (0,3.25);
    \draw (1.5,-.25) -- (1.5,0);
    \draw (2,3) -- (2.5,3) -- (2.5,3.25);
    
    \draw (0,1.5) -- (.5,1.5);
    \draw (.5,0) -- (.5,2);
    \node at (.25,1.75) {\small \(A_0\)};
    
    \draw (0,2) rectangle (.5,2.5);
    \node at (.25,2.25) {\small \(A_1\)};
    \draw (.5,2) rectangle (1.5,2.5);
    \node at (1,2.25) {\small \(\cdots\)};
    \draw (1.5,2) rectangle (2,2.5);
    \node at (1.75,2.25) {\small \(A_w\)};
    
    \node at (1,-.5) {\small \(u\)};
    
    \draw[blue] (.85,-.5) to [out = 90, in = 0] (.2,.2); 
    \draw[blue] (.2,.2) to [out = 200, in = 180] (.2,.4);
    \draw[blue] (.2,.4) to [out = 200, in = 180] (.2,.6);
    \draw[blue] (.2,.6) to [out = 200, in = 180] (.2,.8);
    \draw[blue] (.2,.8) to [out = 200, in = 180] (.2,1);
    \draw[blue] (.2,1) to [out = 200, in = 180] (.2,1.2);
    \draw[blue] (.2,1.2) to [out = 200, in = 180] (.2,1.4);
    \draw[blue] (.2,1.4) to [out = 200, in = 180] (.1,1.75); 
    \draw[blue, ->] (.1,1.75) to [out = 180, in = 270] (-.5,3); 
\end{tikzpicture}
\end{center}

It is then enough to show that
\[
    \sum_{Q_k \in \left[ P_0 \right]} \frac{(-1)^{Q}}{H(Q)} Q_k \in F_1 \otimes \mR_{\lambda \setminus X,n}.
\]
In fact, as \(\tilde{h^Q} = \tilde{h^P}\) and \(H(Q) = H(P)\) for all \(Q_k \in \left[P_0\right]\), it is enough to show that
\[
    \sum_{Q_k \in [P_0]} (-1)^{h^Q_b} Q_k \in F_1 \otimes \mR_{\lambda \setminus X, n}.
\]

Observe that \([P_0]\) can be written as the disjoint union%
\[
    [P_0] = \bigsqcup_{i = 1}^3 [P_0]_i,
\]
where the \([P_0]_i\) are defined as follows.\\

The paths acting on \(\sigma_0^A T\),
\begin{equation*}
    \left[P_0\right]_1 = \{Q_0 \in \left[P_0\right]\},
\end{equation*}
\begin{center}
    \begin{tikzpicture}
        \draw (0,0) rectangle (2,3);
        \draw (0,-.25) -- (0,0);
        \draw (0,3) -- (0,3.25);
        \draw (1.5,-.25) -- (1.5,0);
        \draw (2,3) -- (2.5,3) -- (2.5,3.25);
        
        \draw (0,1.5) rectangle (.5,2);
        \node at (.25,1.75) {\small \(A_0\)};
        
        \draw (0,2) rectangle (.5,2.5);
        \node at (.25,2.25) {\small \(A_1\)};
        \draw (.5,2) rectangle (1.5,2.5);
        \node at (1,2.25) {\small \(\cdots\)};
        \draw (1.5,2) rectangle (2,2.5);
        \node at (1.75,2.25) {\small \(A_w\)};
        
        \node at (1,-.5) {\small \(u\)};
        
        \draw[blue] (.85,-.5) to [out = 200, in = 180] (1,.2); 
        \draw[blue] (1,.2) to [out = 200, in = 180] (1,.4);
        \draw[blue] (1,.4) to [out = 200, in = 180] (1,.6);
        \draw[blue] (1,.6) to [out = 200, in = 180] (1,.8);
        \draw[blue] (1,.8) to [out = 200, in = 180] (1,1);
        \draw[blue] (1,1) to [out = 200, in = 180] (1,1.2);
        \draw[blue] (1,1.2) to [out = 200, in = 180] (1,1.4);
        \draw[blue] (1,1.4) to [out = 90, in = 0] (.4,1.75); 
        \draw[blue, ->] (.1,1.75) to [out = 180, in = 270] (-.5,3); 
    \end{tikzpicture}
\end{center}
the paths acting on \(\sigma_k^AT\), \(k \ne 0\), that hit \(a_0 = \sigma_k^a a_k\),
\begin{equation*}
    \left[P_0\right]_2 = \{Q_k \in \left[P_0\right]  :  k \ne 0, a_0 \in R^Q\},    
\end{equation*}
\begin{center}
    \begin{tikzpicture}
        \draw (0,0) rectangle (2,3);
        \draw (0,-.25) -- (0,0);
        \draw (0,3) -- (0,3.25);
        \draw (1.5,-.25) -- (1.5,0);
        \draw (2,3) -- (2.5,3) -- (2.5,3.25);
        
        \draw (0,1.5) rectangle (.5,2);
        \node at (.25,1.75) {\small \(A_k\)};
        
        \draw (0,2) rectangle (2,2.5);
        \draw (.75,2) -- (.75,2.5);
        \draw (1.25,2) -- (1.25,2.5);
        \node at (1,2.25) {\small \(A_0\)};
        \node at (1,2.7) {\small \(k\)};
        
        \node at (1,-.5) {\small \(u\)};
        
        \draw[blue] (.85,-.5) to [out = 200, in = 180] (1,.2); 
        \draw[blue] (1,.2) to [out = 200, in = 180] (1,.4);
        \draw[blue] (1,.4) to [out = 200, in = 180] (1,.6);
        \draw[blue] (1,.6) to [out = 200, in = 180] (1,.8);
        \draw[blue] (1,.8) to [out = 200, in = 180] (1,1);
        \draw[blue] (1,1) to [out = 200, in = 180] (1,1.2);
        \draw[blue] (1,1.2) to [out = 180, in = 270] (.25,1.75); 
        \draw[blue] (.25,1.9) to [out = 90, in = 180] (.6,2.25); 
        \draw[blue, ->] (.85,2.25) to [out = 180, in = 270] (-.5,3); 
    \end{tikzpicture}
\end{center}
and the paths acting on \(\sigma_k^AT\), \(k \ne 0\), that miss \(a_0 = \sigma_k^a a_k\),
\begin{equation*}
    \left[P_0\right]_3 = \{Q_k \in \left[P_0\right]  :  k \ne 0, a_0 \not \in R^Q\},
\end{equation*}
\begin{center}
    \begin{tikzpicture} 
        \draw (0,0) rectangle (2,3);
        \draw (0,-.25) -- (0,0);
        \draw (0,3) -- (0,3.25);
        \draw (1.5,-.25) -- (1.5,0);
        \draw (2,3) -- (2.5,3) -- (2.5,3.25);
        
        \draw (0,1.5) -- (2,1.5);
        \draw (.5,1.5) -- (.5,2);
        \node at (.25,1.75) {\small \(A_k\)};
        
        \draw (0,2) rectangle (2,2.5);
        \draw (.75,2) -- (.75,2.5);
        \draw (1.25,2) -- (1.25,2.5);
        \node at (1,2.25) {\small \(A_0\)};
        \node at (1,2.7) {\small \(k\)};
        
        \node at (1,-.5) {\small \(u\)};
        
        \draw[blue] (.85,-.5) to [out = 200, in = 180] (1,.2); 
        \draw[blue] (1,.2) to [out = 200, in = 180] (1,.4);
        \draw[blue] (1,.4) to [out = 200, in = 180] (1,.6);
        \draw[blue] (1,.6) to [out = 200, in = 180] (1,.8);
        \draw[blue] (1,.8) to [out = 200, in = 180] (1,1);
        \draw[blue] (1,1) to [out = 200, in = 180] (1,1.2);
        \draw[blue] (1,1.2) to [out = 200, in = 180] (1,1.4);
        \draw[blue] (1,1.4) to [out = 200, in = 180] (1,1.75); 
        \draw[blue] (1,1.75) to [out = 200, in = 180] (.85,2.25); 
        \draw[blue, ->] (.85,2.25) to [out = 180, in = 270] (-.5,3); 
    \end{tikzpicture}
\end{center}

Let \(T^\prime \in \mF_{\lambda \setminus X}\) be the unique tableau with \(T^\prime = T_P\) on \((\lambda \setminus X) \setminus [b]\) and \(T^\prime = T\) on \([b]\) except \(T^\prime_{a_0} = u\).

\begin{center}
\begin{tikzpicture}
    \node at (-1, 1.5) {\(T^\prime = \)};
    
    \draw (0,0) rectangle (2,3);
    \draw (0,-.25) -- (0,0);
    \draw (0,3) -- (0,3.25);
    \draw (1.5,-.25) -- (1.5,0);
    \draw (2,3) -- (2.5,3) -- (2.5,3.25);
    
    \draw (0,1.5) rectangle (.5,2);
    \node at (.25,1.75) {\small \(u\)};
    
    \draw (0,2) rectangle (.5,2.5);
    \node at (.25,2.25) {\small \(A_1\)};
    \draw (.5,2) rectangle (1.5,2.5);
    \node at (1,2.25) {\small \(\cdots\)};
    \draw (1.5,2) rectangle (2,2.5);
    \node at (1.75,2.25) {\small \(A_w\)};
    
    \node at (-.5,3) { \ };
    \node at (1,-.5) { \ };
\end{tikzpicture}
\end{center}

Then by Lemma \ref{Generating Garnir Relations and Tools for Collapsing Sums subsection: Calculation Lemma for 1-paths} and applications of \(G_A\), we have, mod \(F_1 \otimes \mR_{\lambda \setminus X,n}\), %

\begin{align*}
\sum_{Q_0 \in \left[P_0\right]_1} (-1)^{h_b^Q} Q_0 %
	& = (-1)^{i_0 + i_0 - 1} ~ \ytableaushort{{\alpha^P_1}} \otimes T^\prime
	  = - ~ \ytableaushort{{\alpha^P_1}} \otimes T^\prime,\\
\sum_{Q_k \in \left[P_0\right]_2} (-1)^{h_b^Q} Q_k %
	& = (-1)^{i_0 + 1 + i_0 - 1} w_b ~ \ytableaushort{{\alpha^P_1}} \otimes T^\prime 
      = w_b ~ \ytableaushort{{\alpha^P_1}} \otimes T^\prime,\\
\sum_{Q_k \in \left[P_0\right]_3} (-1)^{h_b^Q} Q_k %
	& = (-1)^{i_0 + 1 + 1 + i_0 - 1} (w_b - 1) ~ \ytableaushort{{\alpha^P_1}} \otimes T^\prime 
      = - (w_b - 1) ~ \ytableaushort{{\alpha^P_1}} \otimes T^\prime.
\end{align*}

Now as 
\[
    \sum_{Q_k \in \left[P_0\right]} (-1)^{h^Q_b} Q_k = \sum_{i=1}^3 \sum_{Q_k \in \left[P_0\right]_i} (-1)^{h^Q_b} Q_k,
\]
we have, mod \(F_1 \otimes \mR_{\lambda \setminus X,n}\),
\[
    \sum_{Q_k \in \left[P_0\right]} (-1)^{h^Q_b} Q_k %
        = (-1 + w_b - w_b + 1) ~ \ytableaushort{{\alpha^P_1}} \otimes T^\prime 
        = 0.
\]
\end{proof}

\subsection{One Box Removal Preserves Garnirs for Hooks Contained in a Two Blocks} \label{Showing the Pieri Inclusion Removing One Box is a GL(V)-map subsection: One Box Removal Preserves Garnirs for Hooks Contained in a Two Blocks}

We now show that Equation \ref{Showing the Pieri Inclusion Removing One Box is a GL(V)-map subsection: The Theorem equation: Phi Preserves Garnirs One Box} holds for all hooks \(A \subset [b] \cup [b + 1]\) for some \(1 \le b \le N - 1\).
For the rest of Section \ref{Showing the Pieri Inclusion Removing One Box is a GL(V)-map subsection: One Box Removal Preserves Garnirs for Hooks Contained in a Two Blocks}, fix \(T \in \sT_{\lambda, n}\) and
\[
    A = \{a_0 := [b](h_b,1), a_1 := [b + 1](1,1), \ldots, a_{w_{b + 1}} := [b + 1](1,w_{b + 1})\} \subset T_0,   
\]
so that \(A \subset [b] \cup [b + 1]\). 
Denote the entries of \(A\) in \(T\) by \(A_k = T_{a_k}\) for \(k = 0, 1, \ldots, w_{b + 1}\). 
Then, by Lemma \ref{Generating Garnir Relations and Tools for Collapsing Sums subsection: Simplifying a Garnir}, mod \(F_1 \otimes \mR_{\lambda \setminus X, n}\) we have
\[
    \Phi_1\left(G_A (T)\right) 
    = \sum_{P} \frac{(-1)^{P}}{H(P)} P\left(\sum_{\sigma \in S_A} \sigma T \right)
    = C \sum_{P} \sum_{k=0}^{w_{b + 1}} \frac{(-1)^{P}}{H(P)} P\left(\sigma^A_k T \right),
\]
where the sum is over all \(1\)-paths \(P\) on \(\lambda\) removing \(X\).
The set of all \(P_k := P(\sigma^A_k T)\) appearing in the image \(\Phi_1\left(G_A (T)\right)\) above is the union of the following disjoint sets.\\

The \(P_k\)s that miss \(A\),
\begin{equation} \label{Showing the Pieri Inclusion Removing One Box is a GL(V)-map subsection: One Box Removal Preserves Garnirs for Hooks Contained in a Two Blocks equation: T1}
    \sT_1 = \{P_k : R^P \cap A = \emptyset\}.\\
\end{equation}

The \(P_k\)s that hit \(A\) and keep \(A\) in blocks \(b\) and \(b + 1\),
\begin{equation} \label{Showing the Pieri Inclusion Removing One Box is a GL(V)-map subsection: One Box Removal Preserves Garnirs for Hooks Contained in a Two Blocks equation: T2}
    \sT_2 = \{P_k  :  R^P \cap A \ne \emptyset, ~ P(A) \le [b + 1]\}.\\
\end{equation}

The \(P_k\)s that hit \(A\) and move the entry \(A_i\) above block \(b + 1\), including \(P(\sigma_k^A A_i) \in Y\),
\begin{equation} \label{Showing the Pieri Inclusion Removing One Box is a GL(V)-map subsection: One Box Removal Preserves Garnirs for Hooks Contained in a Two Blocks equation: T3}
    \sT_3 = \bigsqcup_{i = 0}^{w_{b + 1}} \sT_3^i \quad \text{ where } \quad \sT_3^i = \{P_k \in \sT_3 : R^P \cap A \ne \emptyset, ~ P(\sigma^A_k A_i) > [b + 1]\}.\\
\end{equation}

Then we have, mod \(F_1 \otimes \mR_{\lambda \setminus X, n}\),
\[
\Phi_1(G_A (T)) %
	= C \sum_{P} \sum_{k=0}^{w_{b + 1}} \frac{(-1)^{P}}{H(P)} P\left(\sigma^A_k T \right) %
    = C \sum_{j=1}^3 \sum_{P_k \in \sT_j} \frac{(-1)^{P}}{H(P)} P_k.
\]
We show that for each of the cases (\ref{Showing the Pieri Inclusion Removing One Box is a GL(V)-map subsection: One Box Removal Preserves Garnirs for Hooks Contained in a Two Blocks equation: T1}) - (\ref{Showing the Pieri Inclusion Removing One Box is a GL(V)-map subsection: One Box Removal Preserves Garnirs for Hooks Contained in a Two Blocks equation: T3}),
\[
    \sum_{P_k \in T_j} \frac{(-1)^{P}}{H(P)} P_k \in F_1 \otimes \mR_{\lambda \setminus X,n},
\]
and hence Equation \ref{Showing the Pieri Inclusion Removing One Box is a GL(V)-map subsection: The Theorem equation: Phi Preserves Garnirs One Box} holds for all hooks \(A \subset [b] \cup [b + 1]\).

Case (\ref{Showing the Pieri Inclusion Removing One Box is a GL(V)-map subsection: One Box Removal Preserves Garnirs for Hooks Contained in a Two Blocks equation: T1}) follows from the proof of Case (\ref{Showing the Pieri Inclusion Removing One Box is a GL(V)-map subsection: One Box Removal Preserves Garnirs for Hooks Contained in a Single Block equation: T1}) and Case (\ref{Showing the Pieri Inclusion Removing One Box is a GL(V)-map subsection: One Box Removal Preserves Garnirs for Hooks Contained in a Two Blocks equation: T2}) follows from the proof of Case (\ref{Showing the Pieri Inclusion Removing One Box is a GL(V)-map subsection: One Box Removal Preserves Garnirs for Hooks Contained in a Single Block equation: T2}). It remains to show Case (\ref{Showing the Pieri Inclusion Removing One Box is a GL(V)-map subsection: One Box Removal Preserves Garnirs for Hooks Contained in a Two Blocks equation: T3}).

\begin{case*}[\ref{Showing the Pieri Inclusion Removing One Box is a GL(V)-map subsection: One Box Removal Preserves Garnirs for Hooks Contained in a Two Blocks equation: T3}]
In this case we show that the sum over all paths that  hit \(A\) and move the entry \(A_i\) above block \(b + 1\) is in \(F_1 \otimes \mR_{\lambda \setminus X,n}\).
It is enough to show that for \(i = 0, \ldots, w_{b + 1}\),
\[
    \sum_{P_k \in \sT_3^i} \frac{(-1)^{P}}{H(P)} P_k \in F_1 \otimes \mR_{\lambda \setminus X,n}.
\]
We will show the case \(i = 0\), with the cases \(i = 1, \ldots, w_{b + 1}\) being similar.
\end{case*}

\begin{proof} 
Note that as we are considering paths that hit \(A\) in this case, we must have that \(b \ge b_1 - 1\). We will consider the case \(b = b_1 - 1\) (and hence \(a_{w_{b + 1}} = x_1\)) and the case \(b > b_1 - 1\) separately.

\begin{subcase*}[\ref{Showing the Pieri Inclusion Removing One Box is a GL(V)-map subsection: One Box Removal Preserves Garnirs for Hooks Contained in a Two Blocks equation: T3}.1]
    We first show the case where \(b = b_1 - 1\).
    We want to show that
    \[
        \sum_{P_k \in \sT_3^0} \frac{(-1)^{P}}{H(P)} P_k \in F_1 \otimes \mR_{\lambda \setminus X,n}.
    \]
    As \(a_0 < x_1\), for all \(P_k \in \sT_3^0\) we must have \(k \ne 0\).
    We can then write \(\sT_3^0\) as
    \[
        \sT_3^0 = \bigsqcup_{P_1 \in \sT_3^0} \sT_{P_1}
    \]
    where
    \[
        \sT_{P_1} = \{Q_k = \in \sT_3^0 : Q \text{ is a } ([b+1](1), [b+1](1))\text{-path extension of } P\}.
    \]
    It is then enough to show that for each \(P_1 \in \sT_3^0\),
    \[
        \sum_{k = 1}^{w_{b + 1}} \frac{(-1)^{P}}{H(P)} P_1(\sigma_k^A T) \in F_1 \otimes \mR_{\lambda \setminus X,n}.
    \]
    For the rest of Subcase (\ref{Showing the Pieri Inclusion Removing One Box is a GL(V)-map subsection: One Box Removal Preserves Garnirs for Hooks Contained in a Two Blocks equation: T3}.1), fix a \(P_1 \in \sT_3^0\). We will show that
    \[
        \sum_{k = 1}^{w_{b + 1}} P_1(\sigma_k^A T) \in F_1 \otimes \mR_{\lambda \setminus X,n}.
    \]
    Let \(A^\prime = A \setminus \{a_{w_b + 1}\} \subset \lambda \setminus X\). As \(\vert A^\prime \vert = w_{b + 1} > w_{b + 1} - 1\), by the proof of Lemma \ref{Generating Garnir Relations and Tools for Collapsing Sums subsection: Simplifying a Garnir} we have
    \begin{align*}
        \sum_{k = 1}^{w_{b + 1}} P_1(\sigma_k^A T)
            &= \sum_{k = 1}^{w_{b + 1}} \ytableaushort{{\alpha^P_1}} \otimes \sigma_k^A \left( T_P \right)\\
            &= \frac{1}{C_A} \ytableaushort{{\alpha^P_1}} \otimes G_{A^\prime} \left( T_P \right) \in F_1 \otimes \mR_{\lambda \setminus X,n}.
    \end{align*}
\end{subcase*}
\begin{subcase*}[\ref{Showing the Pieri Inclusion Removing One Box is a GL(V)-map subsection: One Box Removal Preserves Garnirs for Hooks Contained in a Two Blocks equation: T3}.2]
We now show the case \(b > b_1\).
For the rest of Subcase (\ref{Showing the Pieri Inclusion Removing One Box is a GL(V)-map subsection: One Box Removal Preserves Garnirs for Hooks Contained in a Two Blocks equation: T3}.2), let \(\sT := \sT_3^0\) and, for  any \(1\)-path \(P\), define \(\tilde{h^P} := h^P - h_b^P\) and 
\[
    \tilde{H(P)} = \dfrac{H(P)}{H_b(P)H_{b + 1}(P)}.
\]
Define the relation \(\sim\) on \(\sT\) by 
\[
    P_k \sim Q_j \iff Q_j \text{ is a } ([b](1),[b + 1](1)) \text{-path extension of } P.
\]
It is clear that this defines an equivalence relation on \(\sT\), so that %
\[
    \sum_{P_k \in \sT} \frac{(-1)^{P}}{H(P)} P_k = \sum_{\left[P_k\right] \in \sT / \sim} ~ \sum_{Q_k \in %
    \left[P_k\right]} \frac{(-1)^{Q}}{H(Q)} Q_k.
\]

Pick \(P_0 \in \sT\) with \([b](i,1) \in R^P\) for all \(i=1, \ldots, h_b\), and let \([b_u](i_u,j_u) = P^{-1}([b](1,1))\) with \(u := T_{[b_u](i_u,j_u)}\). 

\begin{center}
\begin{tikzpicture}
    \node at (-1, 1.5) {\(P_0 = \)};
    
    \draw (0,0) rectangle (1.25,2);
    \draw (0,-.25) -- (0,0);
    \draw (.75,-.25) -- (.75,0);
    
    \draw (0,2) rectangle (2,3);
    \draw (0,3) -- (0,3.25);
    \draw (2,3) -- (2.25,3) -- (2.25,3.25);
    
    \draw (0,1.5) -- (.5,1.5);
    \draw (.5,0) -- (.5,2);
    \node at (.25,1.75) {\small \(A_0\)};
    
    \draw (0,2.5) -- (2,2.5);
    \draw (.5,2) -- (.5,2.5);
    \draw (1.5,2) -- (1.5,2.5);
    \node at (.25,2.25) {\small \(A_1\)};
    \node at (1,2.25) {\small \(\cdots\)};
    \node at (1.75,2.25) {\small \(A_w\)};
    
    \node at (.5,-.5) {\(u\)};
    
    \draw[blue] (.5,-.4) to [out = 90, in = 0] (.2,.2); 
    \draw[blue] (.2,.2) to [out = 200, in = 180] (.2,.4);
    \draw[blue] (.2,.4) to [out = 200, in = 180] (.2,.6);
    \draw[blue] (.2,.6) to [out = 200, in = 180] (.2,.8);
    \draw[blue] (.2,.8) to [out = 200, in = 180] (.2,1);
    \draw[blue] (.2,1) to [out = 200, in = 180] (.2,1.2);
    \draw[blue] (.2,1.2) to [out = 200, in = 180] (.2,1.4);
    \draw[blue] (.2,1.4) to [out = 200, in = 240] (.1,1.75); 
    \draw[blue, ->] (.1,1.75) to [out = 180, in = 270] (-.5,3); 
\end{tikzpicture}
\end{center}

It is then enough to show that
\[
    \sum_{Q_k \in \left[P_0\right]} \frac{(-1)^{Q}}{H(Q)} Q_k \in %
    F_1 \otimes \mR_{\lambda \setminus X, n}.
\]
In fact, as \(\tilde{h^Q} = \tilde{h^P}\) and \(\tilde{H(Q)} = \tilde{H(P)}\) for all \(Q_k \in \left[P_0\right]\), it is enough to show that
\[
    \sum_{Q_k \in \left[P_0\right]} \dfrac{(-1)^{h^Q_b}}{H_b(Q) H_{b + 1}(Q)} Q_k \in %
    F_1 \otimes \mR_{\lambda \setminus X, n}.
\]

Observe that \([P_0]\) can be written as the disjoint union
\[
    \sT = \bigsqcup_{i = 1}^6 [P_0]_i,
\]
where the \([P_0]_i\) are defined as follows.\\

The paths acting on \(\sigma_0^AT\),
\begin{equation*}
    \left[P_0\right]_1 = \{Q_0 \in \left[P_0\right]\},
\end{equation*}
\begin{center}
	\begin{tikzpicture}
        \draw (0,0) rectangle (1.25,2);
        \draw (0,-.25) -- (0,0);
        \draw (.75,-.25) -- (.75,0);
        
        \draw (0,2) rectangle (2,3);
        \draw (0,3) -- (0,3.25);
        \draw (2,3) -- (2.25,3) -- (2.25,3.25);
        
        \draw (0,1.5) -- (.5,1.5) -- (.5,2);
        \node at (.25,1.75) {\small \(A_0\)};
        
        \draw (0,2.5) -- (2,2.5);
        \draw (.5,2) -- (.5,2.5);
        \draw (1.5,2) -- (1.5,2.5);
        \node at (.25,2.25) {\small \(A_1\)};
        \node at (1,2.25) {\small \(\cdots\)};
        \node at (1.75,2.25) {\small \(A_w\)};
        
        \node at (.5,-.5) {\(u\)};
        
        \draw[blue] (.35,-.5) to [out = 160, in = 180] (.5,.2); 
        \draw[blue] (.5,.2) to [out = 200, in = 180] (.5,.4);
        \draw[blue] (.5,.4) to [out = 200, in = 180] (.5,.6);
        \draw[blue] (.5,.6) to [out = 200, in = 180] (.5,.8);
        \draw[blue] (.5,.8) to [out = 200, in = 180] (.5,1);
        \draw[blue] (.5,1) to [out = 200, in = 180] (.5,1.2);
        \draw[blue] (.5,1.2) to [out = 200, in = 180] (.5,1.4);
        \draw[blue] (.5,1.4) to [out = 200, in = 240] (.1,1.75); 
        \draw[blue, ->] (.1,1.75) to [out = 180, in = 270] (-.5,3); 
    \end{tikzpicture}
\end{center}
The paths acting on \(\sigma_k^AT\), \(k \ne 0\), that miss block \(b\),
\begin{equation*}
    \left[P_0\right]_2 = \{Q_k \in \left[P_0\right]  :  k \ne 0, R^Q \cap [b] = \emptyset\},
\end{equation*}
\begin{center}
        \begin{tikzpicture}
        \draw (0,0) rectangle (1.25,2);
        \draw (0,-.25) -- (0,0);
        \draw (.75,-.25) -- (.75,0);
        
        \draw (0,2) rectangle (2,3);
        \draw (0,3) -- (0,3.25);
        \draw (2,3) -- (2.25,3) -- (2.25,3.25);
        
        \draw (0,1.5) -- (.5,1.5) -- (.5,2);
        \node at (.25,1.75) {\small \(A_k\)};
        
        \draw (0,2.5) -- (2,2.5);
        \draw (.75,2) -- (.75,2.5);
        \draw (1.25,2) -- (1.25,2.5);
        \node at (1,2.25) {\small \(A_0\)};
        \node at (1,2.7) {\small \(k\)};
        
        \node at (.5,-.5) {\(u\)};
        
        \draw[blue] (.35,-.5) to [out = 160, in = 180] (.8,2.25); 
        \draw[blue, ->] (.8,2.25) to [out = 180, in = 270] (-.5,3); 
    \end{tikzpicture}
\end{center}
the paths acting on \(\sigma_k^AT\), \(k \ne 0\), that hit \(a_0 = \sigma_k^a a_k\),
\begin{equation*}
    \left[P_0\right]_3 = \{Q_k \in \left[P_0\right]  :  k \ne 0, a_0 \in R^Q\},
\end{equation*}
\begin{center}
    \begin{tikzpicture}
        \draw (0,0) rectangle (1.25,2);
        \draw (0,-.25) -- (0,0);
        \draw (.75,-.25) -- (.75,0);
        
        \draw (0,2) rectangle (2,3);
        \draw (0,3) -- (0,3.25);
        \draw (2,3) -- (2.25,3) -- (2.25,3.25);
        
        \draw (0,1.5) -- (.5,1.5) -- (.5,2);
        \node at (.25,1.75) {\small \(A_k\)};
        
        \draw (0,2.5) -- (2,2.5);
        \draw (.75,2) -- (.75,2.5);
        \draw (1.25,2) -- (1.25,2.5);
        \node at (1,2.25) {\small \(A_0\)};
        \node at (1,2.7) {\small \(k\)};
        
        \node at (.5,-.5) {\(u\)};
        
        \draw[blue] (.35,-.5) to [out = 160, in = 180] (.5,.2); 
        \draw[blue] (.5,.2) to [out = 200, in = 180] (.5,.4);
        \draw[blue] (.5,.4) to [out = 200, in = 180] (.5,.6);
        \draw[blue] (.5,.6) to [out = 200, in = 180] (.5,.8);
        \draw[blue] (.5,.8) to [out = 200, in = 180] (.5,1);
        \draw[blue] (.5,1) to [out = 200, in = 180] (.5,1.2);
        \draw[blue] (.5,1.2) to [out = 200, in = 270] (.25,1.6); 
        \draw[blue] (.25,1.9) to [out = 90, in = 180] (.8,2.25); 
        \draw[blue, ->] (.8,2.25) to [out = 180, in = 270] (-.5,3); 
    \end{tikzpicture}
\end{center}
the paths acting on \(\sigma_k^AT\), \(k \ne 0\), that miss \(a_0 = \sigma_k^a a_k\) but hit row \([b](h_b)\),
\begin{equation*}
    \left[P_0\right]_4 = \{Q_k \in \left[P_0\right]  :  k \ne 0, [b](h_b,j) \in R^Q \text{ for some } 2 \le j \le w_b\},
\end{equation*}
\begin{center}
    \begin{tikzpicture}
        \draw (0,0) rectangle (1.25,2);
        \draw (0,-.25) -- (0,0);
        \draw (.75,-.25) -- (.75,0);
        
        \draw (0,2) rectangle (2,3);
        \draw (0,3) -- (0,3.25);
        \draw (2,3) -- (2.25,3) -- (2.25,3.25);
        
        \draw (0,1.5) -- (1.25,1.5);
        \draw (.5,1.5) -- (.5,2);
        \node at (.25,1.75) {\small \(A_k\)};
        
        \draw (0,2.5) -- (2,2.5);
        \draw (.75,2) -- (.75,2.5);
        \draw (1.25,2) -- (1.25,2.5);
        \node at (1,2.25) {\small \(A_0\)};
        \node at (1,2.7) {\small \(k\)};
        
        \node at (.5,-.5) {\(u\)};
        
        \draw[blue] (.35,-.5) to [out = 200, in = 180] (.5,.2); 
        \draw[blue] (.5,.2) to [out = 200, in = 180] (.5,.4);
        \draw[blue] (.5,.4) to [out = 200, in = 180] (.5,.6);
        \draw[blue] (.5,.6) to [out = 200, in = 180] (.5,.8);
        \draw[blue] (.5,.8) to [out = 200, in = 180] (.5,1);
        \draw[blue] (.5,1) to [out = 200, in = 180] (.5,1.2);
        \draw[blue] (.5,1.2) to [out = 90, in = 180] (.8,1.75); 
        \draw[blue] (.8,1.75) to [out = 200, in = 180] (.8,2.25); 
        \draw[blue, ->] (.8,2.25) to [out = 180, in = 270] (-.5,3); 
    \end{tikzpicture}
\end{center}
the paths acting on \(\sigma_k^AT\), \(k \ne 0\), that miss row \([b](h_b)\) and leave block \(b\) from an odd row,
\begin{equation*}
    \left[P_0\right]_5 = \{Q_k \in E^P : k \ne 0, Q([b](i,j)) = a_k \text{ for some } 1 \le j \le w_b, 1 \le i < h_b, i \text{ odd}\}
\end{equation*}
and the paths acting on \(\sigma_k^AT\), \(k \ne 0\), that miss row \([b](h_b)\) and leave block \(b\) from an even row,
\begin{equation*}
    \left[P_0\right]_6 = \{Q_k \in E^P  :  k \ne 0, Q([b](i,j)) = a_k \text{ for some } 1 \le j \le w_b, 1 \le i < h_b, i \text{ even}\}.
\end{equation*}
\begin{center}
    \begin{tikzpicture}
        \draw (0,0) rectangle (1.25,2);
        \draw (0,-.25) -- (0,0);
        \draw (.75,-.25) -- (.75,0);
        
        \draw (0,2) rectangle (2,3);
        \draw (0,3) -- (0,3.25);
        \draw (2,3) -- (2.25,3) -- (2.25,3.25);
    
        \draw (0,1.5) -- (.5,1.5) -- (.5,2);
        \node at (.25,1.75) {\small \(A_k\)};
        
        \draw (0,2.5) -- (2,2.5);
        \draw (.75,2) -- (.75,2.5);
        \draw (1.25,2) -- (1.25,2.5);
        \node at (1,2.25) {\small \(A_0\)};
        \node at (1,2.7) {\small \(k\)};
        
        \draw (.5,.7) rectangle (.75,.95);
        \node at (1.5,.85) {\(\leftarrow\) row \(i\)};
        
        \node at (.5,-.5) {\(u\)};
    
        \draw[blue] (.35,-.5) to [out = 200, in = 180] (.5,.2); 
        \draw[blue] (.5,.2) to [out = 200, in = 180] (.5,.4);
        \draw[blue] (.5,.4) to [out = 200, in = 180] (.5,.6);
        \draw[blue] (.5,.6) to [out = 200, in = 180] (.625,.8125); 
        \draw[blue] (.625,.8125) to [out = 90, in = 180] (.8,2.25); 
        \draw[blue, ->] (.8,2.25) to [out = 180, in = 270] (-.5,3); 
    \end{tikzpicture}
\end{center}

Let \(T^\prime \in \mF_{\lambda \setminus X, n}\) be the unique tableau with \(T^\prime = T_P\) on \((\lambda \setminus X) \setminus \left([b] \cup [b + 1]\right)\) and \(T^\prime = T\) on \([b] \cup [b + 1]\) except \(T^\prime_{a_0} = u\), 

\[
T^\prime = 
\begin{tikzpicture}[baseline=(O.base)]
    \node (O) at (0,1.625) {~};
    
    \draw (0,0) rectangle (1.25,2);
    \draw (0,-.25) -- (0,0);
    \draw (.75,-.25) -- (.75,0);
    
    \draw (0,2) rectangle (2,3);
    \draw (0,3) -- (0,3.25);
    \draw (2,3) -- (2.25,3) -- (2.25,3.25);
    
    \draw (0,1.5) -- (.5,1.5) -- (.5,2);
    \node at (.25,1.75) {\small \(u\)};
    
    \draw (0,2.5) -- (2,2.5);
    \draw (.5,2) -- (.5,2.5);
    \draw (1.5,2) -- (1.5,2.5);
    \node at (.25,2.25) {\small \(A_1\)};
    \node at (1,2.25) {\small \(\cdots\)};
    \node at (1.75,2.25) {\small \(A_w\)};
    
    \node at (-.5,3) { \ };
    \node at (1,-.5) { \ };
\end{tikzpicture}.
\]

By Lemma \ref{Generating Garnir Relations and Tools for Collapsing Sums subsection: Calculation Lemma for 1-paths} and applications of \(G_A\) we have, mod \(F_1 \otimes \mR_{\lambda \setminus X,n}\),
\begin{align*}
\sum_{Q_0 \in \left[P_0\right]_1} (-1)^{h_b^Q} Q_0 
	& = \dfrac{(-1)^{h_b + h_b - 1}}{H(b)} ~ \ytableaushort{{\alpha^P_1}} \otimes T^\prime\\
    & = \dfrac{-H(b + 1)}{H(b)H(b + 1)} ~ \ytableaushort{{\alpha^P_1}} \otimes T^\prime,\\
\end{align*}
\begin{align*}
\sum_{Q_k \in \left[P_0\right]_2} (-1)^{h_b^Q} Q_k 
	& = \dfrac{(-1)^{1 + 1}}{H(b + 1)} ~ \ytableaushort{{\alpha^P_1}} \otimes T^\prime\\
    & = \dfrac{H(b)}{H(b)H(b + 1)} ~ \ytableaushort{{\alpha^P_1}} \otimes T^\prime,\\
\end{align*}
\begin{align*}
\sum_{Q_k \in \left[P_0\right]_3} (-1)^{h_b^Q} Q_k 
	& = \dfrac{(-1)^{h_b + 1 + h_b - 1}w_{b + 1}}{H(b)H(b + 1)} ~ \ytableaushort{{\alpha^P_1}}
	\otimes T^\prime \\
    & = \dfrac{w_{b + 1}}{H(b)H(b + 1)} ~ \ytableaushort{{\alpha^P_1}} \otimes T^\prime,\\
\end{align*}
\begin{align*}
\sum_{Q_k \in \left[P_0\right]_4} (-1)^{h_b^Q} Q_k 
	& = \dfrac{(-1)^{h_b + 1 + 1 + h_b - 1}(w_b - 1)}{H(b)H(b + 1)} ~ \ytableaushort{{\alpha^P_1}} \otimes T^\prime \\
    & = \dfrac{1 - w_b}{H(b)H(b + 1)} ~ \ytableaushort{{\alpha^P_1}} \otimes T^\prime,\\
\end{align*}
\begin{align*}
\sum_{Q_k \in \left[P_0\right]_5} (-1)^{h_b^Q} Q_k 
	& = \sum_{\substack{1 \le i < h_b\\ i \text{ odd}}} \dfrac{(-1)^{i + 1 + 1 + i - 1 + 1}}{H(b)H(b + 1)} ~ 
	\ytableaushort{{\alpha^P_1}} \otimes T^\prime \\
    & = \sum_{\substack{1 \le i < h_b\\ i \text{ odd}}} \dfrac{1}{H(b)H(b + 1)} ~ 
    \ytableaushort{{\alpha^P_1}} \otimes T^\prime,\\
\end{align*}
and
\begin{align*}
\sum_{Q_k \in \left[P_0\right]_5} (-1)^{h_b^Q} Q_k 
	& = \sum_{\substack{1 \le i < h_b\\ i \text{ even}}} \dfrac{(-1)^{i + 1 + 1 + i - 1 + 1}}{H(b)H(b + 1)} ~ 
	\ytableaushort{{\alpha^P_1}} \otimes T^\prime \\
    & = \sum_{\substack{1 \le i < h_b\\ i \text{ even}}} \dfrac{1}{H(b)H(b + 1)} ~ 
    \ytableaushort{{\alpha^P_1}} \otimes T^\prime.\\
\end{align*}

Then as \(H(b + 1) = H(b) + w_{b + 1} - w_b + h_b\) and
\begin{align*}
    \sum_{\substack{1 \le i < h_b\\  i \text{ odd}}} & \dfrac{1}{H(b)H(b + 1)} ~ 
    \ytableaushort{{\alpha^P_1}} \otimes T^\prime 
    + \sum_{\substack{1 \le i < h_b\\ i \text{ even}}}  \dfrac{1}{H(b)H(b + 1)} ~ 
    \ytableaushort{{\alpha^P_1}} \otimes T^\prime \\
    & = \dfrac{h_b - 1}{H(b)H(b + 1)} ~ 
    \ytableaushort{{\alpha^P_1}} \otimes T^\prime 
\end{align*}
we get, mod \(F_1 \otimes \mR_{\lambda \setminus X,n}\),
\begin{align*}
    \sum_{Q_k \in \left[P_0\right]} (-1)^{h^Q_b} Q_k %
        & = \sum_{i=1,\ldots,6} \sum_{Q_k \in \left[P_0\right]_i} (-1)^{h^Q_b} Q_k\\
        & = \dfrac{- H(b + 1) + H(b) + w_{b + 1} + 1 - w_b + h_b - 1}{H(b)H(b + 1)} %
        ~ \ytableaushort{{\alpha^P_1}} \otimes T^\prime \\
        & = 0.
\end{align*}
\end{subcase*}
\end{proof}

Thus Equation \ref{Showing the Pieri Inclusion Removing One Box is a GL(V)-map subsection: The Theorem equation: Phi Preserves Garnirs One Box} holds for all hooks \(A \subset [b] \cup [b + 1]\), and so Theorem \ref{Showing the Pieri Inclusion Removing One Box is a GL(V)-map subsection: The Theorem} holds.

\section{Showing the Pieri Inclusion Removing Many Boxes is a \(GL(V)\)-map} \label{section: Showing the Pieri Inclusion Removing Many Boxes is a GL(V)-map}

\subsection{} \label{Showing the Pieri Inclusion Removing Many Boxes is a GL(V)-map subsection: The Theorem}

For all of section \ref{section: Showing the Pieri Inclusion Removing Many Boxes is a GL(V)-map}, fix a removal set \(X = \{x_1 = [b_1](1,w_{b_1}), \ldots, x_m = [b_m](i_m,w_{b_m})\} \subset \lambda\). Let
\[
    \Phi_m : \mF_{\lambda, n} \to F_m \otimes \mF_{\lambda \setminus X, n}
\]
be as in \ref{Constructing the Pieri Inclusion for Removing Many Boxes subsection: Defining Phi}.

\begin{theorem*}
    \(\Phi_m\) is a \(\GL(V)\)-map, i.e. \(\Phi_m\) descends to 
    \[
        \Phi_m : \bS_{\lambda}(V) \to F_m \otimes \bS_{\lambda \setminus X}(V)
    \]
    and \(\Phi_m\) is \(\GL(V)\)-equivariant.
\end{theorem*}

As before, it is clear that \(\Phi_m\) is \(\gl(V)\)-equivariant by construction.
To prove Theorem \ref{Showing the Pieri Inclusion Removing Many Boxes is a GL(V)-map subsection: The Theorem}, it remains to show that
\[
    \Phi_m(\mR_{\lambda, n}) \subset F_m \otimes \mR_{\lambda \setminus X, n}.
\]

As before, it is clear that \(\Phi_m\) preserves Property \ref{Constructing Schur--Weyl Modules subsection: The Construction equation: R1} as it is a sum over all \(m\)-paths, and hence we must show that Property \ref{Constructing Schur--Weyl Modules subsection: The Construction equation: R2} holds, i.e. for all \(T \in \mF_{\lambda, n}\) and all \(A \subset T_0\) with \(\vert A \vert > w_A\),
\begin{equation} \label{Showing the Pieri Inclusion Removing Many Boxes is a GL(V)-map subsection: The Theorem equation: Phi of Garnir m Boxes}
    \Phi_m \left( G_A(T) \right) \in F_m \otimes \mR_{\lambda \setminus X, n}.
\end{equation}
We will show that Equation \ref{Showing the Pieri Inclusion Removing Many Boxes is a GL(V)-map subsection: The Theorem equation: Phi of Garnir m Boxes} holds for \(m = 2\). 
The general case can then be shown using similar techniques since, by Theorem \ref{Generating Garnir Relations and Tools for Collapsing Sums subsection: One Garnir Theorem}, it is enough to show that \ref{Showing the Pieri Inclusion Removing Many Boxes is a GL(V)-map subsection: The Theorem equation: Phi of Garnir m Boxes} holds only for hooks, which consist of two rows and so can only intersect at most two orbits of any \(m\)-path. 
As before, there are two options for hooks in \(T_0\), which we consider separately. %
Note that as we are considering the case \(m = 2\), we have the removal set
\[
    X = \{ x_1 = [b_1](1,w_{b_1}), x_2 = [b_2](i_2,w_{b_2})\}.
\]

\subsection{Two Box Removal Preserves Garnirs for Hooks Contained in a Single Block} \label{Showing the Pieri Inclusion Removing Many Boxes is a GL(V)-map subsection: Two Box Removal Preserves Garnirs for Hooks Contained in a Single Block}

We first show that Equation \ref{Showing the Pieri Inclusion Removing Many Boxes is a GL(V)-map subsection: The Theorem equation: Phi of Garnir m Boxes} holds when \(m = 2\) for all hooks \(A \subset [b]\), for some \(1 \le b \le N\). 
For the rest of Section \ref{Showing the Pieri Inclusion Removing Many Boxes is a GL(V)-map subsection: Two Box Removal Preserves Garnirs for Hooks Contained in a Single Block}, fix \(T \in \mF_{\lambda, n}\) and let 
\[
    A = \{ a_0 := [b](i_0,1), a_1 = [b](i_0 + 1,1), \ldots, a_{w_b} = [b](i_0 + 1, w_b)\} \subset T_0
\]
with \(1 \le i_0 < h_b\), so that \(A \subset [b]\).  
Denote the entries of \(A\) in \(T\) by \(A_k = T_{a_k}\) for \(k = 0, 1, \ldots, w_b\).
Then by Lemma \ref{Generating Garnir Relations and Tools for Collapsing Sums subsection: Simplifying a Garnir}, mod \(F_2 \otimes \mR_{\lambda \setminus X, n}\) we have
\begin{align*}
    \Phi_2\left(G_A (T)\right) %
	    & = \sum_{P} \frac{(-1)^{P}}{H(P)} P\left(\sum_{\sigma \in S_A} \sigma T \right) \\%
        & = C \sum_{P} \sum_{k=0}^{w_b} \frac{(-1)^{P}}{H(P)} P\left(\sigma^A_k T \right),
\end{align*}
where the sum is over all \(2\)-paths \(P\) on \(\lambda\) removing \(X\).
The set of all \(P_k := P(\sigma^A_k T)\) appearing in the image \(\Phi_2\left(G_A (T)\right)\) above is the union of the following disjoint sets.\\

The \(P_k\)s that miss \(A\),
\begin{equation} \label{Showing the Pieri Inclusion Removing Many Boxes is a GL(V)-map subsection: Two Box Removal Preserves Garnirs for Hooks Contained in a Single Block equation: T1}
    \sT_1 = \{P_k  :  R^P \cap A = \emptyset\}.\\
\end{equation}

The \(P_k\)s that hit \(A\) and keep \(A\) in block \(b\),
\begin{equation}\label{Showing the Pieri Inclusion Removing Many Boxes is a GL(V)-map subsection: Two Box Removal Preserves Garnirs for Hooks Contained in a Single Block equation: T2}
    \sT_2 = \{P_k  :  R^P \cap A \ne \emptyset, ~ P(A) \subset [b]\}.\\
\end{equation}

The \(P_k\)s that have exactly one orbit in \([b]\) and move \(A_i\) above \([b]\),
\begin{equation}\label{Showing the Pieri Inclusion Removing Many Boxes is a GL(V)-map subsection: Two Box Removal Preserves Garnirs for Hooks Contained in a Single Block equation: T3}
    \begin{split}
        \sT_3 & = \bigsqcup_{0 \le i \le w_b} \sT_3^i, \text{ where }\\
        \sT_3^i & = \{P_k : \text{ exactly one of } R_1, R_2 \text{ intersects } [b] \text{ and } P(\sigma_{k} A_i) > [b] \}.
    \end{split}
\end{equation}

The \(P_k\)s that move \(A_i\) and \(A_j\) above \([b]\),
\begin{equation}\label{Showing the Pieri Inclusion Removing Many Boxes is a GL(V)-map subsection: Two Box Removal Preserves Garnirs for Hooks Contained in a Single Block equation: T4}
    \begin{split}
        \sT_4 & = \bigsqcup_{0 \le i < j \le w_b} \sT_4^{i,j}, \text{ where }\\
	    \sT_4^{i,j} & =\{P_k \in \sT_4  :  P(\sigma^A_k a_i) > [b] \text{ and } P(\sigma^A_k a_j) > [b]\}.
    \end{split}
\end{equation}

The \(P_k\)s that move \(A_i\) and a box \(z \in [b]\), with \(z < A\), above \([b]\),
\begin{equation}\label{Showing the Pieri Inclusion Removing Many Boxes is a GL(V)-map subsection: Two Box Removal Preserves Garnirs for Hooks Contained in a Single Block equation: T5}
    \begin{split}
        \sT_5 &= \bigsqcup_{\substack{0 \le i \le w_b, ~ z = [b](i_z,j_z),\\  %
	            ~ 1 \le i_z \le i_0 - 1, 1 \le j_z \le w_b}} \sT_5^{i,z}, \text{ where }\\
	    \sT_5^{i,z} &= \{ P_k \in \sT_5  :  P(\sigma^A_k a_i) > [b], P(z) > [b]\}.
    \end{split}
\end{equation}

The \(P_k\)s that move \(A_i\) and a box \(z \not \in A\) in row \(i_0\) above \([b]\),
\begin{equation}\label{Showing the Pieri Inclusion Removing Many Boxes is a GL(V)-map subsection: Two Box Removal Preserves Garnirs for Hooks Contained in a Single Block equation: T6}
    \begin{split}
        \sT_6 &= \bigsqcup_{0 \le i \le w_b, ~ 2 \le j \le w_b} \sT_6^{i,j}, \text{ where }\\
	    \sT_6^{i,j} &= \{ P_k \in \sT_6  :  P(\sigma^A_k a_i) > [b], P([b](i_0,j)) > [b]\}.
    \end{split}
\end{equation}

The \(P_k\)s that move \(A_i\) and a box in \([b]\) above \(A\) above \([b]\),
\begin{equation}\label{Showing the Pieri Inclusion Removing Many Boxes is a GL(V)-map subsection: Two Box Removal Preserves Garnirs for Hooks Contained in a Single Block equation: T7}
    \begin{split}
        \sT_7 &= \bigsqcup_{\substack{0 \le i \le w_b\\ 1\le j \le w_b}} \sT_7^{i,j},, \text{ where }\\
	    \sT_7^{i,j} &= \{ P_k \in \sT_7  :  P(\sigma^A_k A_i) > [b], [b](i_0 + 2,j) \in R^P\}.
    \end{split}
\end{equation}

Then we have, mod \(F_2 \otimes \mR_{\lambda \setminus X,n}\),
\[
	\Phi_2\left(G_A (T) \right) %
	   	= C \sum_{P} \sum_{k=0}^{w_b} \frac{(-1)^{P}}{H(P)} P\left(\sigma^A_k T \right) 
        = C \sum_{j=1,\ldots,7} \sum_{P_k \in \sT_j} \frac{(-1)^{P}}{H(P)} P_k.
\]

We show that for \(j = 1, \ldots, 7\), 
\[
    \sum_{P_k \in T_j} \frac{(-1)^{P}}{H(P)} P_k \in F_2 \otimes \mR_{\lambda \setminus X,n},
\]
and hence Equation \ref{Showing the Pieri Inclusion Removing Many Boxes is a GL(V)-map subsection: The Theorem equation: Phi of Garnir m Boxes} holds when \(m = 2\) for all blocks \(A \subset [b]\).

The proofs of Cases (\ref{Showing the Pieri Inclusion Removing Many Boxes is a GL(V)-map subsection: Two Box Removal Preserves Garnirs for Hooks Contained in a Single Block equation: T1}), (\ref{Showing the Pieri Inclusion Removing Many Boxes is a GL(V)-map subsection: Two Box Removal Preserves Garnirs for Hooks Contained in a Single Block equation: T2}), and (\ref{Showing the Pieri Inclusion Removing Many Boxes is a GL(V)-map subsection: Two Box Removal Preserves Garnirs for Hooks Contained in a Single Block equation: T3}) are similar to the proofs of Cases (\ref{Showing the Pieri Inclusion Removing One Box is a GL(V)-map subsection: One Box Removal Preserves Garnirs for Hooks Contained in a Single Block equation: T1}), (\ref{Showing the Pieri Inclusion Removing One Box is a GL(V)-map subsection: One Box Removal Preserves Garnirs for Hooks Contained in a Single Block equation: T2}), and (\ref{Showing the Pieri Inclusion Removing One Box is a GL(V)-map subsection: One Box Removal Preserves Garnirs for Hooks Contained in a Single Block equation: T3}), respectively. 
It remains to show the proofs of Cases (\ref{Showing the Pieri Inclusion Removing Many Boxes is a GL(V)-map subsection: Two Box Removal Preserves Garnirs for Hooks Contained in a Single Block equation: T4}) - (\ref{Showing the Pieri Inclusion Removing Many Boxes is a GL(V)-map subsection: Two Box Removal Preserves Garnirs for Hooks Contained in a Single Block equation: T7}).
In each case we assume \(b > b_1\), with the case \(b = b_1\) being similar. We will also assume in each case that \(A \cap X = \emptyset\), as if \(A \cap X \ne \emptyset\) we may follow the proof of Subcase (\ref{Showing the Pieri Inclusion Removing One Box is a GL(V)-map subsection: One Box Removal Preserves Garnirs for Hooks Contained in a Two Blocks equation: T3}.1).

\begin{case*}[\ref{Showing the Pieri Inclusion Removing Many Boxes is a GL(V)-map subsection: Two Box Removal Preserves Garnirs for Hooks Contained in a Single Block equation: T4}]
In this case we show that the sum over all paths that move \(A_i\) and \(A_j\) above \([b]\) is in \(F_2 \otimes \mR_{\lambda \setminus X,n}\).
Recall that 
\[
    \sT_4 = \bigsqcup_{0 \le i < j \le w_b} \sT_4^{i,j},
\]
where
\[
    \sT_4^{i,j} = \{P_k \in \sT_4  :  P(\sigma^A_k a_i) > [b] \text{ and } P(\sigma^A_k a_j) > [b]\}.
\]
It is enough to show that for \(0 \le i < j \le w_b\),
\[
    \sum_{P_k \in \sT_4^{i,j}} \frac{(-1)^{P}}{H(P)} P_k \in F_2 \otimes \mR_{\lambda \setminus X,n}.
\]
We will show the case  \(i = 0\) and  \(j = 1\), with rest being similar. 
\end{case*}

\begin{proof}
For the rest of Case (\ref{Showing the Pieri Inclusion Removing Many Boxes is a GL(V)-map subsection: Two Box Removal Preserves Garnirs for Hooks Contained in a Single Block equation: T4}) let \(\sT := \sT_4^{0,1}\). 
Observe that for all \(\ds P_k \in \sT\), either \(k = 0\) or \(k = 1\), as otherwise \(\sigma^P_k A_0\) and \(\sigma^A_k A_1\) are in the same row.

Next we will define for each \(\ds P_0 \in \sT\) a unique \(P^\prime_1 \in \sT\) that agrees with \(P\) except on \(\{a_0, a_1\}\). 
The conditions on \(P^\prime_1\) will depend on whether or not \(P\) ``removes'' (i.e. maps to \(Y\)) either or both of \(a_0, a_1\). 
We want to construct \(P^\prime_1\)  so that it sends \(A_0\) and \(A_1\) to the same place \(P\) does, but with the freedom to do so with either the orbit of \(x_1\) or \(x_2\).
For each \(\ds P_0 \in \sT\), let \(P^\prime_1 \in \sT\) such that \(P^\prime \equiv P\) except on \(\{a_0, a_1\}\), and
\begin{itemize}
    \item if \(\{P(a_0),P(a_1)\} \cap Y = \emptyset\),
        \[
            P^\prime(\sigma^A_1 a_0) = P(a_0) \text{ and } P^\prime(\sigma^A_1 (a_1)) = P(a_1).
        \]
    \item if \(P(a_0) \in Y\) and \(P(a_1) \not \in Y\),
        \[
            P^\prime(\sigma^A_1 a_0) \in Y \text{ and } P^\prime(\sigma^A_1 (a_1)) = P(a_1).
        \]
    \item if \(P(a_0) \not \in Y\) and \(P(a_1) \in Y\),
        \[
            P^\prime(\sigma^A_1 a_0) = P(a_0) \text{ and } P^\prime(\sigma^A_1 (a_1)) \in Y.
        \]
    \item if \(\{P(a_0), P(a_1)\} = Y\),
        \[
            \{P^\prime(\sigma^A_1 a_0), P^\prime(\sigma^A_1 (a_1))\} = Y.
        \]
\end{itemize}

\begin{center}
\begin{tikzpicture}
    \node at (-1,2) {\(P_0 =\)};
    
    \draw (0,0) rectangle (3,4);
    \draw (0,-.25) -- (0,0);
    \draw (0,4) -- (0,4.25);
    \draw (2.5,-.25) -- (2.5,0);
    \draw (3,4) -- (3.5,4) -- (3.5,4.25);
    
    \draw (0,2.5) rectangle (.5,3);
    \node at (.25,2.75) {\small \(A_0\)};
    
    \draw (0,3) rectangle (.5,3.5);
    \node at (.25,3.25) {\small \(A_1\)};
    \draw (.5,3) rectangle (2.5,3.5);
    \node at (1.5,3.25) {\small \(\cdots\)};
    \draw (2.5,3) rectangle (3,3.5);
    \node at (2.75,3.25) {\small \(A_w\)};
    
    \node at (.5,2.2) {\(u\)};
    \node at (1.75,2.4) {\(v\)};
    
    \draw[blue] (.4,-.5) to [out = 120, in = 180] (1.25,.2); 
    \draw[blue] (1.25,.2) to [out = 200, in = 180] (1.25,.6);
    \draw[blue] (1.25,.6) to [out = 200, in = 180] (1.25,1);
    \draw[blue] (1.25,1) to [out = 200, in = 180] (1.25,1.4);
    \draw[blue] (1.25,1.4) to [out = 200, in = 180] (1.25,1.8);
    \draw[blue] (1.25,1.8) to [out = 220, in = 220] (.5,2.2); 
    \draw[blue] (.5,2.2) to [out = 200, in = 180] (.1,2.75); 
    \draw[blue,->] (.2,2.75) to [out = 180, in = 270] (-1,4); 
    
    \draw[red, very thick, dotted] (1.6,-.5) to [out = 20, in = 0] (1.25,.4); 
    \draw[red, very thick, dotted] (1.25,.4) to [out = -20, in = 0] (1.25,.8);
    \draw[red, very thick, dotted] (1.25,.8) to [out = -20, in = 0] (1.25,1.2);
    \draw[red, very thick, dotted] (1.25,1.2) to [out = -20, in = 0] (1.25,1.6);
    \draw[red, very thick, dotted] (1.25,1.6) to [out = -20, in = 0] (1.25,2);
    \draw[red, very thick, dotted] (1.25,2) to [out = -20, in = 0] (1.5,2.4); 
    \draw[red, very thick, dotted] (1.5,2.4) to [out = 180, in = 220] (.1,3.25); 
    \draw[red,->, very thick, dotted] (.1,3.25) to [out = 180, in = 270] (-.5,4); 
\end{tikzpicture}
\hspace{1in}
\begin{tikzpicture}
    \node at (-1,2) {\(P^\prime_1 =\)};
    
    \draw (0,0) rectangle (3,4);
    \draw (0,-.25) -- (0,0);
    \draw (0,4) -- (0,4.25);
    \draw (2.5,-.25) -- (2.5,0);
    \draw (3,4) -- (3.5,4) -- (3.5,4.25);
    
    \draw (0,2.5) rectangle (.5,3);
    \node at (.25,2.75) {\small \(A_1\)};
    
    \draw (0,3) rectangle (.5,3.5);
    \node at (.25,3.25) {\small \(A_0\)};
    \draw (.5,3) rectangle (2.5,3.5);
    \node at (1.5,3.25) {\small \(\cdots\)};
    \draw (2.5,3) rectangle (3,3.5);
    \node at (2.75,3.25) {\small \(A_w\)};
    
    \node at (.5,2.2) {\(u\)};
    \node at (1.75,2.4) {\(v\)};
    
    \draw[blue] (.4,-.5) to [out = 120, in = 180] (1.25,.2); 
    \draw[blue] (1.25,.2) to [out = 200, in = 180] (1.25,.6);
    \draw[blue] (1.25,.6) to [out = 200, in = 180] (1.25,1);
    \draw[blue] (1.25,1) to [out = 200, in = 180] (1.25,1.4);
    \draw[blue] (1.25,1.4) to [out = 200, in = 180] (1.25,1.8);
    \draw[blue] (1.25,1.8) to [out = 220, in = 220] (.5,2.2); 
    \draw[blue] (.5,2.2) to [out = 200, in = 180] (.1,2.75); 
    \draw[blue,->] (.2,2.75) to [out = 180, in = 270] (-.5,4); 
    
    \draw[red, very thick, dotted] (1.6,-.5) to [out = 20, in = 0] (1.25,.4); 
    \draw[red, very thick, dotted] (1.25,.4) to [out = -20, in = 0] (1.25,.8);
    \draw[red, very thick, dotted] (1.25,.8) to [out = -20, in = 0] (1.25,1.2);
    \draw[red, very thick, dotted] (1.25,1.2) to [out = -20, in = 0] (1.25,1.6);
    \draw[red, very thick, dotted] (1.25,1.6) to [out = -20, in = 0] (1.25,2);
    \draw[red, very thick, dotted] (1.25,2) to [out = -20, in = 0] (1.5,2.4); 
    \draw[red, very thick, dotted] (1.5,2.4) to [out = 180, in = 220] (.1,3.25); 
    \draw[red,->, very thick, dotted] (.1,3.25) to [out = 180, in = 270] (-1,4); 
\end{tikzpicture}
\end{center}

It is clear that for for each \(P_0 \in \sT\) the choice of \(P^\prime_1\) is unique, and that all \(Q_1 \in \sT\) arise in such a way. 
Thus
\[
    \sum_{P_k \in \sT} \frac{(-1)^{P}}{H(P)} P_k = \sum_{P_0 \in \sT} \left( \frac{(-1)^{P}}{H(P)} P_0 + \frac{(-1)^{P^\prime}}{H(P^\prime)} P^\prime_1 \right)
\]

As \((-1)^P = (-1)^{P^\prime}\) and \(H(P) = H(P^\prime)\), it is then enough to show that
\[
    \sum_{P_0 \in \sT} P_0 + P^\prime_1 \in F_2 \otimes \mR_{\lambda \setminus X,n}.
\]

We will in fact show that for each \(P_0 \in \sT\), \(P_0 + P^\prime_1 \in F_2 \otimes \mR_{\lambda \setminus X,n}\). 
Pick \(P_0\) and the corresponding \(P^\prime_1 \in \sT\) and let \(u = P^{-1}(a_0)\) and \(v = P^{-1}(A_1)\).
Let \(T^\prime \in \mF_{\lambda \setminus X, n}\) be the unique tableau with \(T^\prime = T_P\) on \((\lambda \setminus X) \setminus \{[b](i_0), [b](i_0 + 1)\}\) and \(T^\prime = T \text{ on } \{[b](i_0), [b](i_0 + 1)\}\) except \(T^\prime_{a_0} = u\) and \(T^\prime_{a_1} = v\).

\begin{center}
\begin{tikzpicture}
    \node at (-1,2) {\(T^\prime =\)};
    
    \draw (0,0) rectangle (3,4);
    \draw (0,-.25) -- (0,0);
    \draw (0,4) -- (0,4.25);
    \draw (2.5,-.25) -- (2.5,0);
    \draw (3,4) -- (3.5,4) -- (3.5,4.25);
    
    \draw (0,2.5) rectangle (.5,3);
    \node at (.25,2.75) {\small \(u\)};
    
    \draw (0,3) rectangle (.5,3.5);
    \node at (.25,3.25) {\small \(v\)};
    \draw (.5,3) rectangle (2.5,3.5);
    \node at (.75,3.25) {\small \(A_2\)};
    \draw (1,3) -- (1,3.5);
    \node at (1.75,3.25) {\small \(\cdots\)};
    \draw (2.5,3) rectangle (3,3.5);
    \node at (2.75,3.25) {\small \(A_w\)};
    
    \node at (-.5,4) { \ };
    \node at (-1,4) { \ };
    \node at (.5,-.5) { \ };
    \node at (1.5,-.5) { \ };
    
    \draw[blue] (.4,-.5) to [out = 120, in = 180] (1.25,.2); 
    \draw[blue] (1.25,.2) to [out = 200, in = 180] (1.25,.6);
    \draw[blue] (1.25,.6) to [out = 200, in = 180] (1.25,1);
    \draw[blue] (1.25,1) to [out = 200, in = 180] (1.25,1.4);
    \draw[blue] (1.25,1.4) to [out = 200, in = 180] (1.25,1.8);
    \draw[blue] (1.25,1.8) to [out = 220, in = 220] (.5,2.2); 
    
    \draw[red, very thick, dotted] (1.6,-.5) to [out = 20, in = 0] (1.25,.4); 
    \draw[red, very thick, dotted] (1.25,.4) to [out = -20, in = 0] (1.25,.8);
    \draw[red, very thick, dotted] (1.25,.8) to [out = -20, in = 0] (1.25,1.2);
    \draw[red, very thick, dotted] (1.25,1.2) to [out = -20, in = 0] (1.25,1.6);
    \draw[red, very thick, dotted] (1.25,1.6) to [out = -20, in = 0] (1.25,2);
    \draw[red, very thick, dotted] (1.25,2) to [out = -20, in = 0] (1.5,2.4); 
    
\end{tikzpicture}
\end{center}

Then, mod \(F_2 \otimes \mR_{\lambda \setminus X,n}\) we have
\[
    P_0 + P_1^\prime 
        = \left( ~ \ytableaushort{{\alpha^P_2}, {\alpha^P_1}} + \begin{ytableau} \alpha^P_1 \\ %
	            \alpha^P_2 \end{ytableau}\right) \otimes T^\prime
	    = 0.
\]
\end{proof}

\begin{case*}[\ref{Showing the Pieri Inclusion Removing Many Boxes is a GL(V)-map subsection: Two Box Removal Preserves Garnirs for Hooks Contained in a Single Block equation: T5}]
In this case we show that the sum over all paths that move \(A_i\) and a box \(z \in [b]\), with \(z < A\), above \([b]\) is in \(F_2 \otimes \mR_{\lambda \setminus X, n}\).
Recall that
\[
    \sT_5 = \bigsqcup_{\substack{0 \le i \le w_b, ~ z = [b](i_z,j_z),\\ ~ 1 \le i_z \le i_0 - 1, 1 \le j_z \le w_b}} \sT_5^{i,z},
\]
where
\[
    \sT_5^{i,z} = \{ P_k \in \sT_5  :  P(\sigma^A_k a_i) > [b], P(z) > [b]\}.
\]
It is enough to show that
\[
    \sum_{P_k \in \sT_5^{0,z}} \frac{(-1)^{P}}{H(P)} P_k \in F_2 \otimes \mR_{\lambda \setminus X,n},
\]
for some \(z = [b](i_z,j_z)\) a fixed box with \(Z = T_{z}\), \(1 \le i_z \le i_0 - 1\) odd, and \(1 \le j_z \le w_b\), with the other cases being similar.
\end{case*}

\begin{proof}
For the rest of Case (\ref{Showing the Pieri Inclusion Removing Many Boxes is a GL(V)-map subsection: Two Box Removal Preserves Garnirs for Hooks Contained in a Single Block equation: T5}) let \(\sT := \sT_5^{0,z}\) and, for any \(2\)-path \(P\), let \(\tilde{h^P} := h^P - h_b^P\).
Define the relation \(\sim\) on \(\sT\) by 
\begin{align*}
    P_k \sim Q_j
        \iff & Q \text{ is a } ([b](1),[b](i_0 + 1))\text{-path extension of } P\\
             & \text{ and if } P(u) \in [b](1) \text{ for some box } u < [b], \text{ then } Q(u) \in [b](1).
\end{align*}
It is clear that this defines an equivalence relation on \(\sT\), so that
\[
    \sum_{P_k \in \sT_5^{0,z}} \frac{(-1)^{P}}{H(P)} P_k = \sum_{\left[P_k\right] \in \sT / \sim} ~ %
    \sum_{Q_k \in \left[P_k\right]} \frac{(-1)^{Q}}{H(Q)} Q_k.
\]

Pick \(P_0 \in \sT\) with \([b](i,1) \in R^P\) for all \(1 \le i \ne i_z \le i_0\), and let \([b_u](i_u,j_u) = P^{-1}([b](1,1))\) and \([b_v](i_v,j_v) = P^{-1}([b](2,1))\) with \(u = T_{[b_u](i_u,j_u)}\) and \(v = T_{[b_v](i_v,j_v)}\).

\begin{center}
\begin{tikzpicture}
    \node at (-1,2) {\(P_0 = \)};
    
    \draw (0,0) rectangle (3,4);
    \draw (0,-.25) -- (0,0);
    \draw (0,4) -- (0,4.25);
    \draw (2.5,-.25) -- (2.5,0);
    \draw (3,4) -- (3.5,4) -- (3.5,4.25);
    
    \draw (0,2.5) rectangle (.5,3);
    \draw (.5,0) -- (.5,2.5);
    \node at (.25,2.75) {\small \(A_0\)};
    
    \draw (0,3) rectangle (.5,3.5);
    \node at (.25,3.25) {\small \(A_1\)};
    \draw (.5,3) rectangle (2.5,3.5);
    \node at (1.5,3.25) {\small \(\cdots\)};
    \draw (2.5,3) rectangle (3,3.5);
    \node at (2.75,3.25) {\small \(A_w\)};
    
    \draw (0,1.25) rectangle (3,1.75);
    \draw (1.25,1.25) -- (1.25,1.75);
    \draw (1.75,1.25) -- (1.75,1.75);
    \node at (1.5,1.5) {\small \(Z\)};
    
    \node at (.5,-.5) {\(u\)};
    \node at (1.5,-.5) {\(v\)};
    
    \draw[blue] (.4,-.5) to [out = 200, in = 180] (.2,.2); 
    \draw[blue] (.2,.2) to [out = 200, in = 180] (.2,.6);
    \draw[blue] (.2,.6) to [out = 200, in = 180] (.2,1);
    \draw[blue] (.2,1) to [out = 140, in = 180] (1.4,1.5); 
    \draw[blue,->] (1.4,1.5) to [out = 180, in = 270] (-.75,4); 
    
    \draw[red, very thick, dotted] (1.6,-.5) to [out = 20, in = 0] (.2,.4); 
    \draw[red, very thick, dotted] (.2,.4) to [out = -20, in = 0] (.2,.8);
    \draw[red, very thick, dotted] (.2,.8) to [out = -20, in = 0] (.2,1.2);
    \draw[red, very thick, dotted] (.2,1.2) to [out = 200, in = 180] (.2,1.8); 
    \draw[red, very thick, dotted] (.2,1.8) to [out = 200, in = 180] (.2,2);
    \draw[red, very thick, dotted] (.2,2) to [out = 200, in = 180] (.2,2.2);
    \draw[red, very thick, dotted] (.2,2.2) to [out = 200, in = 180] (.2,2.4);
    \draw[red, very thick, dotted] (.2,2.4) to [out = 200, in = 180] (.1,2.75); 
    \draw[red,->, very thick, dotted] (.1,2.75) to [out = 180, in = 270] (-.5,4); 
\end{tikzpicture}
\end{center}

It is then enough to show that
\[
    \sum_{Q_k \in \left[ P_0 \right]} \frac{(-1)^{Q}}{H(Q)} Q_k \in F_2 \otimes \mR_{\lambda \setminus X, n}.
\]
In fact, as \(\tilde{h^Q} = \tilde{h^P}\) and \(H(Q) = H(P)\) for all \(Q_k \in \left[ P_0 \right]\), it is enough to show that
\[
    \sum_{Q_k \in \left[ P_0 \right]} (-1)^{h_b^Q} Q_k \in F_2 \otimes \mR_{\lambda \setminus X,n}.
\]

Observe that \([P_0]\) can be written as the disjoint union
\[
    [P_0] = \bigsqcup_{i = 1}^3 [P_0]_i,
\]
where the \([P_0]_i\) are defined as follows.\\

The paths acting on \(\sigma_0^A T\), 
\begin{equation*}
    \left[P_0\right]_1 = \{Q_0 \in \left[P_0\right]\}, 
\end{equation*}
\begin{center}
    \begin{tikzpicture}
        \draw (0,0) rectangle (3,4);
        \draw (0,-.25) -- (0,0);
        \draw (0,4) -- (0,4.25);
        \draw (2.5,-.25) -- (2.5,0);
        \draw (3,4) -- (3.5,4) -- (3.5,4.25);
        
        \draw (0,2.5) rectangle (.5,3);
        \node at (.25,2.75) {\small \(A_0\)};
        
        \draw (0,3) rectangle (.5,3.5);
        \node at (.25,3.25) {\small \(A_1\)};
        \draw (.5,3) rectangle (2.5,3.5);
        \node at (1.5,3.25) {\small \(\cdots\)};
        \draw (2.5,3) rectangle (3,3.5);
        \node at (2.75,3.25) {\small \(A_w\)};
        
        \draw (0,1.25) rectangle (3,1.75);
        \draw (1.25,1.25) -- (1.25,1.75);
        \draw (1.75,1.25) -- (1.75,1.75);
        \node at (1.5,1.5) {\small \(Z\)};
        
        \node at (.5,-.5) {\(u\)};
        \node at (1.5,-.5) {\(v\)};
        
        \draw[blue] (.4,-.5) to [out = 120, in = 180] (.8,.2); 
        \draw[blue] (.8,.2) to [out = 200, in = 180] (.8,.6);
        \draw[blue] (.8,.6) to [out = 200, in = 180] (.8,1);
        \draw[blue] (.8,1) to [out = 140, in = 180] (1.4,1.5); 
        \draw[blue,->] (1.4,1.5) to [out = 180, in = 270] (-.75,4); 
        
        \draw[red, very thick, dotted] (1.6,-.5) to [out = 20, in = 0] (.8,.4); 
        \draw[red, very thick, dotted] (.8,.4) to [out = -20, in = 0] (.8,.8);
        \draw[red, very thick, dotted] (.8,.8) to [out = -20, in = 0] (.8,1.2);
        \draw[red, very thick, dotted] (.8,1.2) to [out = 200, in = 180] (.8,1.8); 
        \draw[red, very thick, dotted] (.8,1.8) to [out = 200, in = 180] (.8,2);
        \draw[red, very thick, dotted] (.8,2) to [out = 200, in = 180] (.8,2.2);
        \draw[red, very thick, dotted] (.8,2.2) to [out = 200, in = 180] (.8,2.4);
        \draw[red, very thick, dotted] (.8,2.4) to [out = 200, in = 220] (.1,2.75); 
        \draw[red,->, very thick, dotted] (.1,2.75) to [out = 180, in = 270] (-.5,4); 
    \end{tikzpicture}
\end{center}
the paths acting on \(\sigma_k^AT\), \(k \ne 0\), that hit \(a_0 = \sigma_k^A a_k\),
\begin{equation*}
    \left[P_0\right]_2 = \{Q_k \in \left[P_0\right]  :  k\ne 0, a_0 \in R^Q\},
\end{equation*}
\begin{center}
    \begin{tikzpicture}
        \draw (0,0) rectangle (3,4);
        \draw (0,-.25) -- (0,0);
        \draw (0,4) -- (0,4.25);
        \draw (2.5,-.25) -- (2.5,0);
        \draw (3,4) -- (3.5,4) -- (3.5,4.25);
        
        \draw (0,2.5) rectangle (.5,3);
        \node at (.25,2.75) {\small \(A_k\)};
        
        \draw (0,3) rectangle (3,3.5);
        \draw (1.25,3) -- (1.25,3.5);
        \draw (1.75,3) -- (1.75,3.5);
        \node at (1.5,3.25) {\small \(A_0\)};
        \node at (1.5,3.7) {\small \(k\)};
        
        \draw (0,1.25) rectangle (3,1.75);
        \draw (1.25,1.25) -- (1.25,1.75);
        \draw (1.75,1.25) -- (1.75,1.75);
        \node at (1.5,1.5) {\small \(Z\)};
        
        \node at (.5,-.5) {\(u\)};
        \node at (1.5,-.5) {\(v\)};
    
        \draw[blue] (.4,-.5) to [out = 120, in = 180] (.8,.2); 
        \draw[blue] (.8,.2) to [out = 200, in = 180] (.8,.6);
        \draw[blue] (.8,.6) to [out = 200, in = 180] (.8,1);
        \draw[blue] (.8,1) to [out = 140, in = 180] (1.4,1.5); 
        \draw[blue,->] (1.4,1.5) to [out = 180, in = 270] (-.75,4); 
        
        \draw[red, very thick, dotted] (1.6,-.5) to [out = 20, in = 0] (.8,.4); 
        \draw[red, very thick, dotted] (.8,.4) to [out = -20, in = 0] (.8,.8);
        \draw[red, very thick, dotted] (.8,.8) to [out = -20, in = 0] (.8,1.2);
        \draw[red, very thick, dotted] (.8,1.2) to [out = 200, in = 180] (.8,1.8); 
        \draw[red, very thick, dotted] (.8,1.8) to [out = 200, in = 180] (.8,2);
        \draw[red, very thick, dotted] (.8,2) to [out = 200, in = 180] (.8,2.2);
        \draw[red, very thick, dotted] (.8,2.2) to [out = 200, in = 180] (.8,2.4);
        \draw[red, very thick, dotted] (.8,2.4) to [out = 200, in = 220] (.1,2.75); 
        \draw[red, very thick, dotted] (.1,2.75) to [out = 120, in = 180] (1.3,3.25); 
        \draw[red,->, very thick, dotted] (1.3,3.25) to [out = 180, in = 270] (-.5,4); 
    \end{tikzpicture}
\end{center}
and the paths acting on \(\sigma_k^AT\), \(k \ne 0\), that miss \(a_0 = \sigma_k^A a_k\),
\begin{equation*}
    \left[P_0\right]_3 = \{Q_k \in \left[P_0\right] :  k \ne 0, a_0 \not \in R^Q\}.
\end{equation*}
\begin{center}
    \begin{tikzpicture}
        \draw (0,0) rectangle (3,4);
        \draw (0,-.25) -- (0,0);
        \draw (0,4) -- (0,4.25);
        \draw (2.5,-.25) -- (2.5,0);
        \draw (3,4) -- (3.5,4) -- (3.5,4.25);
        
        \draw (0,2.5) -- (3,2.5);
        \draw (.5,2.5) -- (.5,3);
        \node at (.25,2.75) {\small \(A_k\)};
        
        \draw (0,3) rectangle (3,3.5);
        \draw (1.25,3) -- (1.25,3.5);
        \draw (1.75,3) -- (1.75,3.5);
        \node at (1.5,3.25) {\small \(A_0\)};
        \node at (1.5,3.7) {\small \(k\)};
        
        \draw (0,1.25) rectangle (3,1.75);
        \draw (1.25,1.25) -- (1.25,1.75);
        \draw (1.75,1.25) -- (1.75,1.75);
        \node at (1.5,1.5) {\small \(Z\)};
        
        \node at (.5,-.5) {\(u\)};
        \node at (1.5,-.5) {\(v\)};
        
        \draw[blue] (.4,-.5) to [out = 120, in = 180] (.8,.2); 
        \draw[blue] (.8,.2) to [out = 200, in = 180] (.8,.6);
        \draw[blue] (.8,.6) to [out = 200, in = 180] (.8,1);
        \draw[blue] (.8,1) to [out = 140, in = 180] (1.4,1.5); 
        \draw[blue,->] (1.4,1.5) to [out = 180, in = 270] (-.75,4); 
        
        \draw[red, very thick, dotted] (1.6,-.5) to [out = 20, in = 0] (.8,.4); 
        \draw[red, very thick, dotted] (.8,.4) to [out = -20, in = 0] (.8,.8);
        \draw[red, very thick, dotted] (.8,.8) to [out = -20, in = 0] (.8,1.2);
        \draw[red, very thick, dotted] (.8,1.2) to [out = 200, in = 180] (.8,1.8); 
        \draw[red, very thick, dotted] (.8,1.8) to [out = 200, in = 180] (.8,2);
        \draw[red, very thick, dotted] (.8,2) to [out = 200, in = 180] (.8,2.2);
        \draw[red, very thick, dotted] (.8,2.2) to [out = 200, in = 180] (.8,2.4);
        \draw[red, very thick, dotted] (.8,2.4) to [out = 120, in = 180] (1.5,2.75); 
        \draw[red, very thick, dotted] (1.5,2.75) to [out = 180, in = 180] (1.3,3.25); 
        \draw[red,->, very thick, dotted] (1.3,3.25) to [out = 180, in = 270] (-.5,4); 
    \end{tikzpicture}
\end{center}

Let \(T^\prime \in \mF_{\lambda \setminus X}\) be the unique tableau with \(T^\prime = T_P\) on \((\lambda \setminus X) \setminus [b]\) and \(T^\prime = T\) on \([b]\) except \(T^\prime_{z} = v\) and \(T^\prime_{a_0} = u\).

\begin{center}
\begin{tikzpicture}
    \node at (-1,2) {\(T^\prime = \)};
    
    \draw (0,0) rectangle (3,4);
    \draw (0,-.25) -- (0,0);
    \draw (0,4) -- (0,4.25);
    \draw (2.5,-.25) -- (2.5,0);
    \draw (3,4) -- (3.5,4) -- (3.5,4.25);
    
    \draw (0,2.5) rectangle (.5,3);
    \node at (.25,2.75) {\small \(u\)};
    
    \draw (0,3) rectangle (.5,3.5);
    \node at (.25,3.25) {\small \(A_1\)};
    \draw (.5,3) rectangle (2.5,3.5);
    \node at (1.5,3.25) {\small \(\cdots\)};
    \draw (2.5,3) rectangle (3,3.5);
    \node at (2.75,3.25) {\small \(A_w\)};
    
    \draw (0,1.25) rectangle (3,1.75);
    \draw (1.25,1.25) -- (1.25,1.75);
    \draw (1.75,1.25) -- (1.75,1.75);
    \node at (1.5,1.5) {\small \(v\)};
    
    \node at (-.5,4) { \ };
    \node at (-.75,4) { \ };
    \node at (.5,-.5) { \ };
    \node at (1.5,-.5) { \ };
\end{tikzpicture}
\end{center}

Then by Corollary \ref{Generating Garnir Relations and Tools for Collapsing Sums subsection: Calculation Lemma for 2-paths} and applications of \(G_A\) we have, mod \(F_2 \otimes \mR_{\lambda \setminus X, n}\),

\begin{equation*}
\sum_{Q_k \in \left[P_0\right]_1} (-1)^{h_b^Q} Q_k %
	= (-1)^{i_0 + i_0 - 2 + i_z - 2} ~ \ytableaushort{{\alpha^P_2}, {\alpha^P_1}} \otimes T^\prime
    = - ~ \ytableaushort{{\alpha^P_2}, {\alpha^P_1}} \otimes  T^\prime,
\end{equation*}
\begin{equation*}
\sum_{Q_k \in \left[P_0\right]_2} (-1)^{h_b^Q} Q_k %
	= (-1)^{i_0 + 1 + i_0 - 2 + i_z - 2} (w_b) ~ \begin{ytableau} \alpha^P_2 \\ %
	        \alpha^P_1 \end{ytableau} \otimes  T^\prime 
    = w_b ~ \ytableaushort{{\alpha^P_2}, {\alpha^P_1}} \otimes T^\prime,
\end{equation*}
\begin{equation*}
\sum_{Q_k \in \left[P_0\right]_3} (-1)^{h_b^Q} Q_k %
    = (-1)^{i_0 + 1 + 1 + i_0 - 2 + i_z - 2} (w_b - 1) ~ \begin{ytableau} \alpha^P_2 \\ \alpha^P_1 \end{ytableau} \otimes  T^\prime
    = - w_b + 1 ~ \ytableaushort{{\alpha^P_2}, {\alpha^P_1}} \otimes T^\prime
\end{equation*}

So, mod \(F_2 \otimes \mR_{\lambda \setminus X, n}\) we have
\[
    \sum_{Q_k \in \left[P_0\right]} (-1)^{h_b^Q} Q_k
        = \left( - 1 + w_b - w_b + 1 \right) ~ \ytableaushort{{\alpha^P_2}, {\alpha^P_1}} \otimes T^\prime 
        = 0 \otimes T^\prime.
\]
\end{proof}

\begin{case*}[\ref{Showing the Pieri Inclusion Removing Many Boxes is a GL(V)-map subsection: Two Box Removal Preserves Garnirs for Hooks Contained in a Single Block equation: T6}]
In this case we show that the sum over all paths that that move \(A_i\) and a box \(z \not \in A\) in row \(i_0\) above \([b]\) is in \(F_2 \otimes \mR_{\lambda \setminus X,n}\).
Recall that
\[
    \sT_6 = \bigsqcup_{0 \le i \le w_b, ~ 2 \le j \le w_b} \sT_6^{i,j},
\]
where
\[
    \sT_6^{i,j} = \{ P_k \in \sT_6  :  P(\sigma^A_k a_i) > [b], P([b](i_0,j)) > [b]\}.
\]
It is enough to show that
\[
    \sum_{P_k \in \sT_6^{0,2}} \frac{(-1)^{P}}{H(P)} P_k \in F_2 \otimes \mR_{\lambda \setminus X,n},
\]
with the other cases being similar.
\end{case*}

\begin{proof}
Let \(z := [b](i_0,2)\) and \(Z := T_{z}\), and observe that \(\sT_6^{0,2}\) is the union of the following disjoint sets:

\begin{align*}
    \sT_6^{0,2,1} & = \{P_k \in \sT_6^{0,2}  :  P(\sigma^A_k a_0), P(z) \not \in Y\},\\[5pt]
    \sT_6^{0,2,2} & = \{P_k \in \sT_6^{0,2}  :  P(\sigma^A_k a_0) \in Y, P(z) \not \in Y\},\\[5pt]
    \sT_6^{0,2,3} & = \{P_k \in \sT_6^{0,2}  :  P(z) \in Y, P(\sigma^A_k a_0) \not \in Y\}, \text{ and }\\[5pt]
    \sT_6^{0,2,4} & = \{P_k \in \sT_6^{0,2}  :  P(\sigma^A_k a_0), P(z) \in Y\}.\\
\end{align*}

So it is enough to show that for \(1 \le i \le 4\),
\[
    \sum_{P_k \in \sT_6^{0,2,i}} \frac{(-1)^{P}}{H(P)} P_k \in F_2 \otimes \mR_{\lambda \setminus X,n}.
\]

Define the relation \(\sim^i\) on \(\sT_6^{0,2,i}\), for \(1 \le i \le 4\), as follows.
\begin{itemize}
\item Define \(\sim^1\) on \(\sT_6^{0,2,1}\) by 
	\begin{align*}
        P_k \sim^1 Q_j %
        	\iff & Q \text{ is a } ([b](1),[b](i_0 + 1)) \text{-path extension of } P,\\
            & Q(\sigma^A_j A_0) = P(\sigma^A_k A_0), \text{ and } Q(Z) = P(Z).
    \end{align*}
\item Define \(\sim^2\) on \(\sT_6^{0,2,2}\) by 
	\begin{align*}
        P_k \sim^2 Q_j %
    	    \iff & Q \text{ is a } ([b](1),[b](i_0 + 1)) \text{-path extension of } P,\\
            & \text{ and } Q(Z) = P(Z).
    \end{align*}
\item Define \(\sim^3\) on \(\sT_6^{0,2,3}\) by 
	\begin{align*}
        P_k \sim^3 Q_j %
    	    \iff & Q \text{ is a } ([b](1),[b](i_0 + 1)) \text{-path extension of } P,\\
            & \text{ and } Q(\sigma^A_j A_0) = P(\sigma^A_k A_0).
    \end{align*}
\item Define \(\sim^4\) on \(\sT_6^{0,2,4}\) by 
	\[
        P_k \sim^4 Q_j %
    	    \iff Q \text{ is a } ([b](1),[b](i_0 + 1)) \text{-path extension of } P.
    \]
\end{itemize}

It is clear that, for \(1 \le i \le 4\), \(\sim^i\) an equivalence relation on \(\sT_6^{0,2,i}\), so that
\[
    \sum_{P_k \in \sT_6^{0,2,i}} \frac{(-1)^{P}}{H(P)} P_k = \sum_{\left[ P_k \right] \in \sT_6^{0,2,i} / \sim^i} ~  
    \sum_{Q_k \in \left[P_k\right]} \frac{(-1)^{Q}}{H(Q)} Q_k.
\]
Thus it is enough to show that for \(1 \le i \le 4\),
\[
    \sum_{\left[ P_k \right] \in \sT_6^{0,2,i} / \sim^i} \sum_{Q_k \in \left[P_k\right]} %
    \frac{(-1)^{Q}}{H(Q)} Q_k \in F_2 \otimes \mR_{\lambda \setminus X,n}.
\]
We will show the case \(i = 1\), with the rest being similar. For the rest of Case (\ref{Showing the Pieri Inclusion Removing Many Boxes is a GL(V)-map subsection: Two Box Removal Preserves Garnirs for Hooks Contained in a Single Block equation: T6}), let \(\sT := \sT_6^{0,2,1}\).
For \(l = 1, 2\), let \(\sT_{x_l}\) be the set of all \(Q_k\) in \(\sT\) such that the orbit of \(x_l\) intersects the first row in \([b]\),
\[
    \sT_{x_l} = \{P_k \in \sT  :  R^P_l \cap [b](1) \ne \emptyset\}.
\]

As \(a_0 = [b](i_0,1)\) and \(z = [b](i_0,2)\) are in the same row, it must be that \(1 \le k \le w_b\) for all \(P_k \in \sT\).
Pick \(P_1 \in \sT\) with \([b](i,1) \in R^P\) for all \(i = 1, \ldots, i_0 - 1\), and let \([b_u](i_u,j_u) = P^{-1}([b](1,1))\) and \([b_v](i_v,j_v) = P^{-1}([b](2,1))\) with \(u = T_{[b_u](i_u,j_u)}\) and \(v = T_{[b_v](i_v,j_v)}\).

\begin{center}
\begin{tikzpicture}
    \node at (-1,2) {\(P_1 =\)};
    
    \draw (0,0) rectangle (3,4);
    \draw (0,-.25) -- (0,0);
    \draw (0,4) -- (0,4.25);
    \draw (2.5,-.25) -- (2.5,0);
    \draw (3,4) -- (3.5,4) -- (3.5,4.25);
    
    \draw (0,2.5) -- (3,2.5);
    \draw (.5,0) -- (.5,3);
    \draw (1,2.5) -- (1,3);
    \node at (.75,2.75) {\small \(Z\)};
    \node at (.25,2.75) {\small \(A_k\)};
    
    \draw (0,3) rectangle (3,3.5);
    \draw (1.25,3) -- (1.25,3.5);
    \draw (1.75,3) -- (1.75,3.5);
    \node at (1.5,3.25) {\small \(A_0\)};
    \node at (1.5,3.7) {\small \(k\)};
    
    \node at (.5,-.5) {\(u\)};
    \node at (1.5,-.5) {\(v\)};
    
    \draw[blue] (.4,-.5) to [out = 120, in = 180] (.25,.2); 
    \draw[blue] (.25,.2) to [out = 200, in = 180] (.25,.6);
    \draw[blue] (.25,.6) to [out = 200, in = 180] (.25,1);
    \draw[blue] (.25,1) to [out = 200, in = 180] (.25,1.4);
    \draw[blue] (.25,1.4) to [out = 200, in = 180] (.25,1.8);
    \draw[blue] (.25,1.8) to [out = 200, in = 180] (.25,2.2);
    \draw[blue] (.25,2.2) to [out = 120, in = 180] (.6,2.75); 
    \draw[blue,->] (.6,2.75) to [out = 180, in = 270] (-1,4); 
    
    \draw[red, very thick, dotted] (1.6,-.5) to [out = 20, in = 0] (.25,.4); 
    \draw[red, very thick, dotted] (.25,.4) to [out = -20, in = 0] (.25,.8);
    \draw[red, very thick, dotted] (.25,.8) to [out = -20, in = 0] (.25,1.2);
    \draw[red, very thick, dotted] (.25,1.2) to [out = -20, in = 0] (.25,1.6);
    \draw[red, very thick, dotted] (.25,1.6) to [out = -20, in = 0] (.25,2);
    \draw[red, very thick, dotted] (.25,2) to [out = -20, in = 0] (.25,2.4);
    \draw[red, very thick, dotted] (.25,2.4) to [out = 120, in = 180] (1.35,3.25); 
    \draw[red,->, very thick, dotted] (1.35,3.25) to [out = 180, in = 270] (-.5,4); 
\end{tikzpicture}
\end{center}

It is then enough to show that
\[
    \sum_{Q_k \in \left[ P_1 \right]} \frac{(-1)^{Q}}{H(Q)} Q_k \in F_2 \otimes \mR_{\lambda \setminus X,n}.
\]
In fact, as \((-1)^P = (-1)^Q\) and \(H(Q) = H(P)\) for all \(Q_k \in \left[ P_1 \right]\), it is enough to show that
\[
    \sum_{Q_k \in \left[ P_1 \right]} Q_k \in F_2 \otimes \mR_{\lambda \setminus X,n}.
\]

Assume, without loss of generality, that \(P_1 \in \sT_{x_1}\), and let \(\left[ P_1 \right]_1 = \left[ P_1 \right] \cap \sT_{x_1}\) and \(\left[ P_1 \right]_2 = \left[ P_1 \right] \cap \sT_{x_2}\), so that
\[
    \left[P_1\right] = \left[P_1\right]_{x_1} \bigsqcup \left[P_1\right]_{x_2}.
\]
See Figure \ref{Case Showing the Pieri Inclusion Removing Many Boxes is a GL(V)-map subsection: Two Box Removal Preserves Garnirs for Hooks Contained in a Single Block equation: T6 figure: [P]}. 

\begin{figure}[H]
\begin{subfigure}{0.35\textwidth}
\begin{tikzpicture} 
    \draw (0,0) rectangle (3,4);
    \draw (0,-.25) -- (0,0);
    \draw (0,4) -- (0,4.25);
    \draw (2.5,-.25) -- (2.5,0);
    \draw (3,4) -- (3.5,4) -- (3.5,4.25);
    
    \draw (0,2.5) -- (3,2.5);
    \draw (.5,2.5) -- (.5,3);
    \draw (1,2.5) -- (1,3);
    \node at (.75,2.75) {\small \(Z\)};
    \node at (.25,2.75) {\small \(A_k\)};
    
    \draw (0,3) rectangle (3,3.5);
    \draw (1.25,3) -- (1.25,3.5);
    \draw (1.75,3) -- (1.75,3.5);
    \node at (1.5,3.25) {\small \(A_0\)};
    \node at (1.5,3.7) {\small \(k\)};
    
    \node at (.5,-.5) {\(u\)};
    \node at (1.5,-.5) {\(v\)};
    
    \draw[blue] (.4,-.5) to [out = 120, in = 180] (1.25,.2); 
    \draw[blue] (1.25,.2) to [out = 200, in = 180] (1.25,.6);
    \draw[blue] (1.25,.6) to [out = 200, in = 180] (1.25,1);
    \draw[blue] (1.25,1) to [out = 200, in = 180] (1.25,1.4);
    \draw[blue] (1.25,1.4) to [out = 200, in = 180] (1.25,1.8);
    \draw[blue] (1.25,1.8) to [out = 200, in = 180] (1.25,2.2);
    \draw[blue] (1.25,2.2) to [out = 200, in = 220] (.6,2.75); 
    \draw[blue,->] (.6,2.75) to [out = 200, in = 270] (-1,4); 
    
    \draw[red, very thick, dotted] (1.6,-.5) to [out = 20, in = 0] (1.25,.4); 
    \draw[red, very thick, dotted] (1.25,.4) to [out = -20, in = 0] (1.25,.8);
    \draw[red, very thick, dotted] (1.25,.8) to [out = -20, in = 0] (1.25,1.2);
    \draw[red, very thick, dotted] (1.25,1.2) to [out = -20, in = 0] (1.25,1.6);
    \draw[red, very thick, dotted] (1.25,1.6) to [out = -20, in = 0] (1.25,2);
    \draw[red, very thick, dotted] (1.25,2) to [out = -20, in = 0] (1.25,2.4);
    \draw[red, very thick, dotted] (1.25,2.4) to [out = 120, in = 180] (1.35,3.25); 
    \draw[red,->, very thick, dotted] (1.35,3.25) to [out = 180, in = 270] (-.5,4); 
\end{tikzpicture}
\caption{\(Q_k \in \left[P_1\right]_{x_1}\)}
\end{subfigure}%
\hspace{.5in}
\begin{subfigure}{0.35\textwidth}
\begin{tikzpicture}
    \draw (0,0) rectangle (3,4);
    \draw (0,-.25) -- (0,0);
    \draw (0,4) -- (0,4.25);
    \draw (2.5,-.25) -- (2.5,0);
    \draw (3,4) -- (3.5,4) -- (3.5,4.25);
    
    \draw (0,2.5) -- (3,2.5);
    \draw (.5,2.5) -- (.5,3);
    \draw (1,2.5) -- (1,3);
    \node at (.75,2.75) {\small \(Z\)};
    \node at (.25,2.75) {\small \(A_k\)};
    
    \draw (0,3) rectangle (3,3.5);
    \draw (1.25,3) -- (1.25,3.5);
    \draw (1.75,3) -- (1.75,3.5);
    \node at (1.5,3.25) {\small \(A_0\)};
    \node at (1.5,3.7) {\small \(k\)};
    
    \node at (.5,-.5) {\(u\)};
    \node at (1.5,-.5) {\(v\)};
    
    \draw[blue] (.4,-.5) to [out = 120, in = 180] (1.25,.4); 
    \draw[blue] (1.25,.4) to [out = 200, in = 180] (1.25,.8);
    \draw[blue] (1.25,.8) to [out = 200, in = 180] (1.25,1.2);
    \draw[blue] (1.25,1.2) to [out = 200, in = 180] (1.25,1.6);
    \draw[blue] (1.25,1.6) to [out = 200, in = 180] (1.25,2);
    \draw[blue] (1.25,2) to [out = 200, in = 180] (1.25,2.4);
    \draw[blue] (1.25,2.4) to [out = 120, in = 180] (1.35,3.25); 
    \draw[blue,->] (1.35,3.25) to [out = 180, in = 270] (-.5,4); 
    
    \draw[red, very thick, dotted] (1.6,-.5) to [out = 20, in = 0] (1.25,.2); 
    \draw[red, very thick, dotted] (1.25,.2) to [out = -20, in = 0] (1.25,.6);
    \draw[red, very thick, dotted] (1.25,.6) to [out = -20, in = 0] (1.25,1);
    \draw[red, very thick, dotted] (1.25,1) to [out = -20, in = 0] (1.25,1.4);
    \draw[red, very thick, dotted] (1.25,1.4) to [out = -20, in = 0] (1.25,1.8);
    \draw[red, very thick, dotted] (1.25,1.8) to [out = -20, in = 0] (1.25,2.2);
    \draw[red, very thick, dotted] (1.25,2.2) to [out = 220, in = 200] (.6,2.75); 
    \draw[red,->, very thick, dotted] (.6,2.75) to [out = 180, in = 270] (-1,4); 
\end{tikzpicture}
\caption{\(Q_k \in \left[P_1\right]_{x_2}\)}
\end{subfigure}
\caption{Pictured the case when \(i_0\) is odd.}
\label{Case Showing the Pieri Inclusion Removing Many Boxes is a GL(V)-map subsection: Two Box Removal Preserves Garnirs for Hooks Contained in a Single Block equation: T6 figure: [P]}
\end{figure}

Let \(T^\prime \in \mF_{\lambda \setminus X}\) be the unique tableau with \(T^\prime = T_P\) on \((\lambda \setminus X) \setminus [b]\) and \(T^\prime = T\) on \([b]\) except \(T^\prime_{a_0} = u\) and \(T^\prime_{z} = v\).

\begin{center}
\begin{tikzpicture}
    \node at (-1,2) {\(T' = \)};
    
    \draw (0,0) rectangle (3,4);
    \draw (0,-.25) -- (0,0);
    \draw (0,4) -- (0,4.25);
    \draw (2.5,-.25) -- (2.5,0);
    \draw (3,4) -- (3.5,4) -- (3.5,4.25);
    
    \draw (0,2.5) -- (3,2.5);
    \draw (.5,2.5) -- (.5,3);
    \draw (1,2.5) -- (1,3);
    \node at (.25,2.75) {\small \(u\)};
    \node at (.75,2.75) {\small \(v\)};
    
    \draw (0,3) rectangle (.5,3.5);
    \node at (.25,3.25) {\small \(A_1\)};
    \draw (.5,3) rectangle (2.5,3.5);
    \node at (1.5,3.25) {\small \(\cdots\)};
    \draw (2.5,3) rectangle (3,3.5);
    \node at (2.75,3.25) {\small \(A_w\)};
    
    \node at (-.5,4) { \ };
    \node at (-1,4) { \ };
    \node at (.5,-.5) { \ };
    \node at (1.5,-.5) { \ };
\end{tikzpicture}
\end{center}

By the proof of Lemma \ref{Generating Garnir Relations and Tools for Collapsing Sums subsection: Calculation Lemma for 1-paths}, the result of Corollary \ref{Generating Garnir Relations and Tools for Collapsing Sums subsection: Calculation Lemma for 2-paths} still holds when moving \(u\) and \(v\) to boxes in the same row, which we have here after applying \(G_A\). This gives, mod \(F_2 \otimes \mR_{\lambda \setminus X, n}\),
\begin{align*}
\sum_{Q_k \in \left[P_1\right]} Q_k 
	& = \sum_{Q_k \in \left[P_1\right]_{x_1}} Q_k + \sum_{Q_k \in \left[P_1\right]_{x_2}} Q_k \\ 
    & = (-1)^{i_0 + 1 + 1 + 2(i_0 - 1)} \ytableaushort{{\alpha^P_2}, {\alpha^P_1}} \otimes T^\prime
        + (-1)^{i_0 + 1 + 1 + 2(i_0 - 1)} \begin{ytableau} \alpha^P_1 \\ \alpha^P_2 \end{ytableau} \otimes T^\prime\\ 
    & = 0.
\end{align*}
\end{proof}

\begin{case*}[\ref{Showing the Pieri Inclusion Removing Many Boxes is a GL(V)-map subsection: Two Box Removal Preserves Garnirs for Hooks Contained in a Single Block equation: T7}]
In this section we show that the sum over all paths that move \(A_i\) and a box in \([b]\) above \(A\) above \([b]\) is in \(F_2 \otimes \mR_{\lambda \setminus X,n}\).
Note that for any such path, there must be a box \([b](i_0+2,j) \in R^P\).
Recall that
\[
    \sT_7 = \bigsqcup_{\substack{0 \le i \le w_b\\ 1\le j \le w_b}} \sT_7^{i,j},
\]
where
\[
    \sT_7^{i,j} = \{ P_k \in \sT_7  :  P(\sigma^A_k A_i) > [b], [b](i_0+2,j) \in R^P\}.
\]
For \(l = 1, 2\), let \(\sT_{7,x_l}\) be the set of all \(P_k\) in \(\sT_7\) such that the orbit of \(x_l\) intersects the first row in \([b]\),
\[
    \sT_{7,x_l} = \{P_k \in \sT_7  :  R^P_l \cap [b](1) \ne \emptyset\}.
\]
Then 
\[
    \sT_7 = \sT_{7,x_1} \bigsqcup \sT_{7,x_2},
\]
and letting
\[
    \sT_{7,x_l}^{i,j} = \sT_{7,x_l} \bigcap \sT_7^{i,j},
\]
we have
\[
    \sT_7 = \bigsqcup_{\substack{l=1,2, ~ 0 \le i \le w_b\\ 1\le j \le w_b}} \sT_{7,x_l}^{i,j}.
\]
It is then enough to show that
\[
    \sum_{P_k \in \sT_{7,x_1}^{0,1}} \frac{(-1)^{P}}{H(P)} P_k \in F_2 \otimes \mR_{\lambda \setminus X,n},
\]
with \(z = [b](i_0 + 2,1)\) and \(Z = T_{z}\), with the other cases being similar.
\end{case*}

\begin{proof}
For the rest of Case (\ref{Showing the Pieri Inclusion Removing Many Boxes is a GL(V)-map subsection: Two Box Removal Preserves Garnirs for Hooks Contained in a Single Block equation: T7}) let \(\mS\) be the set of all \(P_k \in \sT_{7,x_1}^{0,1}\) that hit a box other than \(a_0\) in row \([b](i_0)\),
\[
    \mS := \{P_k \in \sT_{7,x_1}^{0,1} : [b](i_0,j) \in R^P \text{ for some } 2 \le j \le w_b\},
\]
and let \(\sT := T_{7,x_1}^{0,1} \setminus S\).
One can show 
\[
    \sum_{P_k \in \mS} \frac{(-1)^{P}}{H(P)} P_k \in F_2 \otimes \mR_{\lambda \setminus X,n}    
\]
by following the proof of Case (\ref{Showing the Pieri Inclusion Removing Many Boxes is a GL(V)-map subsection: Two Box Removal Preserves Garnirs for Hooks Contained in a Single Block equation: T6}).
It remains to show
\[
    \sum_{P_k \in \sT} \frac{(-1)^{P}}{H(P)} P_k \in F_2 \otimes \mR_{\lambda \setminus X,n}.
\]

Observe that \(\sT\) is the union of the following disjoint sets:

\begin{align*}
    \sT^1 & = \{P_k \in \sT  :  P(\sigma^A_k A_0), P(Z) \not \in Y\},\\[5pt]
    \sT^2 & = \{P_k \in \sT  :  P(\sigma^A_k A_0) \in Y, P(Z) \not \in Y\},\\[5pt]
    \sT^3 & = \{P_k \in \sT  :  P(Z) \in Y, P(\sigma^A_k A_0) \not \in Y\}, \text{ and }\\[5pt]
    \sT^4 & = \{P_k \in \sT  :  P(\sigma^A_k A_0), P(Z) \in Y\}.
\end{align*}
So, it is enough to show that for \(1 \le i \le 4\),
\[
    \sum_{P_k \in \sT^i} \frac{(-1)^{P}}{H(P)} P_k \in F_2 \otimes \mR_{\lambda \setminus X,n}.
\]

Define the relation \(\sim^i\) on \(\sT^i\), for \(1 \le i \le 4\), as follows.
\begin{itemize}
\item Define \(\sim^1\) on \(\sT^1\) by 
	\begin{align*}
    P_k \sim^1 Q_j %
    	\iff & Q \text{ is a } ([b](i_0),[b](i_0 + 1)) \text{-path extension of } P,\\
        & Q(\sigma^A_j A_0) = P(\sigma^A_k A_0), \text{ and } Q(Z) = P(Z).
    \end{align*}
\item Define \(\sim^2\) on \(\sT^2\) by 
	\begin{align*}
    P_k \sim^2 Q_j %
    	\iff & Q \text{ is a } ([b](i_0),[b](i_0 + 1)) \text{-path extension of } P,\\
        & \text{ and } Q(Z) = P(Z).
    \end{align*}
\item Define \(\sim^3\) on \(\sT^3\) by 
	\begin{align*}
    P_k \sim^3 Q_j %
    	\iff & Q \text{ is a } ([b](i_0),[b](i_0 + 1)) \text{-path extension of } P,\\
        & \text{ and } Q(\sigma^A_j A_0) = P(\sigma^A_k A_0).
    \end{align*}
\item Define \(\sim^4\) on \(\sT^4\) by 
	\[
    P_k \sim^4 Q_j %
    	\iff Q \text{ is a } ([b](i_0),[b](i_0 + 1)) \text{-path extension of } P.
    \]
\end{itemize}

It is clear that for \(1 \le i \le 4\), \(\sim^i\) an equivalence relation on \(\sT^i\), so that
\[
    \sum_{P_k \in \sT^i} \frac{(-1)^{P}}{H(P)} P_k = \sum_{\left[P_k\right] \in \sT^i / \sim^i} %
    \sum_{Q_k \in \left[P_k\right]} \frac{(-1)^{Q}}{H(Q)} Q_k.
\]

Thus it is enough to show that for \(1 \le i \le 4\),
\[
    \sum_{\left[P_k\right] \in \sT^i / \sim^i} \sum_{Q_k \in \left[P_k\right]} \frac{(-1)^{Q}}{H(Q)} Q_k %
    \in F_2 \otimes \mR_{\lambda \setminus X,n}.
\]

We will show the case \(i = 1\), with the other cases being similar. 
Pick \(P_0 \in \sT^1\) with \(A_1 \in R^P\) and let \(u := P^{-1}(A_0)\) and \(v := P^{-1}(A_1)\).
(Note that in the example below, the image of \(Z\) can be in \([b]\).
\begin{center}
\begin{tikzpicture}
    \node at (-1,2) {\(P_0 =\)};
    
    \draw (0,0) rectangle (3,4);
    \draw (0,-.25) -- (0,0);
    \draw (0,4) -- (0,4.25);
    \draw (2.5,-.25) -- (2.5,0);
    \draw (3,4) -- (3.5,4) -- (3.5,4.25);
    
    \draw (0,2) -- (.5,2);
    \draw (.5,2) -- (.5,2.5);
    \node at (.25,2.25) {\small \(A_0\)};
    
    \draw (0,2.5) rectangle (.5,3);
    \node at (.25,2.75) {\small \(A_1\)};
    \draw (.5,2.5) rectangle (2.5,3);
    \node at (1.5,2.75) {\small \(\cdots\)};
    \draw (2.5,2.5) rectangle (3,3);
    \node at (2.75,2.75) {\small \(A_w\)};
    
    \draw (0,3.5) -- (.5,3.5);
    \draw (.5,3) -- (.5,3.5);
    \node at (.25,3.25) {\small \(Z\)};
    
    \draw (0,1.5) rectangle (3,2);
    \draw (0,1) rectangle (3,1.5);
    \node at (.75,1.75) {\(v\)};
    \node at (2,1.25) {\(u\)};
    
    \draw[blue] (.5,-.5) to [out = 200, in = 180] (.5,.2);
    \draw[blue] (.5,.2) to [out = 200, in = 180] (.5,.6);
    \draw[blue] (.5,.6) to [out = 200, in = 180] (.5,1);
    \draw[blue] (.5,1) to [out = 200, in = 180] (.5,1.3);
    \draw[blue] (.5,1.3) to [out = 160, in = 180] (.6,1.75); 
    \draw[blue] (.6,1.75) to [out = 200, in = 180] (.1,2.75); 
    \draw[blue] (.1,2.75) to [out = 200, in = 180] (.1,3.25); 
    \draw[blue,->] (.1,3.25) to [out = 180, in = 270] (-.5,4); 
    
    \draw[red, very thick, dotted] (2.15,-.5) to [out = -20, in = 0] (2.15,.4);
    \draw[red, very thick, dotted] (2.15,.4) to [out = -20, in = 0] (2.15,.8);
    \draw[red, very thick, dotted] (2.15,.8) to [out = -20, in = 0] (2.15,1.25); 
    \draw[red, very thick, dotted] (2.15,1.25) to [out = 20, in = 0] (.1,2.25); 
    \draw[red, ->, very thick, dotted] (.1,2.25) to [out = 180, in = 270] (-1,4); 
\end{tikzpicture}
\end{center}

It is then enough to show that
\[
    \sum_{Q_k \in \left[ P_0 \right]} \frac{(-1)^{Q}}{H(Q)} Q_k \in F_2 \otimes \mR_{\lambda \setminus X,n},
\]
with the other cases being similar. In fact, as \((-1)^P = (-1)^Q\) and \(H(Q) = H(P)\) for all \(Q_k \in \left[P_0\right]\), it is enough to show that
\[
    \sum_{Q_k \in \left[P_0\right]} Q_k \in F_2 \otimes \mR_{\lambda \setminus X,n}.
\]

Let \(\left[ P_0 \right]_1 = \{Q_0 \in \left[P_0\right]\}\) and \(\left[P_0\right]_2 = \{Q_k \in \left[P_0\right]  :  1 \le k \le w_b\}\), so that
\[
    \left[P_0\right] = \left[P_0\right]_1 \bigsqcup \left[P_0\right]_2.
\]
(Note that in the examples below, the image of \(Z\) can be in \([b]\).)

\begin{center}
\begin{tikzpicture}
    \node at (-1, 2) {\(Q_0 =\)};
    
    \draw (0,0) rectangle (3,4);
    \draw (0,-.25) -- (0,0);
    \draw (0,4) -- (0,4.25);
    \draw (2.5,-.25) -- (2.5,0);
    \draw (3,4) -- (3.5,4) -- (3.5,4.25);
    
    \draw (0,2) -- (.5,2);
    \draw (.5,2) -- (.5,2.5);
    \node at (.25,2.25) {\small \(A_0\)};
    
    \draw (0,2.5) rectangle (3,3);
    \draw (1.25,2.5) -- (1.25,3);
    \draw (1.75,2.5) -- (1.75,3);
    \node at (1.5,2.75) {\small \(A_k\)};
    \node at (1.5,3.2) {\small \(k\)};
    
    \draw (0,3.5) -- (.5,3.5);
    \draw (.5,3) -- (.5,3.5);
    \node at (.25,3.25) {\small \(Z\)};
    
    \draw (0,1.5) rectangle (3,2);
    \node at (.75,1.75) {\(v\)};
    \draw (0,1) rectangle (3,1.5);
    \node at (2,1.25) {\(u\)};
    
    \draw[blue] (.5,-.5) to [out = 200, in = 180] (.5,.2);
    \draw[blue] (.5,.2) to [out = 200, in = 180] (.5,.6);
    \draw[blue] (.5,.6) to [out = 200, in = 180] (.5,1);
    \draw[blue] (.5,1) to [out = 200, in = 180] (.5,1.3);
    \draw[blue] (.5,1.3) to [out = 160, in = 180] (.6,1.75); 
    \draw[blue] (.6,1.75) to [out = 120, in = 180] (1.35,2.75); 
    \draw[blue] (1.35,2.75) to [out = 180, in = 270] (.1,3.25); 
    \draw[blue,->] (.1,3.25) to [out = 180, in = 270] (-.5,4); 
    
    \draw[red, very thick, dotted] (2.15,-.5) to [out = -20, in = 0] (2.15,.4);
    \draw[red, very thick, dotted] (2.15,.4) to [out = -20, in = 0] (2.15,.8);
    \draw[red, very thick, dotted] (2.15,.8) to [out = -20, in = 0] (2.15,1.25); 
    \draw[red, very thick, dotted] (2.15,1.25) to [out = 20, in = 0] (.1,2.25); 
    \draw[red, ->, very thick, dotted] (.1,2.25) to [out = 180, in = 270] (-1,4); 
\end{tikzpicture}
\hspace{1in}
\begin{tikzpicture}
    \node at (-1, 2) {\(Q_k =\)};
    
    \draw (0,0) rectangle (3,4);
    \draw (0,-.25) -- (0,0);
    \draw (0,4) -- (0,4.25);
    \draw (2.5,-.25) -- (2.5,0);
    \draw (3,4) -- (3.5,4) -- (3.5,4.25);
    
    \draw (0,2) -- (.5,2);
    \draw (.5,2) -- (.5,2.5);
    \node at (.25,2.25) {\small \(A_k\)};
    
    \draw (0,2.5) rectangle (3,3);
    \draw (1.25,2.5) -- (1.25,3);
    \draw (1.75,2.5) -- (1.75,3);
    \node at (1.5,2.75) {\small \(A_0\)};
    \node at (1.5,3.2) {\small \(k\)};
    
    \draw (0,3.5) -- (.5,3.5);
    \draw (.5,3) -- (.5,3.5);
    \node at (.25,3.25) {\small \(Z\)};
    
    \draw (0,1.5) rectangle (3,2);
    \node at (.75,1.75) {\(v\)};
    \draw (0,1) rectangle (3,1.5);
    \node at (2,1.25) {\(u\)};
    
    \draw[blue] (.5,-.5) to [out = 200, in = 180] (.5,.2);
    \draw[blue] (.5,.2) to [out = 200, in = 180] (.5,.6);
    \draw[blue] (.5,.6) to [out = 200, in = 180] (.5,1);
    \draw[blue] (.5,1) to [out = 200, in = 180] (.5,1.3);
    \draw[blue] (.5,1.3) to [out = 160, in = 180] (.6,1.75); 
    \draw[blue] (.6,1.75) to [out = 120, in = 180] (1.35,2.75); 
    \draw[blue,->] (1.35,2.75) to [out = 180, in = 270] (-1,4); 
    
    \draw[red, very thick, dotted] (2.15,-.5) to [out = -20, in = 0] (2.15,.4);
    \draw[red, very thick, dotted] (2.15,.4) to [out = -20, in = 0] (2.15,.8);
    \draw[red, very thick, dotted] (2.15,.8) to [out = -20, in = 0] (2.15,1.25); 
    \draw[red, very thick, dotted] (2.15,1.25) to [out = 20, in = 0] (.1,2.25); 
    \draw[red, very thick, dotted] (.1,2.25) to [out = 200, in = 180] (.1,3.25); 
    \draw[red, ->, very thick, dotted] (.1,3.25) to [out = 180, in = 270] (-.5,4); 
\end{tikzpicture}
\end{center}

Let \(T^\prime \in \mF_{\lambda \setminus X}\) be the unique tableau with \(T^\prime = T_P\) on \((\lambda \setminus X) \{([b](i_0), [b](i_0 + 2))\}\) and \(T^\prime = T\) on \(([b](i_0), [b](i_0 + 2))\) except \(T^\prime_{z} = v\) and \(T^\prime_{a_0} = u\).
\begin{center}
\begin{tikzpicture}
    \node at (-1,2) {\(T' =\)};

    \draw (0,0) rectangle (3,4);
    \draw (0,-.25) -- (0,0);
    \draw (0,4) -- (0,4.25);
    \draw (2.5,-.25) -- (2.5,0);
    \draw (3,4) -- (3.5,4) -- (3.5,4.25);
    
    \draw (0,2) -- (.5,2);
    \draw (.5,2) -- (.5,2.5);
    \node at (.25,2.25) {\small \(u\)};
    
    \draw (0,2.5) rectangle (.5,3);
    \node at (.25,2.75) {\small \(A_1\)};
    \draw (.5,2.5) rectangle (2.5,3);
    \node at (1.5,2.75) {\small \(\cdots\)};
    \draw (2.5,2.5) rectangle (3,3);
    \node at (2.75,2.75) {\small \(A_w\)};
    
    \draw (0,3.5) -- (.5,3.5);
    \draw (.5,3) -- (.5,3.5);
    \node at (.25,3.25) {\small \(v\)};
    
    \draw (0,1.5) rectangle (3,2);
    \draw (0,1) rectangle (3,1.5);
    
    \node at (-.5,4) { \ };
    \node at (-1,4) { \ };
    
    \draw[blue] (.5,-.5) to [out = 200, in = 180] (.5,.2);
    \draw[blue] (.5,.2) to [out = 200, in = 180] (.5,.6);
    \draw[blue] (.5,.6) to [out = 200, in = 180] (.5,1);
    \draw[blue] (.5,1) to [out = 200, in = 180] (.5,1.3);
    \draw[blue] (.5,1.3) to [out = 160, in = 180] (.6,1.75); 
    
    \draw[red, very thick, dotted] (2.15,-.5) to [out = -20, in = 0] (2.15,.4);
    \draw[red, very thick, dotted] (2.15,.4) to [out = -20, in = 0] (2.15,.8);
    \draw[red, very thick, dotted] (2.15,.8) to [out = -20, in = 0] (2.15,1.25); 
\end{tikzpicture}
\end{center}

Then by Corollary \ref{Generating Garnir Relations and Tools for Collapsing Sums subsection: Calculation Lemma for 2-paths} and applications of \(G_A\) we have, mod \(F_2 \otimes \mR_{\lambda \setminus X, n}\)
\[
    \sum_{Q_k \in E^P} Q_K 
        = \sum_{Q_k \in \left[ P_0 \right]_1} Q_K + \ds \sum_{Q_k \in \left[ P_0 \right]_2} Q_K 
        = - ~ \ytableaushort{{\alpha^P_2}, {\alpha^P_1}} \otimes T^\prime - ~ \begin{ytableau} \alpha^P_1 \\ \alpha^P_2 \end{ytableau} \otimes T^\prime 
        = 0.
\]
\end{proof}

\subsection{Two Box Removal Preserves Garnirs for Hooks Contained in Two Blocks} \label{Showing the Pieri Inclusion Removing Many Boxes is a GL(V)-map subsection: Two Box Removal Preserves Garnirs for Hooks Contained in Two Blocks}

We now show that Equation \ref{Showing the Pieri Inclusion Removing Many Boxes is a GL(V)-map subsection: The Theorem equation: Phi of Garnir m Boxes} holds when \(m = 2\) for all hooks \(A \subset [b] \cup [b + 1]\) for some \(1 \le b \le N - 1\). 
For the rest of Section \ref{Showing the Pieri Inclusion Removing Many Boxes is a GL(V)-map subsection: Two Box Removal Preserves Garnirs for Hooks Contained in Two Blocks}, fix \(T \in \mF_{\lambda, n}\) and let
\[
    A = \{ a_0 := [b](h_b,1), a_1 := [b + 1](1,1), \ldots, a_{w_{b + 1}} := [b + 1](1,w_{b + 1})\} \subset T_0    
\]
so that \(A \subset [b] \cup [b + 1]\). Denote the entries of \(A\) in \(T\) by \(A_k = T_{a_k}\) for \(k = 0, 1, \ldots, w_{b + 1}\). 
Then by Lemma \ref{Generating Garnir Relations and Tools for Collapsing Sums subsection: Simplifying a Garnir}, mod \(F_2 \otimes \mR_{\lambda \setminus X,n}\) we have
\begin{align*}
    \Phi_2\left(G_A (T) \right) %
        & = \sum_{P} \frac{(-1)^{P}}{H(P)} P\left(\sum_{\sigma \in S_A} \sigma T \right) \\
        & = C \sum_{P} \sum_{k=0}^{w_{b + 1}} \frac{(-1)^{P}}{H(P)} P\left(\sigma^A_k T \right),
\end{align*}
where the sum is over all \(2\)-paths \(P\) on \(\lambda\) removing \(X\).
The set of all \(P_k := P(\sigma^A_k T)\) appearing in the image \(\Phi_2\left(G_A (T) \right)\) above is the union of the following disjoint sets.

The \(P_k\)s that miss \(A\),
\begin{equation} \label{Showing the Pieri Inclusion Removing Many Boxes is a GL(V)-map subsection: Two Box Removal Preserves Garnirs for Hooks Contained in Two Blocks equation: T1}
    \sT_1 = \{P_k  :  R^P \cap A = \emptyset\}.
\end{equation}

The \(P_k\)s that hit \(A\) and keep \(A\) in \([b] \cup [b + 1]\),
\begin{equation} \label{Showing the Pieri Inclusion Removing Many Boxes is a GL(V)-map subsection: Two Box Removal Preserves Garnirs for Hooks Contained in Two Blocks equation: T2}
    \sT_2 = \{P_k  :  R^P \cap A \ne \emptyset, ~ P(A) \le [b + 1]\}.
\end{equation}

The \(P_k\)s that have exactly one orbit in \([b] \cup [b + 1]\) and move \(A_i\) above \([b + 1]\),
\begin{equation} \label{Showing the Pieri Inclusion Removing Many Boxes is a GL(V)-map subsection: Two Box Removal Preserves Garnirs for Hooks Contained in Two Blocks equation: T3}
    \begin{split}
        \sT_3 & = \bigsqcup_{i=0}^{w_{b + 1}} \sT_3^{i}, \text{ where }\\
    	\sT_3^i & = %
    	    \begin{split}
    	        \{P_k \in \sT_3 : &\text{exactly one of } R^P_{x_1}, R^P_{x_2}, \text{ intersect } [b] \cup [b + 1]\\ & \text{ and } P(\sigma^A_k A_i) > [b + 1]\}.
    	    \end{split}
    \end{split}
\end{equation}

The \(P_k\)s that move \(A_i\) and \(A_j\) above \([b + 1]\),
\begin{equation} \label{Showing the Pieri Inclusion Removing Many Boxes is a GL(V)-map subsection: Two Box Removal Preserves Garnirs for Hooks Contained in Two Blocks equation: T4}
    \begin{split}
        \sT_4 & = \bigsqcup_{0 \le i < j \le w_{b + 1}} \sT_4^{i,j}, \text{ where }\\
    	\sT_4^{i,j} & = \{P_k : P(\sigma^A_k A_i) > [b + 1], \text{ and } P(\sigma^A_k A_j) > [b + 1]\}.
    \end{split}
\end{equation}

The \(P_k\)s that move \(A_i\) and a box \(Z\) in \([b]\) below \(A\) above \([b + 1]\),
\begin{equation} \label{Showing the Pieri Inclusion Removing Many Boxes is a GL(V)-map subsection: Two Box Removal Preserves Garnirs for Hooks Contained in Two Blocks equation: T5}
    \begin{split}
        \sT_5 & = \bigsqcup_{\substack{0 \le i \le w_{b + 1}\\ z = [b](j,k), ~ 1\le j < h_b \text{ and } 1 \le k \le w_b}} \sT_5^{i,z}, \text{ where }\\
    	\sT_5^{i,z} & = \{ P_k : P(\sigma^A_k A_i) > [b + 1], P(z) > [b + 1]\}.
    \end{split}
\end{equation}

The \(P_k\)s that move \(A_i\) and a box other than \(a_0\) in row \([b](h_b)\) above \([b + 1]\),
\begin{equation} \label{Showing the Pieri Inclusion Removing Many Boxes is a GL(V)-map subsection: Two Box Removal Preserves Garnirs for Hooks Contained in Two Blocks equation: T6}
    \begin{split}
        \sT_6 & = \bigsqcup_{0 \le i \le w_{b + 1}, ~ 2 \le j \le w_b} \sT_6^{i,j}, \text{ where }\\
    	\sT_6^{i,j} & = \{ P_k : P(\sigma^A_k A_i) > [b + 1], P([b](h_b,j))> [b + 1]\}.
    \end{split}
\end{equation}

The \(P_k\)s that move \(A_i\) and a box above \(A\) above \([b + 1]\),
\begin{equation} \label{Showing the Pieri Inclusion Removing Many Boxes is a GL(V)-map subsection: Two Box Removal Preserves Garnirs for Hooks Contained in Two Blocks equation: T7}
    \begin{split}
        \sT_7 & = \bigsqcup_{\substack{0 \le i \le w_{b + 1}\\ 1 \le j \le w_{b + 1}}} \sT_7^{i,j}, \text{ where }\\
    	\sT_7^{i,z} & = \{ P_k : P(\sigma^A_k A_i) > [b + 1], P([b + 1](2,j)) \in R^P\}.
    \end{split}
\end{equation}

Then we have
\begin{align*}
	\Phi_2\left(G_A (T)\right) %
    	& = C \sum_{P} \sum_{k = 0}^{w_{b + 1}} \frac{(-1)^{P}}{H(P)} P\left(\sigma^A_k T \right) \mod F_2 \otimes \mR_{\lambda \setminus X,n}\\
        & = C \sum_{j=1,\ldots,7} \sum_{P_k \in \sT_j} \frac{(-1)^{P}}{H(P)} P_k \mod F_2 \otimes \mR_{\lambda \setminus X,n}.\\
\end{align*}

We show that for \(1 \le j \le 7\), 
\begin{align*}
    \sum_{P_k \in T_j} \frac{(-1)^{P}}{H(P)} P_k \in F_2 \otimes \mR_{\lambda \setminus X,n},
\end{align*}
and hence Equation \ref{Showing the Pieri Inclusion Removing Many Boxes is a GL(V)-map subsection: The Theorem equation: Phi of Garnir m Boxes} holds when \(m = 2\) for all blocks \(A \subset [b] \cup [b + 1]\).

The proofs of Case (\ref{Showing the Pieri Inclusion Removing Many Boxes is a GL(V)-map subsection: Two Box Removal Preserves Garnirs for Hooks Contained in Two Blocks equation: T1}) and Case (\ref{Showing the Pieri Inclusion Removing Many Boxes is a GL(V)-map subsection: Two Box Removal Preserves Garnirs for Hooks Contained in Two Blocks equation: T2}) are similar to the proofs of Case (\ref{Showing the Pieri Inclusion Removing One Box is a GL(V)-map subsection: One Box Removal Preserves Garnirs for Hooks Contained in a Single Block equation: T1}) and Case (\ref{Showing the Pieri Inclusion Removing One Box is a GL(V)-map subsection: One Box Removal Preserves Garnirs for Hooks Contained in a Single Block equation: T2}), respectively. 
The proof of Case (\ref{Showing the Pieri Inclusion Removing Many Boxes is a GL(V)-map subsection: Two Box Removal Preserves Garnirs for Hooks Contained in Two Blocks equation: T3}) is similar to the proof of Case (\ref{Showing the Pieri Inclusion Removing One Box is a GL(V)-map subsection: One Box Removal Preserves Garnirs for Hooks Contained in a Two Blocks equation: T3}), and goes through by observing that using the definition of \(H(P)\) for a \(2\)-path only adds and subtracts \(1\) in some of the terms.
The proofs of Case (\ref{Showing the Pieri Inclusion Removing Many Boxes is a GL(V)-map subsection: Two Box Removal Preserves Garnirs for Hooks Contained in Two Blocks equation: T4}) and Case (\ref{Showing the Pieri Inclusion Removing Many Boxes is a GL(V)-map subsection: Two Box Removal Preserves Garnirs for Hooks Contained in Two Blocks equation: T7}) are similar to the proofs of Case (\ref{Showing the Pieri Inclusion Removing Many Boxes is a GL(V)-map subsection: Two Box Removal Preserves Garnirs for Hooks Contained in a Single Block equation: T4}) and Case (\ref{Showing the Pieri Inclusion Removing Many Boxes is a GL(V)-map subsection: Two Box Removal Preserves Garnirs for Hooks Contained in a Single Block equation: T7}), respectively, as these proofs did not depend on \(H(P)\).
It remains to show Case (\ref{Showing the Pieri Inclusion Removing Many Boxes is a GL(V)-map subsection: Two Box Removal Preserves Garnirs for Hooks Contained in Two Blocks equation: T5}) and Case (\ref{Showing the Pieri Inclusion Removing Many Boxes is a GL(V)-map subsection: Two Box Removal Preserves Garnirs for Hooks Contained in Two Blocks equation: T6}). 
In both cases we assume \(b > b_1\) and \(A \cap X = \emptyset\), as if \(b = b_1\) or if \(A \cap X \ne \emptyset\) we may follow the proof of Subcase (\ref{Showing the Pieri Inclusion Removing One Box is a GL(V)-map subsection: One Box Removal Preserves Garnirs for Hooks Contained in a Two Blocks equation: T3}.1).

\begin{case*}[\ref{Showing the Pieri Inclusion Removing Many Boxes is a GL(V)-map subsection: Two Box Removal Preserves Garnirs for Hooks Contained in Two Blocks equation: T5}]
In this case we show that the sum over all paths that move \(A_i\) and a box \(Z\) in \([b]\) below \(A\) above \([b + 1]\) is in \(F_2 \otimes \mR_{\lambda \setminus X,n}\).
Recall that
\[
    \sT_5 = \bigsqcup_{\substack{0 \le i \le w_{b + 1}\\ z = [b](j,k), ~ 1\le j < h_b \text{ and } 1 \le k \le w_b}} \sT_5^{i,z},
\]
where
\[
    \sT_5^{i,z} = \{ P_k : P(\sigma^A_k A_i) > [b + 1], P(z) > [b + 1]\}.
\]
It is enough to show that
\[
    \sum_{P_k \in \sT_5^{0,z}} \frac{(-1)^{P}}{H(P)} P_k \in F_2 \otimes \mR_{\lambda \setminus X,n},
\]
where \(z = [b](i_z,j_z)\) is a fixed box with \(1 \le i_z \le h_b - 1\) odd, \(1 \le j_z \le w_{b + 1}\), and \(Z = T_z\), with the other cases being similar.
\end{case*}

\begin{proof}
For the rest of Case (\ref{Showing the Pieri Inclusion Removing Many Boxes is a GL(V)-map subsection: Two Box Removal Preserves Garnirs for Hooks Contained in Two Blocks equation: T5}), let \(\sT := \sT_5^{0,z}\) and, for any \(2\)-path \(P\) on \(\lambda\) removing \(X\), let \(\tilde{h^P} = h^P - h^P_b - h^P_{b + 1}\) and \(\tilde{H(P)} = \dfrac{H(P)}{H_b(P)H_{b + 1}(P)}\).
Observe that \(\sT\) is the union of the following disjoint sets:
\begin{align*}
    \sT^1 & = \{P_k \in \sT : P(\sigma^A_k a_0), P(z) \not \in Y\},\\[5pt]
    \sT^2 & = \{P_k \in \sT : P(\sigma^A_k a_0) \in Y, P(z) \not \in Y\},\\[5pt]
    \sT^3 & = \{P_k \in \sT : P(z) \in Y, P(\sigma^A_k a_0) \not \in Y\}, \text{ and }\\[5pt]
    \sT^4 & = \{P_k \in \sT : P(\sigma^A_k a_0), P(z) \in Y\}.
\end{align*}

So it is enough to show that for \(1 \le i \le 4\),
\[
    \sum_{P_k \in \sT^i} \frac{(-1)^{P}}{H(P)} P_k \in F_2 \otimes \mR_{\lambda \setminus X,n}.
\]

Define the relation \(\sim^i\) on \(\sT^i\), for \(1 \le i \le 4\), as follows.
Define \(\sim^1\) on \(\sT^1\) by 
	\begin{align*}
        P_k \sim^1 Q_j %
    	    \iff & Q \text{ is a } ([b](1), [b + 1](1)) \text{-path extension of } P,\\
            & Q(\sigma^A_j a_0) = P(\sigma^A_k a_0), \text{ and } Q(z) = P(z).
    \end{align*}
Define \(\sim^2\) on \(\sT^2\) by 
	\begin{align*}
        P_k \sim Q_j %
    	    \iff & Q \text{ is a } ([b](1), [b + 1](1)) \text{-path extension of } P,\\
            & \text{ and } Q(z) = P(z).
    \end{align*}
Define \(\sim^3\) on \(\sT^3\) by 
	\begin{align*}
        P_k \sim^3 Q_j %
    	    \iff & Q \text{ is a } ([b](1), [b + 1](1)) \text{-path extension of } P,\\
            & \text{ and } Q(\sigma^A_j a_0) = P(\sigma^A_k a_0).
    \end{align*}
Define \(\sim^4\) on \(\sT^4\) by 
	\[
        P_k \sim^4 Q_j \iff Q \text{ is a } ([b](1), [b + 1](1)) \text{-path extension of } P.
    \]

It is clear that for \(1 \le i \le 4\), \(\sim^i\) an equivalence relation on \(\sT^i\), so that
\[
    \sum_{P_k \in \sT^i} \frac{(-1)^{P}}{H(P)} P_k = \sum_{\left[P_k\right] \in \sT^i / \sim^i} %
    \sum_{Q_k \in \left[P_k\right]} \frac{(-1)^{Q}}{H(Q)} Q_k.
\]
Thus it is enough to show that, for \(1 \le i \le 4\),
\[
    \sum_{\left[P_k\right] \in \sT^i / \sim^i} \sum_{Q_k \in \left[P_k\right]} \frac{(-1)^{Q}}{H(Q)} %
    Q_k \in F_2 \otimes \mR_{\lambda \setminus X, n}.
\]
We will show the case \(i = 1\), with the rest being similar.
Pick \(P_0 \in \sT^1\) with \([b](i,1) \in R^P\) for all \(1 \le i \ne i_z \le h_b\), and let \([b_u](i_u,j_u) = P^{-1}([b](1,1))\) and \([b_v](i_v,j_v) = P^{-1}([b](2,1))\) with \(u = T_{[b_u](i_u,j_u)}\) and \(v = T_{[b_v](i_v,j_v)}\).

\begin{center}
\begin{tikzpicture}
    \node at (-1.5,2) {\(P_0 = \)};
    
    \draw (0,0) rectangle (2,3);
    \draw (0,-.25) -- (0,0);
    \draw (1.5,-.25) -- (1.5,0);
    
    \draw (0,3) rectangle (3,4);
    \draw (0,4) -- (0,4.25);
    \draw (3,4) -- (3.5,4) -- (3.5,4.25);
    
    \draw (0,2.5) rectangle (.5,3);
    \draw (.5,0) -- (.5,2.5);
    \node at (.25,2.75) {\small \(A_0\)};
    
    \draw (0,3) rectangle (.5,3.5);
    \node at (.25,3.25) {\small \(A_1\)};
    \draw (.5,3) rectangle (2.5,3.5);
    \node at (1.5,3.25) {\small \(\cdots\)};
    \draw (2.5,3) rectangle (3,3.5);
    \node at (2.75,3.25) {\small \(A_w\)};
    
    \draw (0,1.25) rectangle (2,1.75);
    \draw (.75,1.25) -- (.75,1.75);
    \draw (1.25,1.25) -- (1.25,1.75);
    \node at (1,1.5) {\small \(Z\)};
    
    \node at (.5,-.5) {\(u\)};
    \node at (1.5,-.5) {\(v\)};
    
    \draw[blue] (.4,-.5) to [out = 200, in = 180] (.2,.2); 
    \draw[blue] (.2,.2) to [out = 200, in = 180] (.2,.6);
    \draw[blue] (.2,.6) to [out = 200, in = 180] (.2,1);
    \draw[blue] (.2,1) to [out = 140, in = 180] (.85,1.5); 
    \draw[blue,->] (.85,1.5) to [out = 180, in = 270] (-1,4); 
    
    \draw[red, very thick, dotted] (1.6,-.5) to [out = 20, in = 0] (.2,.4); 
    \draw[red, very thick, dotted] (.2,.4) to [out = -20, in = 0] (.2,.8);
    \draw[red, very thick, dotted] (.2,.8) to [out = -20, in = 0] (.2,1.2);
    \draw[red, very thick, dotted] (.2,1.2) to [out = -20, in = 0] (.2,1.8); 
    \draw[red, very thick, dotted] (.2,1.8) to [out = -20, in = 0] (.2,2);
    \draw[red, very thick, dotted] (.2,2) to [out = -20, in = 0] (.2,2.2);
    \draw[red, very thick, dotted] (.2,2.2) to [out = -20, in = 0] (.2,2.4);
    \draw[red, very thick, dotted] (.2,2.4) to [out = 200, in = 270] (.1,2.75); 
    \draw[red,->, very thick, dotted] (.1,2.75) to [out = 180, in = 270] (-.5,4); 
\end{tikzpicture}
\end{center}

It is then enough to show that
\[
   \sum_{Q_k \in \left[ P_0 \right]} \frac{(-1)^{Q}}{H(Q)} Q_k \in F_2 \otimes \mR_{\lambda \setminus X, n}.
\]
In fact, as \(\tilde{h^Q} = \tilde{h^P}\) and \(\tilde{H(Q)} = \tilde{H(P)}\) for all \(Q_k \in \left[ P_0 \right]\), it is enough to show that
\[
    \sum_{Q_k \in \left[ P_0 \right]} \frac{(-1)^{h^Q_b + h^Q_{b + 1}}}{H^Q_b H^Q_{b + 1}} Q_k %
    \in F_2 \otimes \mR_{\lambda \setminus X,n}.
\]
Observe that \(\left[ P_0 \right]\) can be written as the disjoint union 
\[
    \left[ P_0 \right] = \bigsqcup_{i = 1}^7 \left[ P_0 \right]_i,
\]
where the \(\left[ P_0 \right]_i\) are defined as follows.\\

The paths acting on \(\sigma_0 T\), 
\begin{equation*}
    \left[ P_0 \right]_1 = \{Q_0 \in \left[ P_0 \right]\},
\end{equation*}
\begin{center}
    \begin{tikzpicture}
        \draw (0,0) rectangle (2,3);
        \draw (0,-.25) -- (0,0);
        \draw (1.5,-.25) -- (1.5,0);
        
        \draw (0,3) rectangle (3,4);
        \draw (0,4) -- (0,4.25);
        \draw (3,4) -- (3.5,4) -- (3.5,4.25);
        
        \draw (0,2.5) rectangle (.5,3);
        \node at (.25,2.75) {\small \(A_0\)};
        
        \draw (0,3) rectangle (.5,3.5);
        \node at (.25,3.25) {\small \(A_1\)};
        \draw (.5,3) rectangle (2.5,3.5);
        \node at (1.5,3.25) {\small \(\cdots\)};
        \draw (2.5,3) rectangle (3,3.5);
        \node at (2.75,3.25) {\small \(A_w\)};
        
        \draw (0,1.25) rectangle (2,1.75);
        \draw (.75,1.25) -- (.75,1.75);
        \draw (1.25,1.25) -- (1.25,1.75);
        \node at (1,1.5) {\small \(Z\)};
        
        \node at (.5,-.5) {\(u\)};
        \node at (1.5,-.5) {\(v\)};
        
        \draw[blue] (.4,-.5) to [out = 200, in = 180] (.75,.2); 
        \draw[blue] (.75,.2) to [out = 200, in = 180] (.75,.6);
        \draw[blue] (.75,.6) to [out = 200, in = 180] (.75,1);
        \draw[blue] (.75,1) to [out = 140, in = 180] (.85,1.5); 
        \draw[blue,->] (.85,1.5) to [out = 180, in = 270] (-1,4); 
        
        \draw[red, very thick, dotted] (1.6,-.5) to [out = 20, in = 0] (.75,.4); 
        \draw[red, very thick, dotted] (.75,.4) to [out = -20, in = 0] (.75,.8);
        \draw[red, very thick, dotted] (.75,.8) to [out = -20, in = 0] (.75,1.2);
        \draw[red, very thick, dotted] (.75,1.2) to [out = -20, in = 0] (.75,1.8); 
        \draw[red, very thick, dotted] (.75,1.8) to [out = -20, in = 0] (.75,2);
        \draw[red, very thick, dotted] (.75,2) to [out = -20, in = 0] (.75,2.2);
        \draw[red, very thick, dotted] (.75,2.2) to [out = -20, in = 0] (.75,2.4);
        \draw[red, very thick, dotted] (.75,2.4) to [out = -20, in = 270] (.1,2.75); 
        \draw[red,->, very thick, dotted] (.1,2.75) to [out = 180, in = 270] (-.5,4); 
    \end{tikzpicture}
\end{center}
the paths acting on \(\sigma^A_k T\), \(k \ne 0\), that hit \(\sigma^A_k a_k\), 
\begin{equation*}
    \left[ P_0 \right]_2 = \{Q_k \in \left[ P_0 \right]  :  k \ne 0, a_0 \in R^Q\},
\end{equation*}
\begin{center}
    \begin{tikzpicture}
        \draw (0,0) rectangle (2,3);
        \draw (0,-.25) -- (0,0);
        \draw (1.5,-.25) -- (1.5,0);
        
        \draw (0,3) rectangle (3,4);
        \draw (0,4) -- (0,4.25);
        \draw (3,4) -- (3.5,4) -- (3.5,4.25);
        
        \draw (0,2.5) rectangle (.5,3);
        \node at (.25,2.75) {\small \(A_k\)};
        
        \draw (0,3) rectangle (3,3.5);
        \draw (1.25,3) -- (1.25,3.5);
        \draw (1.75,3) rectangle (1.75,3.5);
        \node at (1.5,3.25) {\small \(A_0\)};
        \node at (1.5,3.7) {\small \(k\)};
        
        \draw (0,1.25) rectangle (2,1.75);
        \draw (.75,1.25) -- (.75,1.75);
        \draw (1.25,1.25) -- (1.25,1.75);
        \node at (1,1.5) {\small \(Z\)};
        
        \node at (.5,-.5) {\(u\)};
        \node at (1.5,-.5) {\(v\)};
        
        \draw[blue] (.4,-.5) to [out = 200, in = 180] (.75,.2); 
        \draw[blue] (.75,.2) to [out = 200, in = 180] (.75,.6);
        \draw[blue] (.75,.6) to [out = 200, in = 180] (.75,1);
        \draw[blue] (.75,1) to [out = 140, in = 180] (.85,1.5); 
        \draw[blue,->] (.85,1.5) to [out = 180, in = 270] (-1,4); 
        
        \draw[red, very thick, dotted] (1.6,-.5) to [out = 20, in = 0] (.75,.4); 
        \draw[red, very thick, dotted] (.75,.4) to [out = -20, in = 0] (.75,.8);
        \draw[red, very thick, dotted] (.75,.8) to [out = -20, in = 0] (.75,1.2);
        \draw[red, very thick, dotted] (.75,1.2) to [out = -20, in = 0] (.75,1.8); 
        \draw[red, very thick, dotted] (.75,1.8) to [out = -20, in = 0] (.75,2);
        \draw[red, very thick, dotted] (.75,2) to [out = -20, in = 0] (.75,2.2);
        \draw[red, very thick, dotted] (.75,2.2) to [out = -20, in = 0] (.75,2.4);
        \draw[red, very thick, dotted] (.75,2.4) to [out = 200, in = 270] (.1,2.75); 
        \draw[red, very thick, dotted] (.1,2.75) to [out = 120, in = 180] (1.35,3.25); 
        \draw[red,->, very thick, dotted] (1.35,3.25) to [out = 180, in = 270] (-.5,4); 
    \end{tikzpicture}
\end{center}
the paths acting on \(\sigma^A_k T\), \(k \ne 0\), that miss \(\sigma^A_k a_k\) but hit row \([b](h_b)\), 
\begin{equation*}
    \left[ P_0 \right]_3 = \{Q_k \in \left[ P_0 \right]  :  k \ne 0, [b](h_b,j) \in R^Q \text{ for some } 2 \le j \le w_b\},
\end{equation*}
\begin{center}
    \begin{tikzpicture}
        \draw (0,0) rectangle (2,3);
        \draw (0,-.25) -- (0,0);
        \draw (1.5,-.25) -- (1.5,0);
        
        \draw (0,3) rectangle (3,4);
        \draw (0,4) -- (0,4.25);
        \draw (3,4) -- (3.5,4) -- (3.5,4.25);
        
        \draw (0,2.5) -- (2,2.5);
        \draw (.5,2.5) -- (.5,3);
        \node at (.25,2.75) {\small \(A_k\)};
        
        \draw (0,3) rectangle (3,3.5);
        \draw (1.25,3) -- (1.25,3.5);
        \draw (1.75,3) rectangle (1.75,3.5);
        \node at (1.5,3.25) {\small \(A_0\)};
        \node at (1.5,3.7) {\small \(k\)};
        
        \draw (0,1.25) rectangle (2,1.75);
        \draw (.75,1.25) -- (.75,1.75);
        \draw (1.25,1.25) -- (1.25,1.75);
        \node at (1,1.5) {\small \(Z\)};
        
        \node at (.5,-.5) {\(u\)};
        \node at (1.5,-.5) {\(v\)};
        
        \draw[blue] (.4,-.5) to [out = 200, in = 180] (.75,.2); 
        \draw[blue] (.75,.2) to [out = 200, in = 180] (.75,.6);
        \draw[blue] (.75,.6) to [out = 200, in = 180] (.75,1);
        \draw[blue] (.75,1) to [out = 140, in = 180] (.85,1.5); 
        \draw[blue,->] (.85,1.5) to [out = 180, in = 270] (-1,4); 
        
        \draw[red, very thick, dotted] (1.6,-.5) to [out = 20, in = 0] (.75,.4); 
        \draw[red, very thick, dotted] (.75,.4) to [out = -20, in = 0] (.75,.8);
        \draw[red, very thick, dotted] (.75,.8) to [out = -20, in = 0] (.75,1.2);
        \draw[red, very thick, dotted] (.75,1.2) to [out = -20, in = 0] (.75,1.8); 
        \draw[red, very thick, dotted] (.75,1.8) to [out = -20, in = 0] (.75,2);
        \draw[red, very thick, dotted] (.75,2) to [out = -20, in = 0] (.75,2.2);
        \draw[red, very thick, dotted] (.75,2.2) to [out = -20, in = 0] (.75,2.4);
        \draw[red, very thick, dotted] (.75,2.4) to [out = -20, in = 0] (1,2.75); 
        \draw[red, very thick, dotted] (1,2.75) to [out = -20, in = 270] (1.35,3.25); 
        \draw[red,->, very thick, dotted] (1.35,3.25) to [out = 180, in = 270] (-.5,4); 
    \end{tikzpicture}
\end{center}
the paths acting on \(\sigma^A_k T\), \(k \ne 0\), that miss row \([b](h_b)\) with \(R^P_2\) leaving \([b]\) in a row above \(Z\),
\begin{equation*}
    \left[ P_0 \right]_4 = \{Q_k \in \left[ P_0 \right]  :  k \ne 0, Q([b](i,j)) = a_k \text{ for some } i_z < i < h_b, 1 \le j \le w_b\},
\end{equation*}
\begin{center}
    \begin{tikzpicture}
        \draw (0,0) rectangle (2,3);
        \draw (0,-.25) -- (0,0);
        \draw (1.5,-.25) -- (1.5,0);
        
        \draw (0,3) rectangle (3,4);
        \draw (0,4) -- (0,4.25);
        \draw (3,4) -- (3.5,4) -- (3.5,4.25);
        
        \draw (0,2.5) rectangle (.5,3);
        \node at (.25,2.75) {\small \(A_k\)};
        
        \draw (0,3) rectangle (3,3.5);
        \draw (1.25,3) -- (1.25,3.5);
        \draw (1.75,3) rectangle (1.75,3.5);
        \node at (1.5,3.25) {\small \(A_0\)};
        \node at (1.5,3.7) {\small \(k\)};
        
        \draw (0,.75) rectangle (2,1.25);
        \draw (.75,.75) -- (.75,1.25);
        \draw (1.25,.75) -- (1.25,1.25);
        \node at (1,1) {\small \(Z\)};
        
        \draw (0,1.5) rectangle (2,2);
        \node at (2.7,1.75) {\small \(\leftarrow\) row \(i\)};
        
        \node at (.5,-.5) {\(u\)};
        \node at (1.5,-.5) {\(v\)};
        
        \draw[blue] (.4,-.5) to [out = 200, in = 180] (.75,.2); 
        \draw[blue] (.75,.2) to [out = 200, in = 180] (.75,.6);
        \draw[blue] (.75,.6) to [out = 200, in = 180] (.75,1);
        \draw[blue] (.75,1) to [out = 140, in = 180] (.85,1); 
        \draw[blue,->] (.85,1) to [out = 180, in = 270] (-1,4); 
        
        \draw[red, very thick, dotted] (1.6,-.5) to [out = 20, in = 0] (.75,.4); 
        \draw[red, very thick, dotted] (.75,.4) to [out = -20, in = 0] (.75,.7);
        \draw[red, very thick, dotted] (.75,.7) to [out = -20, in = 0] (.75,1.35); 
        \draw[red, very thick, dotted] (.75,1.35) to [out = -20, in = 0] (.75,1.75); 
        \draw[red, very thick, dotted] (.75,1.75) to [out = 20, in = 270] (1.35,3.25); 
        \draw[red,->, very thick, dotted] (1.35,3.25) to [out = 180, in = 270] (-.5,4); 
    \end{tikzpicture}
\end{center}
the paths acting on \(\sigma^A_k T\), \(k \ne 0\), that miss row \([b](h_b)\) with \(R^P_2\) leaving \([b]\) in an even row below \(Z\),
\begin{equation*}
    \left[ P_0 \right]_5 = \{Q_k \in \left[ P_0 \right]  :  k \ne 0, Q([b](i,j)) = \sigma^A_k a_0 \text{ for some } 1 \le i < i_z, i \text{ even }, 1 \le k \le w_b\},
\end{equation*}
\begin{center}
    \begin{tikzpicture}
        \draw (0,0) rectangle (2,3);
        \draw (0,-.25) -- (0,0);
        \draw (1.5,-.25) -- (1.5,0);
        
        \draw (0,3) rectangle (3,4);
        \draw (0,4) -- (0,4.25);
        \draw (3,4) -- (3.5,4) -- (3.5,4.25);
            
        \draw (0,2.5) rectangle (.5,3);
        \node at (.25,2.75) {\small \(A_k\)};
        
        \draw (0,3) rectangle (3,3.5);
        \draw (1.25,3) -- (1.25,3.5);
        \draw (1.75,3) rectangle (1.75,3.5);
        \node at (1.5,3.25) {\small \(A_0\)};
        \node at (1.5,3.7) {\small \(k\)};
        
        \draw (0,1.75) rectangle (2,2.25);
        \draw (.75,1.75) -- (.75,2.25);
        \draw (1.25,1.75) -- (1.25,2.25);
        \node at (1,2) {\small \(Z\)};
        
        \draw (0,.75) rectangle (2,1.25);
        \node at (2.7,1) {\small \(\leftarrow\) row \(i\)};
        
        \node at (.5,-.5) {\(u\)};
        \node at (1.5,-.5) {\(v\)};
        
        \draw[blue] (.4,-.5) to [out = 200, in = 180] (.75,.2); 
        \draw[blue] (.75,.2) to [out = 200, in = 180] (.75,.6);
        \draw[blue] (.75,.6) to [out = 200, in = 180] (.75,1.3);
        \draw[blue] (.75,1.3) to [out = 140, in = 180] (.75,1.5); 
        \draw[blue] (.75,1.5) to [out = 140, in = 180] (.75,1.7); 
        \draw[blue] (.75,1.7) to [out = 140, in = 180] (.85,2); 
        \draw[blue,->] (.85,2) to [out = 180, in = 270] (-1,4); 
        
        \draw[red, very thick, dotted] (1.6,-.5) to [out = 20, in = 0] (.75,.4); 
        \draw[red, very thick, dotted] (.75,.4) to [out = -20, in = 0] (.75,.7);
        \draw[red, very thick, dotted] (.75,.7) to [out = 20, in = 0] (.75,1); 
        \draw[red, very thick, dotted] (.75,1) to [out = 20, in = 270] (1.35,3.25); 
        \draw[red,->, very thick, dotted] (1.35,3.25) to [out = 180, in = 270] (-.5,4); 
    \end{tikzpicture}
\end{center}
the paths acting on \(\sigma^A_k T\), \(k \ne 0\), that miss row \([b](h_b)\) with \(R^P_2\) leaving \([b]\) in an odd row below \(Z\),
\begin{equation*}
    \left[ P_0 \right]_6 = \{Q_k \in \left[ P_0 \right]  :  k \ne 0, Q([b](j,l)) = \sigma^A_k a_0 \text{ for some } 1 \le i < i_z, i \text{ odd }, 1 \le k \le w_b\},
\end{equation*}
\begin{center}
    \begin{tikzpicture}
        \draw (0,0) rectangle (2,3);
        \draw (0,-.25) -- (0,0);
        \draw (1.5,-.25) -- (1.5,0);
    
        \draw (0,3) rectangle (3,4);
        \draw (0,4) -- (0,4.25);
        \draw (3,4) -- (3.5,4) -- (3.5,4.25);
        
        \draw (0,2.5) rectangle (.5,3);
        \node at (.25,2.75) {\small \(A_k\)};
        
        \draw (0,3) rectangle (3,3.5);
        \draw (1.25,3) -- (1.25,3.5);
        \draw (1.75,3) rectangle (1.75,3.5);
        \node at (1.5,3.25) {\small \(A_0\)};
        \node at (1.5,3.7) {\small \(k\)};
        
        \draw (0,1.75) rectangle (2,2.25);
        \draw (.75,1.75) -- (.75,2.25);
        \draw (1.25,1.75) -- (1.25,2.25);
        \node at (1,2) {\small \(Z\)};
        
        \draw (0,.75) rectangle (2,1.25);
        \node at (2.7,1) {\small \(\leftarrow\) row \(i\)};
        
        \node at (.5,-.5) {\(u\)};
        \node at (1.5,-.5) {\(v\)};
        
        \draw[blue] (.4,-.5) to [out = 200, in = 180] (.75,.2); 
        \draw[blue] (.75,.2) to [out = 200, in = 180] (.75,.6);
        \draw[blue] (.75,.6) to [out = 120, in = 220] (.75,1); 
        \draw[blue] (.75,1) to [out = 120, in = 180] (1.35,3.25); 
        \draw[blue,->] (1.35,3.25) to [out = 180, in = 270] (-.5,4); 
        
        \draw[red, very thick, dotted] (1.6,-.5) to [out = 20, in = 0] (.75,.4); 
        \draw[red, very thick, dotted] (.75,.4) to [out = -20, in = 0] (.75,.7);
        \draw[red, very thick, dotted] (.75,.7) to [out = -20, in = 0] (.75,1.3); 
        \draw[red, very thick, dotted] (.75,1.3) to [out = -20, in = 0] (.75,1.5);
        \draw[red, very thick, dotted] (.75,1.5) to [out = -20, in = 0] (.75,1.7);
        \draw[red, very thick, dotted] (.75,1.7) to [out = 200, in = 180] (.85,2);
        \draw[red,->, very thick, dotted] (.85,2) to [out = 180, in = 270] (-1,4); 
    \end{tikzpicture}
\end{center}
and the paths acting on \(\sigma^A_k T\), \(k \ne 0\), with \(R^P_2\) missing \([b]\), 
\begin{equation*}
    \left[ P_0 \right]_7 = \{Q_k \in \left[ P_0 \right]  :  k \ne 0, R^P_2 \cap [b] = \emptyset\}.
\end{equation*}
\begin{center}
    \begin{tikzpicture}
        \draw (0,0) rectangle (2,3);
        \draw (0,-.25) -- (0,0);
        \draw (1.5,-.25) -- (1.5,0);
        
        \draw (0,3) rectangle (3,4);
        \draw (0,4) -- (0,4.25);
        \draw (3,4) -- (3.5,4) -- (3.5,4.25);
        
        \draw (0,2.5) rectangle (.5,3);
        \node at (.25,2.75) {\small \(A_k\)};
        
        \draw (0,3) rectangle (3,3.5);
        \draw (1.25,3) -- (1.25,3.5);
        \draw (1.75,3) rectangle (1.75,3.5);
        \node at (1.5,3.25) {\small \(A_0\)};
        \node at (1.5,3.7) {\small \(k\)};
        
        \draw (0,1.25) rectangle (2,1.75);
        \draw (.75,1.25) -- (.75,1.75);
        \draw (1.25,1.25) -- (1.25,1.75);
        \node at (1,1.5) {\small \(Z\)};
        
        \node at (.5,-.5) {\(u\)};
        \node at (1.5,-.5) {\(v\)};
        
        \draw[blue] (.4,-.5) to [out = 160, in = 180] (1.35,3.25); 
        \draw[blue,->] (1.35,3.25) to [out = 180, in = 270] (-.5,4); 

        \draw[red, very thick, dotted] (1.35,-.5) to [out = 160, in = 270] (.75,.25); 
        \draw[red, very thick, dotted] (.75,.2) to [out = 200, in = 180] (.75,.4);
        \draw[red, very thick, dotted] (.75,.4) to [out = 200, in = 180] (.75,.6);
        \draw[red, very thick, dotted] (.75,.6) to [out = 200, in = 180] (.75,.8);
        \draw[red, very thick, dotted] (.75,.8) to [out = 200, in = 180] (.75,1);
        \draw[red, very thick, dotted] (.75,1) to [out = 200, in = 180] (.75,1.2);
        \draw[red, very thick, dotted] (.75,1.2) to [out = 140, in = 180] (.85,1.5); 
        \draw[red, very thick, dotted,->] (.85,1.5) to [out = 180, in = 270] (-1,4); 
    \end{tikzpicture}
\end{center}

Let \(T^\prime \in \mF_{\lambda \setminus X}\) be the unique tableau with \(T^\prime = T_P\) on \((\lambda \setminus X) \setminus [b] \cup [b + 1]\) and \(T^\prime = T\) on \([b] \cup [b + 1]\) except \(T^\prime_{z} = u\) and \(T^\prime_{a_0} = v\).
\begin{center}
\begin{tikzpicture}
    \node at (-1,2) {\(T^\prime = \)};
    
    \draw (0,0) rectangle (2,3);
    \draw (0,-.25) -- (0,0);
    \draw (1.5,-.25) -- (1.5,0);
    
    \draw (0,3) rectangle (3,4);
    \draw (0,4) -- (0,4.25);
    \draw (3,4) -- (3.5,4) -- (3.5,4.25);
    
    \draw (0,2.5) rectangle (.5,3);
    \node at (.25,2.75) {\small \(u\)};
    
    \draw (0,3) rectangle (.5,3.5);
    \node at (.25,3.25) {\small \(A_1\)};
    \draw (.5,3) rectangle (2.5,3.5);
    \node at (1.5,3.25) {\small \(\cdots\)};
    \draw (2.5,3) rectangle (3,3.5);
    \node at (2.75,3.25) {\small \(A_w\)};
    
    \draw (0,1.25) rectangle (2,1.75);
    \draw (.75,1.25) -- (.75,1.75);
    \draw (1.25,1.25) -- (1.25,1.75);
    \node at (1,1.5) {\small \(v\)};
    
    \node at (-.5,4) { \ };
    \node at (-1,4) { \ };
    \node at (.5,-.5) { \ };
    \node at (1.5,-.5) { \ };
\end{tikzpicture}
\end{center}

Then by Corollary \ref{Generating Garnir Relations and Tools for Collapsing Sums subsection: Calculation Lemma for 2-paths} and applications of \(G_A\) we have, mod \(F_2 \otimes \mR_{\lambda \setminus X,n}\),

\begin{align*}
    \sum_{Q_k \in \left[ P_0 \right]_1} \frac{(-1)^{h^Q_b + h^Q_{b + 1}}}{H^Q_b H^Q_{b + 1}} Q_k %
        = & \dfrac{(-1)^{h_b + h_b - 2 + i_z - 2}}{H(b)(H(b) - 1)} ~ \begin{ytableau} A_0 \\ Z \end{ytableau} \otimes T^\prime \\
        = & \dfrac{-H(b + 1) + 1}{H(b)(H(b) - 1)(H(b + 1) - 1)} ~ \begin{ytableau} A_0 \\ Z \end{ytableau} \otimes T^\prime,
\end{align*}
\begin{align*}
    \sum_{Q_k \in \left[ P_0 \right]_2} \frac{(-1)^{h^Q_b + h^Q_{b + 1}}}{H^Q_b H^Q_{b + 1}} Q_k %
        = & \dfrac{(-1)^{h_b + 1 + h_b - 2 + i_z - 2} (w_{b + 1})}{H(b)(H(b) - 1)(H(b + 1) - 1)} ~ %
                    \begin{ytableau} A_0 \\ Z \end{ytableau} \otimes T^\prime\\
        = & \dfrac{w_{b + 1}}{H(b)(H(b) - 1)(H(b + 1) - 1)} ~ \begin{ytableau} A_0 \\ Z \end{ytableau} \otimes T^\prime
\end{align*}
\begin{align*}
    \sum_{Q_k \in \left[ P_0 \right]_3} \frac{(-1)^{h^Q_b + h^Q_{b + 1}}}{H^Q_b H^Q_{b + 1}} Q_k %
        = & \dfrac{(-1)^{h_b + 1 + 1 + h_b - 2 + i_z - 2}(w_b-1)}{H(b)(H(b) - 1)(H(b + 1) - 1)} ~ %
                    \begin{ytableau} A_0 \\ Z \end{ytableau} \otimes T^\prime \\
        = & \dfrac{-w_b + 1}{H(b)(H(b) - 1)(H(b + 1) - 1)} ~ \begin{ytableau} A_0 \\ Z \end{ytableau} \otimes T^\prime
\end{align*}
\begin{align*}
    \sum_{Q_k \in \left[ P_0 \right]_4} \frac{(-1)^{h^Q_b + h^Q_{b + 1}}}{H^Q_b H^Q_{b + 1}} Q_k %
        = & \sum_{i=i_z + 1}^{h_b - 1} \dfrac{(-1)^{i + 1 + 1 + i - 2 + i_z - 2 + 1}}{H(b)(H(b) - 1)(H(b + 1) - 1)} ~ %
                    \begin{ytableau} A_0 \\ Z \end{ytableau} \otimes T^\prime \\
        = & \dfrac{h_b - 1 - i_z}{H(b)(H(b) - 1)(H(b + 1) - 1)} ~ \begin{ytableau} A_0 \\ Z \end{ytableau} \otimes T^\prime
\end{align*}
\begin{align*}
    \sum_{Q_k \in \left[ P_0 \right]_5} \frac{(-1)^{h^Q_b + h^Q_{b + 1}}}{H^Q_b H^Q_{b + 1}} Q_k %
        = & \sum_{\substack{1 \le i < i_z,\\ i \text{ even}}} \dfrac{(-1)^{i_z + 1 + 1 + i_z - 2 + i - 2 + 1}}%
        {H(b)(H(b) - 1)(H(b + 1) - 1)} ~ \begin{ytableau} A_0 \\ Z \end{ytableau} \otimes T^\prime \\
        = & \sum_{\substack{1 \le i < i_z,\\ i \text{ even}}} \dfrac{(-1)^{i+1}}%
        {H(b)(H(b) - 1)(H(b + 1) - 1)} ~ \begin{ytableau} A_0 \\ Z \end{ytableau} \otimes T^\prime \\
\end{align*}
\begin{align*}
    \sum_{Q_k \in \left[ P_0 \right]_6} \frac{(-1)^{h^Q_b + h^Q_{b + 1}}}{H^Q_b H^Q_{b + 1}} Q_k %
        = & \sum_{\substack{1 \le i < i_z,\\ i \text{ odd}}} \dfrac{(-1)^{i_z + 1 + 1 + i_z - 2 + i - 2 + 1}}%
        {H(b)(H(b) - 1)(H(b + 1) - 1)} ~ \begin{ytableau} Z \\ A_0 \end{ytableau} \otimes T^\prime \\
        = & \sum_{\substack{1 \le i < i_z,\\ i \text{ odd}}} \dfrac{(-1)^{i+2}}%
        {H(b)(H(b) - 1)(H(b + 1) - 1)} ~ \begin{ytableau} A_0 \\ Z \end{ytableau} \otimes T^\prime \\
\end{align*}
\begin{align*}
    \sum_{Q_k \in \left[ P_0 \right]_7} \frac{(-1)^{h^Q_b + h^Q_{b + 1}}}{H^Q_b H^Q_{b + 1}} Q_k %
        = & \dfrac{(-1)^{i_z + 1 + 1 + i_z - 1}}{(H(b) - 1)(H(b + 1) - 1)} ~ %
                    \begin{ytableau} Z \\ A_0 \end{ytableau} \otimes T^\prime\\
        = & \dfrac{H(b)}{H(b)(H(b) - 1)(H(b + 1) - 1)} ~ \begin{ytableau} A_0 \\ Z \end{ytableau} \otimes T^\prime
\end{align*}

Then as \(H(b + 1) = H(b) + w_{b + 1} - w_b + h_b\) and
\begin{align*}
    \sum_{\substack{1 \le i < i_z,\\ i \text{ even}}} & \dfrac{(-1)^i} {H(b)(H(b) - 1)(H(b + 1) - 1)} ~ \begin{ytableau} A_0 \\ Z \end{ytableau} \otimes T^\prime\\
    & + \sum_{\substack{1 \le i < i_z,\\ i \text{ odd}}} \dfrac{(-1)^{i+1}} {H(b)(H(b) - 1)(H(b + 1) - 1)} ~ \begin{ytableau} A_0 \\ Z \end{ytableau} \otimes T^\prime \\
    & = \dfrac{i_z - 1}{H(b)(H(b) - 1)(H(b + 1) - 1)} ~ \begin{ytableau} A_0 \\ Z \end{ytableau} \otimes T^\prime
\end{align*}
we get, mod \(F_2 \otimes \mR_{\lambda \setminus X,n}\)
\begin{align*}
    & \sum_{Q_k \in \left[P_0\right]} \frac{(-1)^{h^Q_b + h^Q_{b + 1}}}{H^Q_b H^Q_{b + 1}} Q_k %
             =  \sum_{i = 1,\ldots,7} ~ \sum_{Q_k \in \left[P_0\right]_i} \frac{(-1)^{h^Q_b + h^Q_{b + 1}}}{H^Q_b H^Q_{b + 1}} Q_k\\
    & = \dfrac{-H(b + 1) + 1 + w_{b + 1} - w_b + 1 + h_b - 1 - i_z - + i_z - 1 + H(b)}{H(b) \left(H(b) - 1\right) \left(H(b + 1) - 1\right)} %
            ~ \begin{ytableau} A_0 \\ Z \end{ytableau} \otimes T^\prime\\
    & = 0
\end{align*}
\end{proof}

\begin{case*}[\ref{Showing the Pieri Inclusion Removing Many Boxes is a GL(V)-map subsection: Two Box Removal Preserves Garnirs for Hooks Contained in Two Blocks equation: T6}]
In this case we show that the sum over all paths that move \(A_i\) and a box other than \(a_0\) in row \([b](h_b)\) above \([b + 1]\) is in \(F_2 \otimes \mR_{\lambda \setminus X,n}\).
Recall that
\[
    \sT_6 = \bigsqcup_{0 \le i \le w_{b + 1}, ~ 2 \le j \le w_b} \sT_6^{i,j},
\]
where
\[
    \sT_6^{i,j} = \{ P_k : P(\sigma^A_k A_i) > [b + 1], P([b](h_b,j))> [b + 1]\}.        
\]
It is enough to show that
\[
    \sum_{P_k \in \sT_6^{0,2}} \frac{(-1)^{P}}{H(P)} P_k \in F_2 \otimes \mR_{\lambda \setminus X,n},
\]
with the other cases being similar. 
\end{case*}

\begin{proof}
For the rest of Case (\ref{Showing the Pieri Inclusion Removing Many Boxes is a GL(V)-map subsection: Two Box Removal Preserves Garnirs for Hooks Contained in Two Blocks equation: T6}), let \(z = [b](h_b,2)\) and \(Z = T_z\), and observe that \(\sT_6^{0,2}\) is the union of the following disjoint sets.

\begin{align*}
    \sT_6^{0,2,1} & = \{P_k \in \sT_6^{0,2}  :  P(\sigma^A_k a_0), P(z) \not \in Y\},\\[5pt]
    \sT_6^{0,2,2} & = \{P_k \in \sT_6^{0,2}  :  P(\sigma^A_k a_0) \in Y, P(z) \not \in Y\},\\[5pt]
    \sT_6^{0,2,3} & = \{P_k \in \sT_6^{0,2}  :  P(z) \in Y, P(\sigma^A_k a_0) \not \in Y\}, \text{ and }\\[5pt]
    \sT_6^{0,2,4} &= \{P_k \in \sT_6^{0,2}  :  P(\sigma^A_k a_0), P(z) \in Y\}.
\end{align*}

So it is enough to show that for \(1 \le i \le 4\),
\[
    \sum_{P_k \in \sT_6^{0,2,i}} \frac{(-1)^{P}}{H(P)} P_k \in F_2 \otimes \mR_{\lambda \setminus X,n}.
\]

Now define the relation \(\sim^i\) on \(\sT_6^{0,2,i}\), for \(1 \le i \le 4\), as follows.
\begin{itemize}
\item Define \(\sim^1\) on \(\sT_6^{0,2,1}\) by 
	\begin{align*}
        P_k \sim^1 Q_j %
        	\iff & Q \text{ is a } ([b](1), [b+1](1)) \text{-path extension of } P,\\
                & Q(\sigma^A_j A_0) = P(\sigma^A_k A_0), \text{ and } Q(Z) = P(Z), \text{ and }\\
                & Q^{-1}(\sigma^A_j A_0) \text{ and } P^{-1}(\sigma^A_k A_0) \text{ are in the same row.}
    \end{align*}
\item Define \(\sim^2\) on \(\sT_6^{0,2,2}\) by 
	\begin{align*}
    P_k \sim^2 Q_j %
        \iff & Q \text{ is a } ([b](1), [b+1](1)) \text{-path extension of } P,\\
            & Q(Z) = P(Z), \text{ and }\\
            & Q^{-1}(\sigma^A_j A_0) \text{ and } P^{-1}(\sigma^A_k A_0) \text{ are in the same row.}
    \end{align*}
\item Define \(\sim^3\) on \(\sT_6^{0,2,3}\) by 
	\begin{align*}
    P_k \sim^3 Q_j %
    	\iff & Q \text{ is a } ([b](1), [b+1](1)) \text{-path extension of } P,\\
            & Q(\sigma^A_j A_0) = P(\sigma^A_k A_0), \text{ and }\\
            & Q^{-1}(\sigma^A_j A_0) \text{ and } P^{-1}(\sigma^A_k A_0) \text{ are in the same row.}
    \end{align*}
\item Define \(\sim^4\) on \(\sT_6^{0,2,4}\) by 
	\begin{align*}
    P_k \sim^4 Q_j %
    	\iff & Q \text{ is a } ([b](1), [b+1](1)) \text{-path extension of } P, \text{ and }\\
            & Q^{-1}(\sigma^A_j A_0) \text{ and } P^{-1}(\sigma^A_k A_0) \text{ are in the same row.}
    \end{align*}
\end{itemize}

It is clear that for \(1 \le i \le 4\), \(\sim^i\) an equivalence relation on \(\sT_6^{0,2,i}\), so that
\[
    \sum_{P_k \in \sT_6^{0,2,i}} \frac{(-1)^{P}}{H(P)} P_k = \sum_{\left[P_k\right] \in \sT_6^{0,2,i} / \sim^i} %
    \sum_{Q_k \in \left[P_k\right]} \frac{(-1)^{Q}}{H(Q)} Q_k.
\]
Thus it is enough to show that for \(1 \le i \le 4\),
\[
    \sum_{\left[P_k\right] \in \sT_6^{0,2,i} / \sim^i} \sum_{Q_k \in \left[P_k\right]} %
    \frac{(-1)^{Q}}{H(Q)} Q_k \in F_2 \otimes \mR_{\lambda \setminus X,n}.
\]
We will show the case \(i = 1\), with the rest being similar. For the rest of Case (\ref{Showing the Pieri Inclusion Removing Many Boxes is a GL(V)-map subsection: Two Box Removal Preserves Garnirs for Hooks Contained in Two Blocks equation: T6}), let \(\sT := \sT_6^{0,2,1}\).

For \(l = 1, 2\), let \(\sT_{x_l}\) be the set of all \(Q_k\) in \(\sT\) such that the orbit of \(x_l\) intersects the first row in \([b]\),
\[
    \sT_{x_l} = \{P_k \in \sT  :  R^P_l \cap [b](1) \ne \emptyset\}.
\]
Note that as \(a_0 = [b](h_b,1)\) is in the same row as \(z = [b](h_b,2)\), it must be that \(1 \le k \le w_b\) for all \(P_k \in \sT\). 
Pick \(P_1 \in \sT\) with \([b](i,1) \in R^P\) for all \(i = 1, \ldots, h_b - 1\), and \(P^{-1}(\sigma_1 A_0) \in [b](i)\) with \(i\) odd, and let \([b_u](i_u,j_u) = P^{-1}([b](1,1))\) and \([b_v](i_v,j_v) = P^{-1}([b](2,1))\) with \(u = T_{[b_u](i_u,j_u)}\) and \(v = T_{[b_v](i_v,j_v)}\).

\begin{center}
\begin{tikzpicture}
    \node at (-1,2) {\(P_1 = \)};
    
    \draw (0,0) rectangle (2,3);
    \draw (0,-.25) -- (0,0);
    \draw (1.5,-.25) -- (1.5,0);
    
    \draw (0,3) rectangle (3,4);
    \draw (0,4) -- (0,4.25);
    \draw (3,4) -- (3.5,4) -- (3.5,4.25);
    
    \draw (0,2.5) rectangle (1,3);
    \draw (.5,0) -- (.5,3);
    \node at (.25,2.75) {\small \(A_k\)};
    \node at (.75,2.75) {\small \(Z\)};
    
    \draw (0,3) rectangle (.5,3.5);
    \node at (.25,3.25) {\small \(A_0\)};
    
    \draw (0,1.25) rectangle (2,1.75);
    \node at (2.75,1.5) {\small \(\leftarrow\) row \(i\)};
    
    \node at (.5,-.5) {\(u\)};
    \node at (1.5,-.5) {\(v\)};
    
    \draw[blue] (.4,-.5) to [out = 200, in = 180] (.2,.2); 
    \draw[blue] (.2,.2) to [out = 200, in = 180] (.2,.6);
    \draw[blue] (.2,.6) to [out = 200, in = 180] (.2,1);
    \draw[blue] (.2,1) to [out = 140, in = 180] (.2,1.5); 
    \draw[blue] (.2,1.5) to [out = 140, in = 180] (.1,3.25); 
    \draw[blue,->] (.1,3.25) to [out = 180, in = 270] (-.5,4); 
    
    \draw[red, very thick, dotted] (1.6,-.5) to [out = 20, in = 0] (.2,.4); 
    \draw[red, very thick, dotted] (.2,.4) to [out = -20, in = 0] (.2,.8);
    \draw[red, very thick, dotted] (.2,.8) to [out = -20, in = 0] (.2,1.2);
    \draw[red, very thick, dotted] (.2,1.2) to [out = 20, in = -20] (.2,1.8); 
    \draw[red, very thick, dotted] (.2,1.8) to [out = 200, in = 180] (.2,2);
    \draw[red, very thick, dotted] (.2,2) to [out = 200, in = 180] (.2,2.2);
    \draw[red, very thick, dotted] (.2,2.2) to [out = 200, in = 180] (.2,2.4);
    \draw[red, very thick, dotted] (.2,2.4) to [out = 20, in = 270] (.6,2.75); 
    \draw[red,->, very thick, dotted] (.6,2.75) to [out = 180, in = 270] (-1,4); 
\end{tikzpicture}
\end{center}

It is then enough to show that
\[
    \sum_{Q_k \in \left[P_1\right]} \frac{(-1)^{Q}}{H(Q)} Q_k \in F_2 \otimes \mR_{\lambda \setminus X,n},
\]
with the other cases being similar. In fact, as \((-1)^P = (-1)^Q\) and \(H(Q) = H(P)\) for all \(Q_k \in \left[P_1\right]\), it is enough to show that
\[
    \sum_{Q_k \in \left[P_1\right]} Q_k \in F_2 \otimes \mR_{\lambda \setminus X,n}.
\]

Without loss of generality, assume \(P_1 \in \sT_{x_1}\) and let \(\left[ P_1 \right]_{x_1} = \left[ P_1 \right] \cap \sT_{x_1}\) and \(\left[ P_1 \right]_{x_2} = \left[ P_1 \right] \cap \sT_{x_2}\), so that
\[
    \left[ P_1 \right] = \left[ P_1 \right]_{x_1} \bigsqcup \left[ P_1 \right]_{x_2}.
\]
See Figure \ref{Case Showing the Pieri Inclusion Removing Many Boxes is a GL(V)-map subsection: Two Box Removal Preserves Garnirs for Hooks Contained in Two Blocks equation: T6 figure: [P]1}.

\begin{figure}[h]
\begin{subfigure}{0.35\textwidth}
\begin{tikzpicture}
    \draw (0,0) rectangle (2,3);
    \draw (0,-.25) -- (0,0);
    \draw (1.5,-.25) -- (1.5,0);
    
    \draw (0,3) rectangle (3,4);
    \draw (0,4) -- (0,4.25);
    \draw (3,4) -- (3.5,4) -- (3.5,4.25);
    
    \draw (0,2.5) rectangle (1,3);
    \draw (.5,2.5) -- (.5,3);
    \node at (.25,2.75) {\small \(A_k\)};
    \node at (.75,2.75) {\small \(Z\)};
    
    \draw (0,3) rectangle (3,3.5);
    \draw (1.25,3) rectangle (1.75,3.5);
    \node at (1.5,3.25) {\small \(A_0\)};
    \node at (1.5,3.7) {\small \(k\)};
    
    \draw (0,1.25) rectangle (2,1.75);
    \node at (2.75,1.5) {\small \(\leftarrow\) row \(i\)};
    
    \node at (.5,-.5) {\(u\)};
    \node at (1.5,-.5) {\(v\)};
    
    \draw[blue] (.4,-.5) to [out = 200, in = 180] (.75,.2); 
    \draw[blue] (.75,.2) to [out = 200, in = 180] (.75,.6);
    \draw[blue] (.75,.6) to [out = 200, in = 180] (.75,1);
    \draw[blue] (.75,1) to [out = 140, in = 180] (.75,1.5); 
    \draw[blue] (.75,1.5) to [out = 140, in = 180] (1.35,3.25); 
    \draw[blue,->] (1.35,3.25) to [out = 180, in = 270] (-.5,4); 
    
    \draw[red, very thick, dotted] (1.6,-.5) to [out = 20, in = 0] (.75,.4); 
    \draw[red, very thick, dotted] (.75,.4) to [out = -20, in = 0] (.75,.8);
    \draw[red, very thick, dotted] (.75,.8) to [out = -20, in = 0] (.75,1.2);
    \draw[red, very thick, dotted] (.75,1.2) to [out = 20, in = -20] (.75,1.8); 
    \draw[red, very thick, dotted] (.75,1.8) to [out = 200, in = 180] (.75,2);
    \draw[red, very thick, dotted] (.75,2) to [out = 200, in = 180] (.75,2.2);
    \draw[red, very thick, dotted] (.75,2.2) to [out = 200, in = 180] (.75,2.4);
    \draw[red, very thick, dotted] (.75,2.4) to [out = 200, in = 180] (.6,2.75); 
    \draw[red,->, very thick, dotted] (.6,2.75) to [out = 180, in = 270] (-1,4); 
\end{tikzpicture}
\caption{\(Q_k \in \left[P_1\right]_{x_1}\)}
\end{subfigure}
\hspace{.5in}
\begin{subfigure}{0.35\textwidth}
\begin{tikzpicture}
    \draw (0,0) rectangle (2,3);
    \draw (0,-.25) -- (0,0);
    \draw (1.5,-.25) -- (1.5,0);
    
    \draw (0,3) rectangle (3,4);
    \draw (0,4) -- (0,4.25);
    \draw (3,4) -- (3.5,4) -- (3.5,4.25);
    
    \draw (0,2.5) rectangle (1,3);
    \draw (.5,2.5) -- (.5,3);
    \node at (.25,2.75) {\small \(A_k\)};
    \node at (.75,2.75) {\small \(Z\)};
    
    \draw (0,3) rectangle (3,3.5);
    \draw (1.25,3) rectangle (1.75,3.5);
    \node at (1.5,3.25) {\small \(A_0\)};
    \node at (1.5,3.7) {\small \(k\)};
    
    \draw (0,1.25) rectangle (2,1.75);
    \node at (2.75,1.5) {\small \(\leftarrow\) row \(i\)};
    
    \node at (.5,-.5) {\(u\)};
    \node at (1.5,-.5) {\(v\)};
    
    \draw[blue] (.4,-.5) to [out = 200, in = 180] (.75,.4); 
    \draw[blue] (.75,.4) to [out = 200, in = 180] (.75,.8);
    \draw[blue] (.75,.8) to [out = 200, in = 180] (.75,1.2);
    \draw[blue] (.75,1.2) to [out = 200, in = 180] (.75,1.8); 
    \draw[blue] (.75,1.8) to [out = 200, in = 180] (.75,2);
    \draw[blue] (.75,2) to [out = 200, in = 180] (.75,2.2);
    \draw[blue] (.75,2.2) to [out = 200, in = 180] (.75,2.4);
    \draw[blue] (.75,2.4) to [out = 200, in = 180] (.6,2.75); 
    \draw[blue, ->] (.6,2.75) to [out = 180, in = 270] (-1,4); 
    
    \draw[red, very thick, dotted] (1.6,-.5) to [out = 20, in = 0] (.75,.2); 
    \draw[red, very thick, dotted] (.75,.2) to [out = -20, in = 0] (.75,.6);
    \draw[red, very thick, dotted] (.75,.6) to [out = -20, in = 0] (.75,1);
    \draw[red, very thick, dotted] (.75,1) to [out = -20, in = 0] (.75,1.5); 
    \draw[red, very thick, dotted] (.75,1.5) to [out = 20, in = 270] (1.35,3.25); 
    \draw[red, very thick, dotted,->] (1.35,3.25) to [out = 180, in = 270] (-.5,4); 
\end{tikzpicture}
\caption{\(Q_k \in \left[P_1\right]_{x_2}\)}
\end{subfigure}
\caption{}
\label{Case Showing the Pieri Inclusion Removing Many Boxes is a GL(V)-map subsection: Two Box Removal Preserves Garnirs for Hooks Contained in Two Blocks equation: T6 figure: [P]1}
\end{figure}

Let \(T^\prime \in \mF_{\lambda \setminus X}\) be the unique tableau with \(T^\prime = T_P\) on \((\lambda \setminus X) \setminus \left([b] \cup [b + 1]\right)\) and \(T^\prime = T\) on \([b] \cup [b + 1]\) except \(T^\prime_{z} = v\) and \(T^\prime_{a_0} = u\).

\begin{center}
\begin{tikzpicture}
    \node at (-1,2) {\(T^\prime = \)};
    
    \draw (0,0) rectangle (2,3);
    \draw (0,-.25) -- (0,0);
    \draw (1.5,-.25) -- (1.5,0);
    
    \draw (0,3) rectangle (3,4);
    \draw (0,4) -- (0,4.25);
    \draw (3,4) -- (3.5,4) -- (3.5,4.25);
    
    \draw (0,2.5) rectangle (1,3);
    \draw (.5,2.5) -- (.5,3);
    \node at (.25,2.75) {\small \(u\)};
    \node at (.75,2.75) {\small \(v\)};
    
    \draw (0,3) rectangle (3,3.5);
    \draw (.5,3) -- (.5,3.5);
    \draw (2.5,3) -- (2.5,3.5);
    \node at (.25,3.25) {\small \(A_1\)};
    \node at (1.5,3.25) {\small \(\cdots\)};
    \node at (2.75,3.25) {\small \(A_w\)};
    
    \node at (-.5,4) { \ };
    \node at (-1,4) { \ };
    \node at (.5,-.5) { \ };
    \node at (1.5,-.5) { \ };
\end{tikzpicture}
\end{center}

As in the calculations for the proof of Case \ref{Showing the Pieri Inclusion Removing Many Boxes is a GL(V)-map subsection: Two Box Removal Preserves Garnirs for Hooks Contained in a Single Block equation: T6}, by the proof of Lemma \ref{Generating Garnir Relations and Tools for Collapsing Sums subsection: Calculation Lemma for 1-paths}, the result of Corollary \ref{Generating Garnir Relations and Tools for Collapsing Sums subsection: Calculation Lemma for 2-paths} still holds when moving \(u\) and \(v\) to boxes in the same row, which we have here after applying \(G_A\). 
This gives, mod \(F_2 \otimes \mR_{\lambda \setminus X, n}\),
\begin{align*}
    \sum_{Q_k \in \left[P_1\right]} Q_k %
        & = \sum_{Q_k \in \left[P_1\right]_{x_1}} Q_k + \sum_{Q_k \in \left[P_1\right]_{x_2}} Q_k\\
        & = (-1)^{h_b + 1 + 1 h_b - 2 + i - 2 + 1} ~ \ytableaushort{{\alpha^P_2}, {\alpha^P_1}} \otimes %
            T^\prime 
            + (-1)^{h_b + 1 + 1 h_b - 2 + i - 2 + 1} ~ \begin{ytableau} \alpha^P_1 \\ %
            \alpha^P_2 \end{ytableau} \otimes T^\prime\\
        & = 0 \otimes T^\prime.
\end{align*}
\end{proof}

Thus Equation \ref{Showing the Pieri Inclusion Removing Many Boxes is a GL(V)-map subsection: The Theorem equation: Phi of Garnir m Boxes} holds for all hooks \(A \subset [b] \cup [b + 1]\), and so Theorem \ref{Showing the Pieri Inclusion Removing Many Boxes is a GL(V)-map subsection: The Theorem} holds for \(m = 2\).

\section{Relating Pieri Inclusion Descriptions} \label{section: Relating Pieri Inclusion Descriptions}

\subsection{} \label{Relating Pieri Inclusion Descriptions subsection: Comparing with Olver's map}

In this section we show that our description of the Pieri inclusion removing one box is the negative of that given by Olver. 
We then show that iterating our description of the Pieri inclusion removing one box is equal to our description of the Pieri inclusion removing many boxes.
Finally, we show that our description of the Pieri inclusion removing many boxes also describes the symmetric case.
 
\begin{theorem*}
For \(\widetilde\Phi_1\) and \(\Phi_1\) as above,
\[
    \Phi_1 = - \widetilde\Phi_1.
\]
\end{theorem*}

\begin{proof}
Let \(T_{\lambda} \in \mT_{\lambda, n}\) and \(T_{\lambda \setminus X} \in \mT_{\lambda \setminus X, n}\) be the diagrams corresponding to highest weight vectors as in Section \ref{Constructing Schur--Weyl Modules subsection: Highest Weights}.
Then, in the image of \(T_{\lambda}\), the coefficient of 
\[
   \ytableaushort{{\alpha}} \otimes T_{\lambda \setminus X}
\]
where
\[
    \alpha = \sum_{i = b_1}^N h_{i}
\]
is readily seen to be \(- w_{b_1}\) in the image of \(\Phi\) and \(w_{b_1}\) in the image of \(\widetilde{\Phi}_1\). By uniqueness of the Pieri inclusion up to scalar multiple (Schur's Lemma), we then have
\[
    \Phi = - \widetilde{\Phi}_1.
\]
\end{proof}

\subsection{} \label{Relating Pieri Inclusion Descriptions subsection: Equating 1 Box Removal and m Box Removal}
Given removal set \(X = \{x_1 = [b_1](1,w_{b_1}), \ldots, x_m = [b_m](i_m,w_{b_m})\} \subset \lambda\), let \(\Phi_1^m\) be the map given by composing the one box removal map where the column of removed boxes is extended by one each time, i.e. \(\Phi_1^1 = \Phi_1\) and 
\[
    \Phi_1^m\left( T \right) = \sum_{P} \frac{\left( -1 \right)^P}{H(P)} P\left( \Phi_{1}^{m - 1}(T) \right)
\]
where the sum is over all \(1\)-paths \(P\) on \(\lambda \setminus \{x_1, \ldots, x_{m - 1}\}\) removing \(x_m\) and
\[
    P \left( Y_Q \otimes T_Q \right) = \substack{Y_P \\ Y_Q} \otimes P\left( T_Q \right)
\]
and \(Y_P\) is the box removed by from \(P\) from \(T_Q\). 

\begin{lemma*}
    \(\Phi_1^m\) is a \(\GL(V)\)-map.
\end{lemma*}

\begin{proof}
    By the previous theorem this follows from the proof in \cite[Corollary 1.8]{sam2011pieri}, where it is shown for the iteration of Olver's map.
\end{proof}

Let 
\[
    \Phi_m : \bS_{\lambda} (V) \to F_m \otimes \bS_{\lambda \setminus X} (V)
\]
be the Pieri inclusion constructed in Section \ref{section: Constructing the Pieri Inclusion for Removing Many Boxes}.

\begin{theorem*} 
For \(\Phi_1^m\) and \(\Phi_m\) as above,
\[
    \Phi_1^m = \Phi_m.
\]
\end{theorem*}

\begin{proof}
Let \(T_{\lambda} \in \mT_{\lambda, n}\) and \(T_{\lambda \setminus X} \in \mT_{\lambda \setminus X, n}\)  be the diagrams corresponding to highest weight vectors as in Section \ref{Constructing Schur--Weyl Modules subsection: Highest Weights}.
Then in the image of \(T_{\lambda}\), the coefficient of 
\[
   \begin{ytableau} \alpha_m \\ \vdots \\ \alpha_1 \end{ytableau} \otimes T_{\lambda \setminus X}
\]
where
\[
    \alpha_1 = \sum_{i = b_1}^N h_{i}
\]
and
\[
    \alpha_k = \alpha_1 - (k - 1)
\]
for \(2 \le k \le m\) is readily seen to be \(\left( -1 \right)^m w_{b_1}^m\) in the image of both \(\Phi_1^m\) and \(\Phi_m\). By uniqueness of the Pieri inclusion up to scalar multiple, we then have
\[
    \Phi_1^m = \Phi_m.
\]
\end{proof}

\subsection{} \label{Relating Pieri Inclusion Descriptions subsection: The Symmetric Case}
Define the map
\[
    \Phi_m^\prime : \bS_{\lambda} (V) \to S^m(V) \otimes \bS_{\lambda \setminus X} (V)
\]
just as we have defined \(\Phi_m\) in Section \ref{section: Constructing the Pieri Inclusion for Removing Many Boxes} except for redefining, for all \(m\)-paths \(P\) on \(\lambda\) removing \(X\),
\[
    Y_P = E_X \ytableaushort{{\alpha_m^P} \cdots {\alpha_1^P}}
\]
which is standard form notation is \(e_{\alpha_1^P} \cdots e_{\alpha_m^P} \in S^m V\).

\begin{theorem*}
    The map
    \[
        \Phi_m^\prime : \bS_{\lambda} (V) \to S^m(V) \otimes \bS_{\lambda \setminus X} (V)
    \]
    is a \(\GL(V)\)-map.
\end{theorem*}
\begin{proof}
    As \(\Phi_m\) is a \(GL(V)\)-map, similar to \cite[Corollary 1.8]{sam2011pieri}, this follows by the results of Sections \ref{section: Showing the Pieri Inclusion Removing One Box is a GL(V)-map} and \ref{section: Showing the Pieri Inclusion Removing Many Boxes is a GL(V)-map} by keeping track of a sign.
\end{proof}


\section{The Image of a Highest Weight Vector and Computational Complexity} \label{section: The Image of a Highest Weight Vector and Computational Complexity}

\subsection{} \label{The Image of a Highest Weight Vector and Computational Complexity subsection: Description in Terms of Phi_1}

In this section we use the Pieri inclusion \(\Phi_1\) constructed in Section \ref{section: Constructing the Pieri Inclusion for Removing One Box} to give an optimal description of the image under a Pieri inclusion removing one box of a highest weight vector.
Given a removal set \(X = \{x_1 = [b_1](1,w_{b_1})\} \subset \lambda\), it is clear by the construction of \(1\)-paths that for all \(1\)-paths on \(\lambda\) removing \(X\), \((T_{\lambda})_P\) is semi-standard.

Define the relation \(\sim\) on the set of all \(1\)-paths on \(\lambda\) removing \(X\) by
\[
    P \sim Q \iff R^Q \text{and } R^P \text{ intersect the same set of rows}.
\]
This clearly defines an equivalence relation. Let
\[
    [P] = \{ Q : Q \sim P\}.
\]
Then for all \(Q \in [P]\) we have \((-1)^Q = (-1)^P\) and \(H(Q) = H(P)\), and, when considering the image of a highest weight vector where each entry in a given row is the same,
\[
    Y_Q \otimes (T_\lambda)_Q = Y_P \otimes (T_\lambda)_P.
\]

For distinct \([P]\) and \([P^\prime]\) we have (by construction) that \(Y_P \otimes (T_\lambda)_P\) and \(Y_{P^\prime} \otimes (T_\lambda)_{P^\prime}\) are linearly independent.
Thus, \(\Phi_1(T_{\lambda})\) can be written as
\[
    \Phi_1(T_{\lambda}) = \sum_{[P_0]} \frac{(-1)^{P_0} \vert [P_0] \vert}{H(P_0)} P_0(T_{\lambda})
\]
where the sum is over all \(1\)-paths \(P_0\) on \(\lambda\) removing \(X\) which only hit boxes in the first column of \(\lambda\).
From the above, the terms in the image of \(\Phi_1(T_{\lambda})\) are linearly independent and do not require straightening, and so this description is optimal.
Recall that such an example was computed in Section \ref{Introduction subsection: Example HW}.
To see the optimal description from this example, take only the first six terms shown.

\subsection{} \label{The Image of a Highest Weight Vector and Computational Complexity subsection: The Size of |P_0|}

For a given such \(1\)-path \(P_0\), we now describe the corresponding term in the image of \(T_{\lambda}\).
Let \(\{r_i\}_{1 \le i \le \vert R^{P_0} \vert}\) be the rows in \(\lambda\) that \(P_0\) hits, so that \(\lambda_i > \lambda_{i + 1}\) and \(r_{\vert R^{P_0} \vert} = [b_1](1)\).
Then
\[
    \left\vert [P_0] \right\vert = \prod_{i = 1}^{\vert P \vert} \lambda_{r_i}
\]
and \((T_{\lambda})_{P_0} \in \bS_{\lambda \setminus X}\) has \(\lambda_1\) ones in the first row, \(\lambda_2\) twos in the first row, etc. except for each row \(r_i\), \(1 \le i \le \vert R^{P_0} \vert\), where the last entry in row \(r_i\) of \((T_{\lambda})_{P_0}\) is 
\[
    \left((T_{\lambda})_{P_0} \right)_{(r_i, \lambda_{r_i})} = r_{i + 1}.
\]

\subsection{} \label{The Image of a Highest Weight Vector and Computational Complexity subsection: HW Macaulay2 Examples}

The second author has implemented this optimal description using Macaulay2, with the output given as a hash table, where one can quickly compute the image of the highest weight for very large examples. 
Figure \ref{The Image of a Highest Weight Vector and Computational Complexity subsection: Macaulay2 Examples figure: 18 row} shows the timed computation for the image of a highest vector, where the partition is given as the first input of the function oneboxremovalHW and the second input of the function is the row (from the top of the tableau) of the box to be removed.

\begin{figure}[H]
    \centering
    \raisebox{-.9\height}{\includegraphics[width=.8\textwidth]{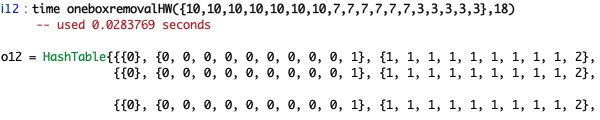}}
    \caption{Computing the image of the highest weight for the inclusion \(\bS_{(10,10,10,10,10,10,10,7,7,7,7,7,7,3,3,3,3,3)} \to \bS_{(1)} \otimes \bS_{(10,10,10,10,10,10,10,7,7,7,7,7,7,3,3,3,3,2)}\).}
    \label{The Image of a Highest Weight Vector and Computational Complexity subsection: Macaulay2 Examples figure: 18 row}
\end{figure}

\subsection{} \label{The Image of a Highest Weight Vector and Computational Complexity subsection: Complexity of Phi_1 vs Phi_O}

We now describe the computational complexity of \(\Phi_1\) and compare this to the computational complexity of the one-box removal Pieri inclusion described by Olver.
For a removal set \(X = \{ x_1 = [b_1](1, w_{b_1})\}\), let
\[
    \widetilde\Phi_1 : \bS_{\lambda} (V) \to V \otimes \bS_{\lambda \setminus X} (V)
\]
be the Pieri inclusion given by Olver (see \cite[\S 1.2]{sam2011pieri} and \cite[\S 4]{sam2009computing}). 

\begin{theorem*}
    Fix a positive integer \(N\) and consider partitions \(\lambda\) that have at most \(N\) blocks.
    Then our algorithm to compute the image of a highest weight vector under a Pieri inclusion \(\Phi_1: \bS_{\lambda}(V) \hookrightarrow V \otimes \bS_{\lambda \setminus X}(V)\) has a worst-case time complexity of \(O(l(\lambda)^N)\).
    On the other hand, the algorithm to compute the image of a highest weight vector under a Pieri inclusion \(\widetilde\Phi_1: \bS_{\lambda}(V) \hookrightarrow V \otimes \bS_{\lambda \setminus X}(V)\) has a worst-case time complexity of \(\Omega(2^{\,l(\lambda)})\).
\end{theorem*}
 
\begin{proof}
Let \(\lambda = (w_1^{h_1}, \ldots, w_N^{h_N})\). 
We first consider the time complexity of the algorithm as given by Olver's construction. 
As in Section \ref{The Image of a Highest Weight Vector and Computational Complexity subsection: Description in Terms of Phi_1}, when considering the image of a highest weight vector we only need to select paths on \(\lambda\) removing \(X\) that act on the first column of \(\lambda\).
From the description of the map \(\widetilde\Phi_1\) removing \(X\) in \cite[\S 1.2]{sam2011pieri}, the number of such paths in the computation of \(\widetilde\Phi_1\) is equal to the number of choices of rows in \(\lambda\) above row \([b_1](1)\).
Thus the complexity of the map \(\widetilde\Phi_1\) acting on a highest weight vector is
\[
      2^{h_{b_1} - 1} \cdot  \prod_{i = b_1 + 1}^N 2^{h_i} \leq {\textstyle\frac{1}{2}}\cdot \prod_{i = 1}^N 2^{h_i}
      = {\textstyle\frac{1}{2}}\cdot  2^{\sum_{i=1}^{N} h_i } = {\textstyle\frac{1}{2}}\cdot  2^{\,l(\lambda)}.
\]
In the worst-case when \(b_1=1\), the inequality is in fact an equality.
Furthermore, the paths that act on the first column of \(\lambda\) using Olver's algorithm can result in tableaux which are not semi-standard, and so must be straightened.
Hence the worst-case complexity of Olver's algorithm is \(\Omega(2^{\, l(\lambda)})\).
 
The map \(\Phi_1\) removing \(X\) restricts the choices of rows to those which describe an evacuation route, and hence the number of \(1\)-paths acting on the first column of \(\lambda\) in the computation of \(\Phi_1\) is equal to the number of choices of rows in \(\lambda\) above row \([b_1](1)\) made without skipping rows within blocks.
It also clear from the definition of \(1\)-paths that the image of a highest weight vector under a \(1\)-path is semi-standard.
Thus the complexity of the map \(\Phi_1\) acting on a highest weight vector is
\[
    h_{b_1} \cdot \prod_{i = b_1 + 1}^N \left( h_i + 1 \right) < \prod_{i = 1}^N \left( h_i + 1 \right)
  \leq (l(\lambda)+1)^N = \Theta(l(\lambda)^N).
\]
\end{proof}

\begin{remark*}
    Similar to the previous theorem, by restricting the maximum possible width of a block in \(\lambda\) we get that \(\Phi_1\) is an exponential speed up of \(\widetilde\Phi_1\) on the image of basis vectors (semi-standard tableaux) in \(\bS_{\lambda}\).
\end{remark*}

\subsection{} \label{The Image of a Highest Weight Vector and Computational Complexity subsection: Macaulay2 Examples}

This exponential to polynomial speed up can be seen in the computation time for computing Pieri maps in Macaulay2, for which the second author has implemented the description of \(\Phi_1\) given in Section \ref{section: Constructing the Pieri Inclusion for Removing One Box} within Sam's PieriMaps package \cite{sam2009computing}.

In Figure \ref{Comparing the Computational Complexity of the Pieri Inclusion Descriptions subsection: Macaulay2 Examples figure: 8,8,8} below we show the timed computations for computing the map
\[
    \bS_{(8,8,8)} \to \bS_{(1)} \otimes \bS_{(8,8,7)}.
\]
Using Olver's algorithm (as built in to PieriMaps), the process was interrupted after an hour with no output.
Using our algorithm implemented in PieriMaps, comuting this map takes only 0.07 seconds.

\begin{figure}[H]
    \centering
    \begin{subfigure}{.4\textwidth}
        \raisebox{-.9\height}{\includegraphics[width=\textwidth]{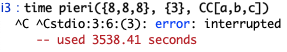}}
        \caption{Using Olver's algorithm.}
    \end{subfigure}
    \hspace{.5in}
    \begin{subfigure}{.4\textwidth}
        \raisebox{-.9\height}{\includegraphics[width=\textwidth]{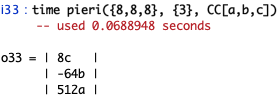}}
        \caption{Using our algorithm.}
    \end{subfigure}
    \caption{Computing the inclusion \(\bS_{(8,8,8)} \to \bS_{(1)} \otimes \bS_{(8,8,7)}\).}
    \label{Comparing the Computational Complexity of the Pieri Inclusion Descriptions subsection: Macaulay2 Examples figure: 8,8,8}
\end{figure}

We can also see this exponential speed up for small examples with more than one block. 
In the figure below we show the computation time for the Pieri inclusion
\[
    \bS_{(3,1,1,1,1,1,1,1,1,1)} \to \bS_{(1)} \otimes \bS_{(3,1,1,1,1,1,1,1,1)}
\]
using Olver's algorithm.
\begin{center}
    \raisebox{-.9\height}{\includegraphics[width=.85\textwidth]{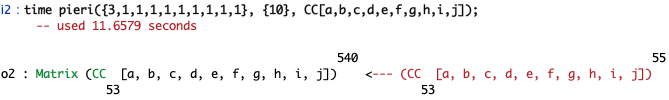}}\\
\end{center}
Using Olver's algorithm this computation takes over eleven seconds, while with the new algorithm this computation (shown below) takes less than two seconds.
\begin{center}
    \raisebox{-.9\height}{\includegraphics[width=.85\textwidth]{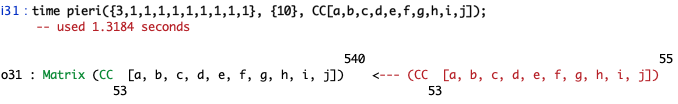}}
\end{center}


\newpage

\bibliographystyle{alpha}
\bibliography{refs}

\end{document}